\documentclass[11 pt]{article}  

\usepackage[utf8]{inputenc}
\usepackage{amsmath,indentfirst,qsymbols,graphicx,psfrag,amsfonts,stmaryrd,hyperref,color,enumerate,amsthm,ulem,dsfont,mathscinet}
\usepackage{url}
\usepackage{enumitem}
\usepackage[dvipsnames]{xcolor}
\usepackage{pgf,tikz}
\usetikzlibrary{arrows}
\usepackage{algorithm}
\usepackage{algorithmic}
\usepackage{listings}
\usepackage{eqparbox}
\usepackage{cleveref}
\usepackage{xparse}
\usepackage{subcaption}
\usepackage{todonotes}
\usepackage{comment}

\newcommand{\tm}{{\sf m}}

\newtheorem{lem}{Lemma}
\newtheorem{defi}[lem]{Definition}
\newtheorem{pro}[lem]{Proposition}
\newtheorem{theo}[lem]{Theorem}

\newtheorem{conj}{Conjecture}
\newtheorem{Ques}{Open question}
\newtheorem{rem}[lem]{Remark}

\renewcommand{\baselinestretch}{1.1}

\newcommand{\LET}{{\sf LastExitTree}}
\newcommand{\FET}{{\sf FirstEntranceTree}}

\def \bls{{\tiny $\blacksquare$ \normalsize }}

\def \degree{{\sf deg}}

\DeclareMathOperator{\argmin}{argmin}

\def \1{\textbf{1}}

\def \Z{\mathbb{Z}}
\def \bC{{\bf C}}
\def \bu{{\bf u}}
\def \bt{{\bf t}}
\def \bT{{\bf T}}

\def \SP{\textsf{SpanningTrees}}
\newcommand{\TTT}[1]{{\sf Subtrees}(#1)}
\newcommand{\GGG}[1]{{\sf Subgraphs}(#1)}

\newcommand{\TT}[2]{{\sf Subtrees}(#1,#2)}
\newcommand{\TTr}[3]{{\sf Subtrees}_{\,#1}^{\bullet}(#2,#3)}
\newcommand{\TTTr}[2]{{\sf Subtrees}_{\,#1}^{\bullet}(#2)}

\newcommand{\TTTrLd}[2]{{\sf Subtrees}_{\,#1}^{\bullet,L,\downarrow}(#2)}

\newcommand{\rhor}[1]{{\rho_{\,#1}^{\bullet}}}
\newcommand{\FF}[2]{{\sf Forests}_{\,#1}(#2)}

\def \ar#1{\overrightarrow{#1}}

\def \bpar#1{\left\{\begin{array}{#1} }
	\def \epar { \end{array}\right.}

\def \app#1#2#3#4#5{\begin{array}{rccl} #1:&#2&\longrightarrow&#3\\ &#4&\longmapsto&#5\end{array}}

\def \N{\mathbb{N}}
\def \R{\mathbb{R}}

\def \bD{{\bf D}}

\def \bK{{\bf K}}
\def \bL{{\bf L}}

\def \bP{{\bf P}}

\def \bT{{\bf T}}
\def \bI{{\bf I}}

\def \bfe{{\bf e}}

\def \bar{\overline}

\def \ba{\begin{align}}
	\def \ea{\end{align}}
\def \be{\begin{eqnarray*}}
	\def \ee{\end{eqnarray*}}
\def \ben{\begin{eqnarray}}
	\def \een{\end{eqnarray}}

\def \beq{\begin{equation}}
	\def \eq{\end{equation}}

\def \build#1#2#3{\mathrel{\mathop{\kern 0pt#1}\limits_{#2}^{#3}}}

\def \ba{{\bf a}}

\def \captionn#1{\begin{center}\begin{minipage}{16cm}\sf\caption{\small \textsf{#1}}\end{minipage}\end{center}}

\def \dd{\xrightarrow[]{(d)}}

\def \equi{\Leftrightarrow}
\def \eref#1{(\ref{#1})}

\def \P{{\mathbb{P}}}

\def \imp{\Rightarrow}

\def \l{\left}

\def \pp{\preccurlyeq}
\def \proba{\xrightarrow[n]{(proba.)}}

\def \r{\right}
\def \sous#1#2{\mathrel{\mathop{\kern 0pt#1}\limits_{#2}}}
\def \sur#1#2{\mathrel{\mathop{\kern 0pt#1}\limits^{#2}}}
\def \eqd{\sur{=}{(d)}}

\newqsymbol{`B}{\mathcal{B}}
\newqsymbol{`E}{\mathbb{E}}
\newqsymbol{`I}{\mathbb{I}}
\newqsymbol{`N}{\mathbb{N}}
\newqsymbol{`O}{\Omega}
\newqsymbol{`P}{\mathbb{P}}
\newqsymbol{`Q}{\mathbb{Q}}
\newqsymbol{`R}{\mathbb{R}}
\newqsymbol{`W}{\mathbb{W}}
\newqsymbol{`Z}{\mathbb{Z}}
\newqsymbol{`a}{\alpha}
\newqsymbol{`e}{\varepsilon}
\newqsymbol{`o}{\omega}
\newqsymbol{`t}{\tau}
\newqsymbol{`w}{{\cal W}}
\def\cro#1{\llbracket#1\rrbracket}

\def \Ex{{\sf Exchange}}
\def \Add{{\sf Add}}

\def \Rem{{\sf Remove}}

\def \pp{\mathsf{p}}
\def \rr{\mathsf{r}}

\def \qq{\mathsf{q}}

\def \uniform{{\sf Uniform}}

\newcommand{\compact}{ \topsep0pt   \itemsep=0pt   \partopsep=0pt   \parsep=0pt}

\newcounter{c}
\def \bir{\begin{itemize}\compact \setcounter{c}{0}}
	\def \itr{\addtocounter{c}{1}\item[($\roman{c}$)]} 
	\def \eir{\end{itemize}\vspace{-2em}~}

\newcounter{d}
\def \bia{\begin{itemize}\compact \setcounter{d}{0}}
	\def \eia{\end{itemize}\vspace{-2em}~}

\newcounter{b}
\def \bi{\begin{itemize}\compact \setcounter{b}{0}}
	
	\def \ei{\end{itemize}\vspace{-2em}~}

\tikzset{
	treenode/.style = {align=center, inner sep=0pt, text centered, font=\sffamily},
	arn_n/.style = {treenode, circle, white, draw=black,fill=black, text width=0.5em},
	arn_r/.style = {treenode, circle, white, draw=black, fill=red, text width=0.5em},
	arn_x/.style = {treenode, circle, white, draw=black, fill=white, text width=0.5em}
}

 \usepackage[english]{babel}
\usepackage{mdframed,caption}

\def \ra#1{\overleftarrow{#1}}
\def \Analysis{\underbar{\bf Analysis}}
\def \Drawback{\underbar{\bf Drawbacks}}

\def \TDLA{{\sf TDLA}}

\newcommand{\pass}{\\}
\begin{document}
\newcommand{\Torus}[1]{{\sf Torus}(#1)}
\newcommand{\cv}[1]{|#1|}
\def\ligne{\centerline{------------------------------}\\}
\newcounter{Mod}
\setcounter{Mod}{0}
\renewcommand{\theMod}{\Alph{Mod}}
\newcommand{\Modref}[1]{\ref{#1}}
\newcommand \NewModel[3]{\refstepcounter{Mod}
\begin{mdframed}[topline=false,rightline=false,bottomline=false,linewidth=2pt]
\noindent {\bf Definition of Model \theMod :}  {\sf #1}.\label{#2}~\\#3
\end{mdframed}
}

\newcounter{SMod}
\setcounter{SMod}{0}
\renewcommand{\theSMod}{\Alph{SMod}}
\newcommand{\SModref}[1]{(\ref{#1})}

\newcommand\NewSModel[3]{\refstepcounter{SMod}
\begin{mdframed}[topline=false,rightline=false,bottomline=false,linewidth=2pt]
\noindent {\bf Subtree of tree. Model \theSMod :}  {\sf #1}.\label{#2}~\\#3
\end{mdframed}}

\newcounter{Ker}
\setcounter{Ker}{0}
\renewcommand{\theKer}{\Alph{Ker}}
\newcommand{\Kerref}[1]{(\ref{#1})}

\newcommand \NewKernel[3]
{\refstepcounter{Ker}
\begin{mdframed}[topline=false,rightline=false,bottomline=false,linewidth=2pt]
\noindent{\bf Definition of the kernel $K^{(\theKer)}$:} {\sf #1}\label{#2}.\\#3
\end{mdframed}}

\setlength{\belowcaptionskip}{-1pt}
\setlength{\abovecaptionskip}{-20pt}
\setlength\intextsep{12pt}
 \begin{center}
 \huge\bf
Models of random subtrees of a graph \\
 {\large \bf Luis Fredes$^{\dagger}$ and Jean-Fran\c{c}ois Marckert$^{*}$}
 \rm \\
 
 \large{$^\dagger$Universit\'e Paris-Saclay.\\
 	$^{*}$CNRS, LaBRI, Universit\'e Bordeaux}

 \normalsize  
 \end{center}
 \begin{abstract} Consider a connected graph $G=(E,V)$ with $N=|V|$ vertices.
   The main purpose of this paper is to explore the question of uniform sampling of a subtree of $G$ with $n$ nodes, for some $n\leq N$ (the spanning tree case correspond to $n=N$, and is already deeply studied in the literature).
   We provide new asymptotically exact simulation methods using Markov chains for general connected graphs $G$, and any $n\leq N$. We highlight the case of the uniform subtree of $\Z^2$ with $n$ nodes, containing the origin $(0,0)$ for which Schramm asked several questions. We produce pictures, statistics, and some conjectures.\par
 A second aim of the paper is devoted to surveying other models of random subtrees of a graph, among them, DLA models, the first passage percolation, the uniform spanning tree and the minimum spanning tree. We also provide new models, some statistics, and some conjectures.  \end{abstract}

\subsection*{Acknowledgments} We acknowledge support from ERC 740943 \emph{GeoBrown}.

 \section{Introduction}

 \subsection{Random subtrees of a graph, the motivation}
\label{sec:rs}
The very origin of this work is the reading of Oded Schramm conference paper \cite{OS} of the International {C}ongress of {M}athematicians, Madrid, 2006, where he was one of the plenary speakers. In his Section 2.5, devoted to lattice trees, he raised two questions that motivated us to work in this domain. We take the liberty to copy it, here, verbatim:\par
\centerline{------------------------}
\begin{quotation} 
  {\bf Section 2.5. Lattice trees}. We now present an example of a discrete model where we suspect that perhaps conformal invariance might hold. However, we do not presently have a candidate for the scaling limit. \par
  Fix $n\in \N_+$, and consider the collection of all trees contained in the grid $G$ that contain the origin and have $n$ vertices. Select a tree $T$ from this measure, uniformly at random.\\
  {\bf Problem 2.8.} What is the growth rate of the expected diameter of such a tree? If we rescale the tree so that the expected (or median) diameter is 1, is there a limit for the law of the tree as $n\to+\infty$? What are its geometric and topological properties? Can the limit be determined? \medbreak 
  It would be good to be able to produce some pictures. However, we presently do not know how to sample from this measure.\medbreak
 \noindent {\bf Problem 2.9.} Produce an efficient algorithm which samples lattice trees approximately uniformly, or prove that such an algorithm does not exist \end {quotation}
\centerline{------------------------}
  
Excellent questions for which no real advances have been published during the last 14 years. Nevertheless, some images and statistics concerning lattice trees with a fixed size were already present in the literature in 2006: notably, Rensburg \& Madras \cite{RM} (1992) provided two ergodic Markov chains with uniform invariant measures (see also references therein), as well as Monte Carlo estimation of some parameters. Many more results, often coming from the mathematical physics literature (using sometimes heuristics) were available (Rensburg \& Rechnitzer \cite{RR}, Hsu et al. \cite{Hsu_2005} and Jensen \cite{Jensen} and numerous references therein); see Section \ref{sec:SUBG} for additional details. The question concerning the scaling limits of these objects seems stuck up to now. \par

Motivated by the understanding of the apparent obstruction to the construction of exact simulations for these lattice trees, we started to examine this question as a particular case of a more general question: is it possible to sample  a uniform subtree of a given size of a connected graph? This leads us to provide general Markov chains working on any finite connected graphs, and to produce new models of random trees embedded in a graph, as well as to survey already studied models of random subtrees of a graph.

Before discussing the content of the present paper, let us fix some notation. 

\paragraph{Convention and notation}

A {\it graph} $G$ is a pair $(V,E)$, where $V$ is the finite or countable set of vertices, and $E$ the multiset of edges. Each edge is a set of the form $\{a,b\}$ where $a$ and $b$ are different vertices. The word {\it multiset} means that each edge $\{a,b\}$ has a multiplicity, which is a positive integer.\par
As usual, a subgraph $G'=(V',E')$ of $G=(V,E)$ is a graph satisfying $V'\subset V$ and $E'\subset E$.
\\
We use the standard definition of paths, cycles, connectivity and connected components, induced subgraphs (see e.g.\cite{BondyMurty}).
A {\it tree} is a connected graph $T=(V_T,E_T)$ with no cycle: it satisfies $|E_T|=|V_T|-1$ (where $|S|$ stands for the cardinality of $S$).  A subtree of $G$ is a tree which is also a subgraph of $G$. A subtree $T$ is said to be spanning if $V_T=V$.

We will call $E(G)$ and $V(G)$ the edges and vertices of the undirected graph $G$, respectively and, we  write $\ar{E}(G)$ for the set of oriented edges associated with $E(G)$, which is the set containing for each edge $\{a,b\} \in E(G)$ two oriented copies: $(a,b)$ and $(b,a)$. For an oriented edge $\ar{e}=(e_1,e_2)\in \ar{E}(G)$, we denote simply by $e$ the unoriented version $\{e_1,e_2\}\in E(G)$.\par
A {\it rooted tree} is a pair $(T,r)$, where  $T$ is a tree and $r\in V_T$ is a distinguished vertex, called the root. It is often convenient to consider that the edges of a rooted tree $(t,r)$ are oriented toward the root $r$. A rooted tree can be thought as the genealogical tree of a population, with ancestor $r$.  The leaves are the nodes of $T$ having no incoming edges,  that is, that have no children. The set of leaves is denoted $\partial T$.
\par
 For a finite connected graph $G$ and some positive integer $n \leq |V|$, the notation  $\TT{G}{n}$  stands for the \textbf{set of subtrees} of $G$ with $n$ vertices. For a vertex $r\in V$, let $\TTr{r}{G}{n}$ be the subset of $\TT{G}{n}$ of trees which contains $r$ (they can be seen as being rooted at  $r$). We also define the set $\TTT{G} = \cup_n\TT{G}{n}$ of all subtrees of $G$, and $\TTTr{r}{G}= \cup_n\TTr{r}{G}{n}$ the set of those rooted at $r$.
 
For any finite set $S$, the uniform distribution on $S$ is denoted $\uniform(S)$.

If $G$ has several connected components, $\TT{G}{n}$ is the union of the sets of subtrees with size $n$ of each component. Hence, it can be assumed, and this is what we will do, that all the graphs $G$ considered in the paper are connected.

\begin{rem} Most of the models presented in the paper can be defined on multigraphs (in which multiple edges are allowed) as well as loops, up to small extra-cost. For the sake of clarity, we focus only on the case of simple graphs.
\end{rem}

\paragraph{Content of the paper}
Schramm's question  is a particular case of the following more general question:  Let $G=(V,E)$ be a finite connected graph.\medbreak

\noindent{\bf Question $[\star]$:}  Is there an efficient way to sample $\uniform(\TT{G}{n})$ or $\uniform(\TTr{r}{G}{n})$? \medbreak

Indeed, consider the discrete torus
\[\Torus{N}:= (\Z/N\Z)^2\] seen as a graph, with edges between pair of nodes of the type $(x,y)$ and $(x,y +1 \mod N)$, and between  $(x,y)$ and $(x+1 \mod N,y )$. Schramm's question about a way to sample $\uniform\l(\TTr{(0,0)}{\Z^2}{n}\r)$ is equivalent to finding a way to sample  $\uniform(\TTr{(0,0)}{{\sf Torus}(n)}{n})$, since the graphs $\Z^2$ and the finite graph $\Torus{n}$ coincide locally in a $n-1$ neighbourhood of their origin. \medbreak
  
  Trying to solve {\bf Question $[\star]$} on a general graph leads to investigate a lot of methods allowing one to sample random trees embedded in a graph, for example, Markov chain simulations, combinatorial methods, acceptance/rejection methods relying on simple to sample models, to design new models, to proceed to partial ``evaporation'' of uniform spanning-trees, etc.

  \medbreak

  The paper is organized as follows:
  
  \begin{itemize}[topsep=0pt, leftmargin=*,noitemsep] 
  	\item In Section \ref{sec:SC} we give a small list of simple graphs on which the simulation of uniform subtrees of a given size is easy, or well-known.
  	\item In Section \ref{sec:STC}, we recall some facts concerning the spanning-tree case $n=|V|$, for which efficient algorithms are known, with many recent developments. 
  	The problem to sample $\uniform(\TT{G}{n})$ for $n\leq |V|$ can be seen as a generalization of the uniform spanning-tree case so that, it can be useful to expose further this particular case. Moreover, a natural strategy discussed in this paper to sample according to  $\uniform(\TT{G}{n})$ consists in trying to extract a subtree of a uniform spanning-tree (instead of extracting this tree directly from the graph).
  	\item Section \ref{sec:enum} presents the combinatorics of the set of subtrees of a graph (mainly, Tutte polynomial like approaches), and then, applies these considerations to the uniform sampling in  $\TT{G}{n}$. In practice, these approaches can be applied  to small graphs only, due to the complexity cost of the methods involved.
  	\item Section \ref{sec:MURT} explores Markov chains taking their values in $\TT{G}{n}$. We propose three new different models of ergodic Markov chains whose invariant distribution is the uniform distribution (two of them, are new). We insist on the fact that these chains can be defined on any connected graph (not only on lattices).
  	\item In Section \ref{sec:SUBG} we focus on the grid case and on Oded Schramm questions: using one of the Markov chains of Section \ref{sec:MURT}, we made approximate simulations of uniform subtrees of the grid with $n$ vertices (for $n$ up to some thousands). 
We provide pictures, statistics and conjectures. Since the
trees are drawn in the plane, there are two main topologies to define scaling limits:\\
-- firstly, the Hausdorff metric topology, in which case, trees are seen as rescaled compact subsets of the plane, and\\
-- secondly,  the Gromov Hausdorff topology (in which case, the graph distance is rescaled).\\
      The empirical results we have, support the idea that a limiting distribution exists for rescaled trees in both cases, under suitable normalization. However, if the limit for the Gromov-Hausdorff topology is likely to be a random continuous tree, it seems that it is not the case for the Hausdorff distance: the limiting objects seem to have empty interior (no space filling phenomena), but portions of the drawn simulated objects form patterns close to macroscopic loops, so that it is tempting to conjecture that the limiting object is not a tree.  \it Intuitive and partial justifications \rm  about the fact that the stationary regime has been reached in our simulations are given in Section \ref{sec:simu},  Fig. \ref{fig:MC} and  \ref{fig:MC2} and videos at \cite{FM}.
  	\item In Section \ref{sec:SST}, we provide several Markov chains with state space $\TTT{G}$, the set of subtrees of a graph (without fixing the size of the subtrees). The stationary distribution is uniform conditional to the size of the sampled tree, and the random size has an explicit ``tunable'' distribution.
  	\item In Section \ref{sec:survey}, we survey many models -- different from the uniform distribution -- of random subtrees with $n$ nodes of a graph; for most of them we provide simulation pictures, description of the distribution and sometimes open questions.
  	\begin{itemize}[topsep=0pt, leftmargin=*,noitemsep]
  		\item[\bls] in Section \ref{sec:PT}, we introduce a new model of random subtree with $n$ nodes: the pioneer tree, which coincides with the tree formed by the first steps of Aldous--Broder algorithm, 
  		\item[\bls] in Section \ref{sec:qdegppzr} we present a principle showing that it is impossible to construct a uniform element of $\TT{G}{n}$ using ``the last steps'' of a simple Markov chain (and similar constructions),
  		\item[\bls] in Section \ref{Ver:notfull}, we present models on $\TT{G}{n}$ inspired by Wilson's cycle popping algorithm,
  		\item[\bls] in Section \ref{sec:biased}, we give a model of distinguished connected component in a size biased forest,
  		\item[\bls] in Section \ref{sec:dqskpd}, we discuss two ways to extract a random subtree with $n$ nodes of a UST,
  		\item[\bls]  in Section \ref{sec:dla}, we provide a model motivated by directed limited aggregation (DLA) and which is defined on any graph (and coincides with the original model on $\Z^2$);
                  \item[\bls] in Section \ref{sec:dqffsd}, a model motivated by the internal DLA,
  		\item[\bls] in Section \ref{sec:MST}, we propose several models of random trees defined on weighted graphs: among them, a model uses Prim's algorithm, one Kruskal's and another uses first passage percolation. 
  	\end{itemize}
  \item Finally, in Section \ref{sec:exact_samp}, we investigate the case of random subtrees of a tree. 
  \begin{itemize}[topsep=0pt, leftmargin=*,noitemsep]
  	\item[\bls] In Section \ref{sec:egu}, we give an exact sampling method of a uniform subtree of a tree (a coupling from the past method),
  	\item[\bls] in Section \ref{sec:LEV} we propose several models of extractions of a subtree of size $n$ relying on some models of leaf-evaporation.
  \end{itemize}
\end{itemize}

\subsection{Simple cases and other questions}
\label{sec:SC}
 Sampling uniformly in $\TT{G}{n}$ is easy for some families of graphs $G$. Among others :\\
$\star$ If $G=K_N$ the complete graph on $N$ vertices, then a uniform element $T$  in $\TTr{1}{K_N}{n}$ is a uniform labelled tree on $n$ vertices $\{x_1,\cdots,x_n\}$ where this set is itself a uniform subset with $n$ elements of $\{1,\cdots,N\}$ containing 1. Hence, up to a relabelling of the vertices, $T$ is a uniform Cayley tree of size $n$, also called a uniform labelled tree. These trees are among the simplest and most studied model of random trees in the literature (with a very long history going back to Cayley \cite{Cayley} in 1889, see also Moon \cite{Moon}): they are moreover easy to sample, for example, using Prüfer sequences \cite{prufer_neuer_1918}, Neville code \cite{neville_1953} (see additional elements and codes in  Caminati et al. \cite{CAMINITI200797}), a  bijection with parking sequences (see Chassaing \& Marckert \cite{MR1814521}, Chassaing \& Louchard \cite{MR1913079}), a relation with additive coalescence (see Aldous \& Pitman \cite{MR1675063}, \cite{MR1913079}, Marckert \& Wang \cite{MR3912100}), or as a uniform spanning tree of the complete graph, see Section \ref{sec:STC}, and Aldous \cite{Al90}. Their asymptotic behaviour when $n\to+\infty$ is well known; they converge in distribution, after rescaling of the graph distance by $\sqrt{n}$ to the so called {\sl continuum random tree} (also called Brownian tree or Aldous' continuum random tree):  see  Aldous \cite{aldous1997}, Pitman \cite{Pitman} (and additional  combinatorial properties),  Marckert \& Mokkadem \cite{Mar-Mokka} and Duquesne \& Le Gall \cite{Duq-Leg} for additional information.\\
$\star$ If $G$ is the cycle $\Z/N\Z$, a path (the graph with vertices $1$ to $N$, with edges between $i$ and $i+1$), the set of subtrees of size $n$ coincide with the set of length $n-1$ intervals, and their simulations are trivial. To some extent, the same can be said for regular graph such as $\{0,1\}\times\{1,\cdots,N\}$ or  $\{0,1,\dots,k\}\times\{1,\cdots,N\}$, for $k$ fixed (with complexity growing in $k$: transfer matrices  allow to count the number of subtrees with a given first column, see for example da Silva et al.  \cite{da_Silva_2021} and references therein, and this allows one to successively sample the uniform random subtree column by column) or any family of graphs on which some simple combinatorial decompositions can be performed easily, then one may find some ad hoc methods to sample $\uniform(\TT{G}{n})$.\\
$\star$ If $(T(m),r)$ is the  infinite regular $m$-ary tree with root $r$ (the only node with degree $m$), then sampling uniformly in $\TTr{r}{T(m)}{n}$ is also a simple task, since each element of $\TTr{r}{T(m)}{n}$, can be seen as the set of internal nodes of a (non embedded planar) uniform $m$-ary tree with $n$ internal nodes (that is $1+nm$ nodes): sampling such a uniform tree is an easy task with several known methods, linear in the tree size (or with cost $n\log n$ depending on the cost model \footnote{The cost in terms of elementary operations performed on the data basis; it depends on the representation of the data, and need to be defined before talking of the cost of an algorithm}), since it is a model of ``simple trees'', which can also be seen as a Galton-Watson tree conditioned on the size (see e.g. Devroye \cite{LucD} for an overview of random generation of Galton-Watson trees conditioned by the size and Marckert \cite{marckert2021growing}, for a new ad hoc method for $m$ ary trees; in the binary tree case, additional methods are available, among other Rémy algorithm \cite{MR803997} which is an efficient method, with many properties, see e.g. Marchal \cite{MR2042386} and Evans et al. \cite{MR3601650}).\par In \cite{MR2060629}, Luczak \& Winkler provide a way to grow a sequence of trees $(t_n)$, such that $(t_n)$ is increasing for the inclusion partial order, and such that moreover, for each $n$, $t_n$ is uniform in $\TTr{r}{T(m)}{n}$.  \\
$\star$ Another line of research is the study of the uniform random subtree of some families of random graphs. It turns out that in some cases, the random generation is simple:\par
-- as shown by Fredes and Sepulveda \cite{FS20}: the random generation of a uniform subtree $t$ of size $m$ in a random rooted quadrangulation with $n$ faces can be done for any $(n,m)$ with $m<n+1$ in a reasonable time. This comes from the existence of a one-to-one correspondence between, on the one hand, quadrangulations marked by a distinguished subtree, and on the other hand, a pair formed by a quadrangulations with a simple boundary together with a planar tree (this bijection also can be extended to other models of planar maps, for example, with different restrictions on faces or vertices degrees). \par
-- In the case Erdös-Rényi graphs $G(n,p)$, it is known that the giant connected components have a phase transition for $p$ being approximatively $1/n$. Aldous \cite{aldous1997} established that for $p=p_\lambda(n)=1/n+\lambda/n^{4/3}$ (with $\lambda$ fixed), the size of the largest connected component, divided by $n^{2/3}$ converges in distribution; this is true also for the sizes of the $k$ largest components, and true also as a process indexed by $\lambda$ (see in \cite{BroutinMarckert}). In fact, as proved by Addario-Berry et al. \cite{MR2892951}, these connected components, seen as random graphs, have a scaling limit (when the graph distance is normalised by $n^{1/3}$). Moreover, the excesses\footnote{the excess of a connected graph is the minimal number of edges needed to be removed to turn the graph into a tree} of these components are well understood:  with a positive probability (bounded from below, when $n\to +\infty$, for a fixed $\lambda$), these connected components are trees. These results are  somehow the starting point to the paper Addario-Berry et al. in \cite{ABBG}, in which is established that the minimum spanning tree of the complete graph $K_n$ (with any reasonable models of random weights), possesses a scaling limit, after normalisation by $n^{1/3}$.

\section{The spanning-tree case}
\label{sec:STC}

Given a finite connected graph $G=(V,E)$, there are  several kinds of approaches to sample a UST of $G$, with many recent developments (a recent survey can be found in Schild \cite{A_Schild}).

\textbf{-- Random walk approaches:} Two famous algorithms, recalled in the two following sections are Aldous--Broder algorithm (Broder \cite{Bro89}, Aldous\cite{Al90}, see also Hu et al. \cite{HLT} and Fredes \& Marckert \cite{LFJFM}  for a variant) and Wilson's algorithm \cite{Wil96} (see also Lyons \& Peres \cite[Section 4]{lyons_peres_2017}, Járai \cite{Jarai09}).
Their expected running time for undirected graphs are $O(\tau_c)$ and $O(\tau)$ respectively, where $\tau_c$ and $\tau$ are the mean cover time\footnote{The mean cover time $\tau_c$ is here the maximum expected time to visit all the vertices of the graph, where the maximum is taken over all starting points.} and mean hitting time\footnote{The mean hitting time  is defined as $\tau=\sum_{i,j}\pi(i)\pi(j)E_{i,j}$, where $\pi$ is the invariant distribution and $E_{i,j}$ is the mean time starting from $i$ to reach $j$.} of the simple random walk in $G$, respectively. Wilson's algorithm is the fastest of the two since the mean hitting time is always smaller than the cover time. \par
Wilson and Aldous--Broder algorithms permit also the generation of trees with a probability proportional to the product of the edge weights, as stated below, in \Cref{theo:AB} (in the positive weighted edge models, where the edges of the initial graphs possess some positive weights, which can be seen as conductances).

It is tempting to try to tune these algorithms to sample uniformly in $\TT{n}{G}$, that is with a given size $n$ for some $n <|V|$. Some of these modifications  will be discussed in the paper, but none of them allows one to sample uniformly in $\TT{n}{G}$ (when $n\leq |V|$);  a kind of meta argument will be developed in Theorem \ref{theo:nolocal} to explain why it is in general not possible to obtain the uniform distribution using random walks when $n\ll |V|$.\par

\noindent -- \textbf{Laplacian methods.}
If the graph is small, Tutte's formula (see e.g. Tutte \cite{Tutte54},  Bernardi \cite{OB}, Welsh \cite{Welsh} and  Section \ref{sec:enum} for additional details) or more efficiently, Kirchhoff matrix tree theorem (Kirchhoff \cite{Kirchhoff}, Chaiken \& Kleitman \cite{MR480115}, Zeilberger \cite{DZ}) can be used to design some generation algorithms. The Laplacian matrix of a graph $G=(V=\{1,\cdots,n\},E)$ is the matrix
\[{\sf Lap}(G):=\begin{bmatrix} {\sf deg}_G(i)\1_{i=j} - A_{i,j} \end{bmatrix}_{1\leq i,j\leq n}\] where $A_{i,j}$ is the number of edges between $i$ and $j$ (a variant using weighted edges can be used instead, in which case $A_{i,j}$ is the weight of the edge $(i,j)$, and the diagonal term ${\sf deg}_G(i)\1_{i,j}$ has to be replaced by $\sum_{j} A_{i,j}$). The famous matrix tree theorem asserts that the number of (unrooted) spanning trees of $G$ is
\begin{align}\label{MT}|\{ {\sf Spanning\ trees\ of }(G)\}|=|\det({\sf Lap}(G)^{\star})|,
\end{align} 
where ${\sf Lap}(G)^{\star}$ is obtained from ${\sf Lap}(G)$ by the suppression of a row and a column  This formula also gives the cardinality of the set of spanning trees rooted at some fixed vertex $r$.\par
In the weighted case, removing the row and column $r$ in the Laplacian matrix gives the sum over the weighted rooted trees at $r$, more formally,
\[|\det({\sf Lap}(G)^{(r)})|=\sum_{(t,r): t \textsf{ spanning }} W(t,r)\]
where $W(t,r)=\prod_{(u,v)\in {\sf Edges}(t,r)} w_{(u,v)}$ with $w_{(a,b)}$ the weight of the oriented edge $(a,b)$ and where in the rooted tree $(t,r)$, each edge $(u,v)$ is oriented toward the root $r$.

This theorem can be used to determine the probability of presence of a given edge of $G$ in a UST, which can be taken into account recursively for the complete random generation (see \Cref{sec:usucf}).\\
 This fact is used and discussed in Colbourn et al. \cite{Colbourn}, who designed an algorithm with time cost $O(|V|^3)$ to sample a uniform spanning tree of $G$ (improving on Guénoche \cite{Guenoche} and Kulkarni \cite{Kulkarni}). In  \cite{Colbourn2}, Colbourn et al. provide a method running according to the cost of the best-known matrix multiplication (which is $O(|V|^{\omega})$ for $\omega< 2.373$).

\noindent--\textbf{Hybrid methods.} In the very last years, the previous results have been improved by mixing random walk methods with computation methods relying on Laplacians, connections with electrical networks, with the aim to be able to provide some shortcuts to the random walks.
Kelner \&   M\polhk{a}dry \cite{KM} provide an algorithm with time cost $\tilde{O}(|E| \sqrt{|V|}\log (1/\delta))$ (the $\tilde{O}$, meaning ``up to polylog factors'') to sample a tree within a multiplicative $(1+\delta)$ of a uniform spanning tree, result improved by M\polhk{a}dry et al. \cite{Madry} (time cost $\tilde{O}(|E|^{4/3})$, by Durfee et al. \cite{Durfee} (time cost $\tilde{O}(|V|^{4/3}|E|^{1/2}+|V|^2)$, for the more general case of edge-weighted trees). Finally, very recently, Schild \cite{Schild} (long version in \cite{A_Schild}) provides an algorithm with time cost $|E|^{1+o(1)}\beta^{o(1)}$ (in the general weighted graph case, with max-to-min ratio $\beta$).

We refer the reader to Schild \cite{A_Schild}, Durfee et al. \cite{Durfee} for the complete history on these lines of research.
\begin{figure}[h!]
  \centerline{\includegraphics[width = 9cm]{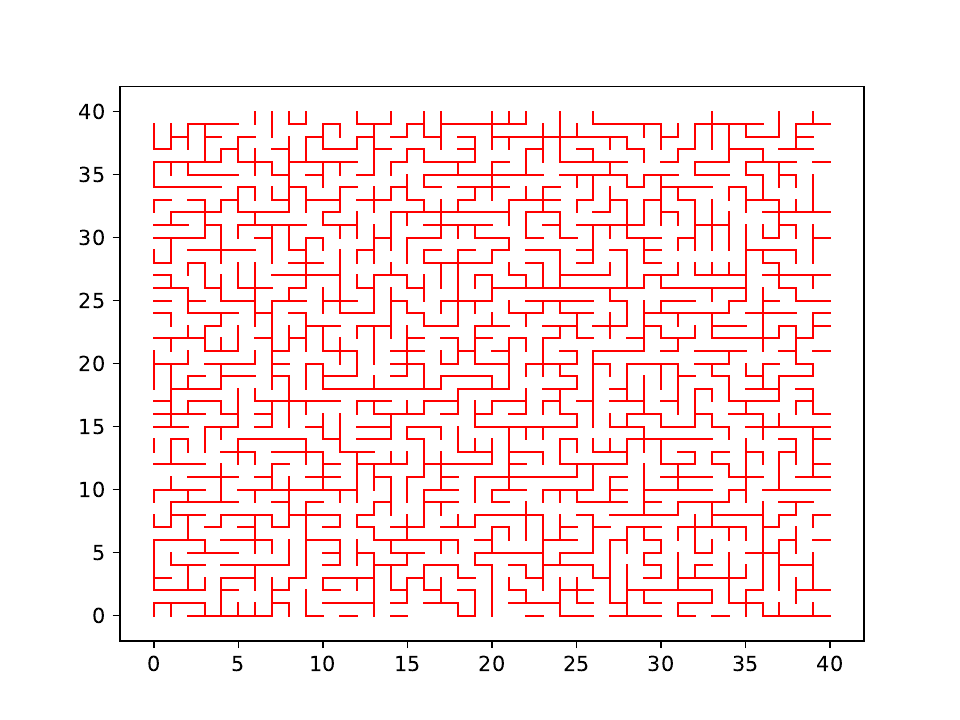}}
  \captionn{ \label{FIG:simu}Simulation of a UST of the graph $(\Z/40\Z)^2$ using Wilson's algorithm (on this picture, identify the right and left sides, and of the top and bottom sides to get the actual spanning-tree). }
 \end{figure}
 
\subsection{Aldous--Broder algorithm}
\label{sec:ABA}

\paragraph{Reversible transition matrices} A Markov chain with transition matrix  $M$ on $E\times E$ is said to be reversible with respect to a distribution $\rho$, if it satisfies the \textbf{detailed balance equations}:
\begin{align}\label{reveqs}
	\rho_iM_{i,j}=\rho_jM_{j,i} \textrm{ for any } i,j\in E.
\end{align}

In this case $\rho$ is invariant for this Markov chain.

We say that a transition matrix $M=(M_{a,b},a,b \in V)$ is \textit{positive}  on a connected graph $G=(V,E)$, if $\{a,b\}\in E \equi M_{a,b}>0$. 
Denote by $\rho=(\rho_v,v\in V)$ the unique stationary distribution of this transition matrix.
Consider $W=(W_k, k\geq 0)$ a Markov chain with transition matrix $M$. Set
\be
\tau_k(W) = \inf\{j, |\{W_0,\cdots,W_j\}|=k\},~~~1\leq k \leq |V|\ee
the first time $k$ different points have been visited: hence $\tau_1(W)=0$, and the cover time is $\tau_{|V|}(W)$ (we will write $\tau_k$ instead of $\tau_k(W)$ when it is clear from the context).

\begin{defi}\label{defi:dqsd} Denote by $\FET(W_0,\cdots,W_{\tau_{|V|}(W)})$ the rooted spanning-tree\footnote{The first entrance tree is associated to a path, random or not.}  with root $W_0$ and whose $|V|-1$ edges are given by the oriented edge $(W_{\tau_k},W_{-1+\tau_k})$ for $2\leq k \leq \tau_{|V|}$.
\end{defi}
\begin{theo}\label{theo:AB}[Aldous \cite{Al90} and Broder \cite{Bro89}] If $M$ is positive and reversible with invariant distribution $\rho$, then
\ben\label{eq:qsdq17}
\P\l[\FET\l(W_0,\cdots,W_{\tau_{|V|}}\r)=(t,r)~|~W_0=r\r]= {\sf Const.}\left(\prod_{e\in E(t,r)} M_e\right)/\rho(r). \een
\end{theo}
Here, and elsewhere, for any rooted spanning-tree $(t,r)$, the edges $E(t,r)$ of $(t,r)$ are oriented toward the root $r$ (so that if $e=(e_1,e_2)$, $e_2$ is the parent of $e_1$, and $M_e:=M_{e_1,e_2}$).\\
Proofs can be found in \cite{Bro89, Al90,lyons_peres_2017,Jarai09}. The original proof relies on the so-called ``tree Markov chain'': this is a Markov chain whose state space is the set ${\sf Spanning\ trees}(G)$, and whose evolution is defined using a step of a random walk with transition matrix $M$ on $V$.\\
As a consequence, if $M$ is the transition matrix corresponding to the simple random walk on $G$, $M_{a,b}=\1_{\{a,b\}\in E}/\degree(a)$, the invariant distribution $\rho_v$ is proportional to $\degree_G(v)$, so that \eref{eq:qsdq17} is independent of $t$ and $\FET(W_0,\cdots,W_{\tau_{|V|}})$ is a UST, rooted at $W_0$ (for the non-rooted case, any choice of distribution of $W_0$ is fine: just project on non-rooted trees).\\

\noindent In fact, the reversibility condition in Theorem \ref{theo:AB} can be dropped, but the conclusion of the theorem has to be adapted (see Hu et al. \cite{HLT}, and Fredes \& Marckert \cite{LFJFM}). 
For a positive transition matrix $M$ on $G$, there exists a unique invariant distribution $\rho$.  Define $\ra{M}$ by
 \ben\label{eq:dqsd22}
 \ra{M}_{x,y}:= \rho_{y} M_{y,x} / \rho_x, \textrm{ for all }(x,y)\in V^2,
 \een
so that $\ra{M}$ is simply the transition matrix of the time-reversal of a Markov chain with transition matrix $M$ under its invariant distribution. 
\begin{theo}[\cite{HLT}, \cite{LFJFM}]\label{the:kgyqsd} If $M$ is positive on $G$, and $W$ is a Markov chain with transition matrix $M$ and invariant distribution $\rho$, then for any rooted spanning-tree $(t,r)$
  \ben\label{eq:qsdq33}
  \P\l[\FET\l(W_0,\cdots,W_{\tau_{|V|}}\r)=(t,r)~|~W_0=r\r]= {\sf Const.}\left(\prod_{e\in E(t,r)} \ra{M_e}\right) / \rho(r).\een
\end{theo}
This theorem implies Aldous--Broder result since in the reversible case, $\ra{M}=M$.

\subsection{Wilson's algorithm}

We refer to Schramm \cite{SC}, Lawler \cite{Lawler1999}, Marchal \cite{Marchal}  and Viennot \cite[Prop.6.3]{VX} for more information concerning loop erased random walks (and to Schramm \cite{SC}, Lawler \cite{Law2005}, Lawler et al. \cite{LSW04} for conformal invariant scaling limit considerations, in the lattice case, in which deep links between scaling limits of loop erased random walks and scaling limits of uniform spanning trees are discussed). 
Let $M$ be a positive transition matrix on a connected graph $G=(V,E)$ and $r\in V$ a distinguished node.
For any starting point $v\in V$ and non-empty subset $S$ of $V$, we denote by ${\sf LERW}_M[v,S]$ the distribution of a $M$-loop erased random walk starting at $v$ and killed at its hitting time of $S$ (meaning that before erasure, the random walk is a Markov chain with transition matrix $M$). 

Wilson's algorithm can be stated as follows: Consider an ordering of the vertices $(v_1=r,v_2,\dots,v_{|V|})$ of $V$, and set ${\bf T}_1$ as the initial tree reduced to the point $v_1=r$.
For any $2\leq i\leq N$, consider a loop erased random walk $L_i$ with distribution ${\sf LERW}_M[v_i,{\bf T}_{i-1}]$, starting at $v_i$ and stopped at the vertex set of the current tree ${\bf T}_{i-1}$. 
The tree ${\bf T}_{i}$ is the tree having as set of edges those of ${\bf T}_{i-1}$ union the set of steps of $L_i$ (meaning that if $L_i=(a_0,\cdots,a_m)$, the new edges are the $(a_j,a_{j+1})$). If $v_i$ is already in  ${\bf T}_{i-1}$, there is no new edges. Denote by ${\sf WilsonTree}_r$ the final tree ${\bf T}_{|V|}$. We have
\begin{pro}[\cite{Wil96}] For any positive transition matrix $M$, for any rooted spanning-tree $(t,x)$ of $G$,
  \ben\label{eq:Wilson}
  \P({\sf WilsonTree}_r=(t,x))= {\sf Const}.\1_{x=r} \prod_{e \in E(t,r)} M_e.
  \een
\end{pro}
For proofs, see Wilson \cite{Wil96}, Propp \& Wilson \cite{PW98}, Járai \cite{Jarai09}, or Lawler \cite{Lawler1999}.

\par
When $M_{a,b}=\frac{1}{\degree(a)}$, then $\P({\sf WilsonTree}=(t,x))={\sf Const}.\1_{x=r}/  \prod_{v \neq r} \degree(v)$, which again, does not depend on the tree $t$, so that again, ${\sf WilsonTree}_r$ is a UST rooted at $r$.

This construction admits a companion description called cycle popping (detailed in \cite{Wil96}, \cite{PW98}, \cite{Jarai09}) which is more suitable for generalizations (see Section \ref{Ver:notfull}). \\
{\bf Cycle popping algorithm.}
Consider for each vertex $v$, different from $r$, a random outgoing edge $\ar{\bf e}_v$, independent of the others, such that $\P(\ar{\bf e}_v= (v,w)) =M_{v,w}$, and call such orientation ${\bf O}=(\ar{\bf e}_v, v\in V\setminus\{r\})$.  
It is simple to check that if the set of oriented edges in ${\bf O}$  forms a tree, it will be a spanning-tree rooted at $r$, and the probability that this spanning-tree equals $(t,r)$ is $\prod_{e \in E(t,r)} M_e$.
When the edges in ${\bf O}$ does not form a spanning-tree, the connected component ${\bf t}(r)$ of $r$ is a tree and all other connected components contain a (unique) oriented cycle. 
The cycle popping algorithm \cite[Sec. 6]{PW98} consists in choosing a cycle and re-sampling the outgoing edges of all the vertices it contains; this operation is repeated until the resulting orientation does not contain any cycle, so that it corresponds to a tree ${\bf T}^{\star}$.  Wilson \cite{Wil96} proved that ${\bf T}^{\star}$ has also the distribution given in \eref{eq:Wilson}; better than that, he explains how the construction using the LERW is just a way to view/order the cycle poppings.

As suggested by Theorem \ref{the:kgyqsd}, if a coupling between Aldous--Broder and Wilson algorithms could be found, then probably Wilson's algorithm should be run using the transition matrix $\ra{M}$ instead of $M$.
\begin{Ques} It is possible to couple Wilson and Aldous--Broder constructions so that they depend on the same trajectories (and give the same results)?
\end{Ques}

\paragraph{Mixing the UST question with the configuration model?}
The following question is open to our knowledge and seems particularly interesting: it is the question of the sampling of a UST with prescribed degrees.
\begin{Ques}\label{Qu:cond} Given a connected graph $G=(V,E)$ and some positive integers $(d_u,u \in V)$ associated with the vertices of $V$, find an algorithm that produces a UST ${\bf t}$ of $G$ conditioned on the event $\{\degree_{\bf t}(u)= d_u,  u\in V\}$ when there exists such a spanning-tree.
\end{Ques}
The existence of a spanning-tree satisfying  $\{\degree_{\bf t}(u)= d_u,  u\in V\}$ can be decided by exhaustive approach or using the matrix tree theorem as shown in \eqref{MT} (set $u_iu_j$ as the weight of the edge $(i,j)$, where the $u_j$ are formal monomials; then extract the coefficient of $\prod_{j} u_j^{d_{j}}$ in the determinant of the Laplacian matrix of the graph, with the first row and column, removed: this coefficient gives the number of such spanning-trees). The problem of the uniform generation of a Hamiltonian path (a path that visits each node exactly once) is equivalent to that of a spanning-tree whose nodes have all degree 2, except for the extremal nodes that have  degree 1. There are no efficient algorithm for this task, since even the decision problem of existence of a Hamiltonian path  is NP-complete (Karp \cite{Karp}).
The previous discussion implies that deciding the existence of a spanning tree with some prescribed degree sequence is NP-complete, which implies, a priori, that the answer to Question \ref{Qu:cond} is difficult without additional hypothesis. Indeed, take the example of the complete graph $K_n$; a uniform spanning tree rooted at 1, in which one sets that exactly $n_k$ nodes must have $k$ children for a fixed sequence $(n_i, 0\leq i\leq n-1)$, can be simulated by taking a uniform permutation of the sequence $0^{n_0}1^{n_1}...n^{n_{n-1}}$ (made of the concatenation of $n_0$ zeroes, $n_1$ ones, $n_2$ twos, ...);  this gives a sequence $(X_0,\cdots,X_{n-1})$; the rotation principle (Otter \cite{MR30716}), then asserts that there is a single $a\in \Z/n\Z$ (easy to compute, \cite[Section 2.5]{marckert2021growing}) such that $X^{(a)}:=(X_a, X_{a+1 \mod n}, \cdots, X_{n-1+a\mod n})$ is the sequence of node degrees of a planar tree, traversed in the lexicographical order. Now, put label 1 at the root, and for $i$ going from 2 to $n$, put label $\sigma_{i-1}$ to the $i$th node sorted according to the lexicographical order, where $\sigma$ is a uniform random permutation $\{2,\cdots,n\}$.
It is a simple exercice to show that this method provides the uniform distribution on the set of labeled trees with root 1, respecting the degree sequence given.
  Even the scaling limit of this model is known, under some hypothesis on the limiting proportion $p_i$ of nodes of degree $i$, see Broutin and Marckert \cite{MR3188597}.

\paragraph{Uniform spanning tree of infinite lattices}
Some results exist concerning the asymptotic behaviour of UST of the grid, either locally, or after rescaling. Since the paper is rather devoted to more general random subtrees, we just give here few pointers.  
Taking a UST on a graph as $[-n,n]^d$ (with edges between points with integer coordinates at Euclidean distance 1), and letting $n\to+\infty$, Pemantle \cite{MR1127715}  showed that the limit is a tree on $\mathbb{Z}^d$ for $d\leq 4$ (and a forest for $d>5$ with infinitely many components). See Benjamini et al. \cite{benjamini2001uniform} for extension to general graphs.  \par
Another line of research concerns the asymptotics of uniform spanning tree after rescaling, and the conformal invariance of the limit (see Schramm \cite{SC}, Lawler et al. \cite{MR2044671}).\par
The weak limit of the UST on the torus $(\Z/n\Z)^d$ is shown to be, after an appropriate rescaling,  the continuum random tree (which is the natural limit of the UST of the complete graph after a scaling by $\sqrt{n}$). It has been proved by Peres \& Revelle \cite{arxiv.math/0410430} (for $d\geq 5$)  and by Schweinsberg \cite{MR2496437} for $d=4$.

\section{The combinatorial approach to sample $\uniform(\TT{G}{n})$}

\label{sec:enum}

When $G$ is finite, $\uniform(\TT{G}{n})$ can be sampled if one knows a way to sample uniformly in $\TTr{r}{G}{n}$ for all $r\in V$ (that is when a root is fixed) and if $|\TTr{r}{G}{n}|$ is known for each $r$ (or if they are known to be equal for some reasons, for example, if a group acts transitively on the graph).
Indeed, since the trees of $\TT{G}{n}$ have the same number of nodes, it suffices to first pick a random node ${\bf r}$ according to the unique probability distribution proportional to  $(|\TTr{r}{G}{n}|, r \in V)$, and then, conditionally on ${\bf r}=r$, to pick a tree uniformly in  $\TTr{r}{G}{n}$. 

It turns out that the sequence $(|\TT{G}{n}|,n\geq 1)$ can be computed using a decomposition similar to that used when deriving Tutte's formula (Tutte \cite{Tutte54}, Bernardi \cite{OB}). The first part of what follows and which concerns a Tutte polynomial for subtrees of a graph, is present \textit{mutatis mutandis} in \cite[Prop. 4.4.]{Chin18}, for unrooted subtrees.

Apart from their own interest,  these algebraic considerations  bring some additional insight, and possibly, potential methods to sample $\uniform(\TT{G}{n})$: in general, the cost of the computation of $(|\TT{G}{n}|,n\geq 1)$ is significant, and can be done only on small graphs (or  particular ones); more elements on the complexity of these costs are discussed below.

Tutte recursion produces loops, multiple edges, and may disconnect the graph (if we allow the deletion of bridges, which is the case here). In this section, we then consider multigraphs $G=(V,E)$, possibly disconnected, having possibly some loops (edges of the form  $\{a,b\}=\{a\}$). Of course,  the number of subtrees of a graph having some loops is unchanged by their removal. Since we deal with rooted subtrees of $G$, any part of the graph disconnected from the root of the tree can be ignored.

Consider a multigraph $G=(V,E)$, and $e$ an edge (possibly not in $E$). Recall the two classical operations, contraction and suppression of edges:\medbreak
\noindent$\bullet$ The graph $G\setminus e$ obtained from  {\bf the suppression of $e$}, is the multigraph $G'=(V',E')$ coinciding with $G=(V,E)$ except that a copy of the edge $e$ is suppressed from $E$ if any, and $G'=G$ otherwise,\\
$\bullet$ The graph $G.e$ obtained from {\bf the contraction of $e$}, is the multigraph $G'=(V',E')$ defined as follows: if $e$ is a loop, then $V'=V$ and $E'$ is obtained from $E$ by removing 1 to the multiplicity of the edge $e$; and if $e$ is not a loop, say $e=\{a,b\}$, we define $V'=V\setminus \{b\}$, and $E'$ from $E$, by replacing every occurrence of $b$ in an edge $e''$ of $E$ by $a$.

Consider the following polynomial 
\be
{\bf T}_r(G)=\sum_{t \in \TTTr{r}{G}} x^{|E(t)|}
\ee
which is the generating function of the sequence $(\TTr{r}{G}{n},n\geq 1)$, with size function, the number of edges.

If the connected component of $r$ in $G$ has a single vertex (for example, if $G=(\{r\}, \{r,r\}^k)$ for some $k\geq 0$), then ${\bf T}_r(G)=1$.
Notice that if an edge $e=\{a,b\}$ is not included in the connected component of $r$, or, if $e$ is a loop, then
\be
{\bf T}_r(G)= {\bf T}_r(G\setminus e)={\bf T}_r(G.e).
\ee
\begin{pro}\label{pro2} Let $G=(V,E)$ be a multigraph and $r\in V$. For any edge $e\in E$ adjacent to $r$,
  \ben\label{eq:Tcondrec}
  {\bf T}_r(G)=x{\bf T}_r(G. e)+ {\bf T}_r(G \setminus e).
  \een
\end{pro}
\begin{proof}
Any tree counted in the left-hand side either contains $e$ or not. 
\end{proof}

\begin{rem}\label{trivialG}
Removing or contracting edges adjacent to $r$ reduces the number of edges, so that \eref{eq:Tcondrec} indeed defines ${\bf T}_r(G)$ (using eventually ${\bf T}_r(G')=1$ when $V(G')=\{r\}$).
\end{rem}

This formula is very similar to Tutte's formula, which has been a key tool for the development of algebraic graph theory. However, the computation of ${\bf T}_r(G)$ using \eref{eq:Tcondrec}
is at least linear in the number of subtrees, since each expansion in \eref{eq:Tcondrec} can be seen as describing a subtree edge per edge: a contracted edge is in the subtree, while a deleted one, is not. 

Formula \eref{eq:Tcondrec} can be used to compute the first values of ${\bf T}_r(G)$ for $G=\TTTr{r}{\Torus{N}}$ for $N$ from 1 to 4:
\be
&&1\\
&&32\,{x}^{3}+12\,{x}^{2}+4\,{x}+1\\
&&11664\,{x}^{8}+9408\,{x}^{7}+4074\,{x}^{6}+1308\,{x}^{5}+345\,{x}^{4}+
80\,{x}^{3}+18\,{x}^{2}+4\,x+1
\\
&&42467328\,{x}^{15}+56597760\,{x}^{14}+39892832\,{x}^{13}+19618560\,{x}
^{12}+7588872\,{x}^{11}+2461360\,{x}^{10}
\\
&&+698700\,{x}^{9}+178848\,{x}^
{8}+42496\,{x}^{7}+9534\,{x}^{6}+2052\,{x}^{5}+425\,{x}^{4}+88\,{x}^{3
}+18\,{x}^{2}+4\,x+1. \ee
After that, the computer costs become an obstacle.
\begin{rem} Kirchhoff matrix tree theorem \cite{Kirchhoff} can also be used to enumerate the number of subtrees of size $n$ of a given graph $G$, by considering one by one all the  induced subgraphs with $n$ vertices of $G$, and by summing their number of spanning-trees. It gives $\binom{|V|}{n}$ different graphs, for which a determinant of  size $(n-1)\times(n-1)$ has to be computed. This cannot be used in practice when $\binom{|V|}{n}$ is large.
\end{rem}
Counting the number of subtrees of a graph is a \#P-complete problem as proved by Jerrum  \cite{JERRUM1994111} (see also Jaeger et al. \cite{jaeger_vertigan_welsh_1990}), so that, in principle, these complete enumeration methods can be done only on small graphs.

See Chin et al. \cite[Prop. 3.1.]{Chin18} for some explicit polynomials in different classes of graphs.

This can be generalized to forests. A graph $F=(V_F,E_F)$ is said to be a forest if its connected components are trees. Given,  $r_1,\cdots,r_k$ distinct elements of $V$ (for $k\geq 1$), we denote by  $\FF{r_1,\cdots,r_k}{G}$ the set of forests composed of $k$ non intersecting trees, where for each $i\in\cro{1,k}$, $r_i \in t_i$.

Define the multivariate generating function of forests in $\FF{u_1,\cdots,u_k}{G}$
\ben \label{eq:Tcondrec2}
{\bf F}_{u\cro{1,k}}(G)= \sum_{(t_1,\cdots,t_k)\in \FF{u\cro{1,k}}{G}} \prod_{j=1}^k x_j^{|E(t_j)|},
\een
 counted according to the size of its connected components.
Following the same idea in \Cref{pro2} we obtain the following proposition.
\begin{pro} For any edge $e$ with only one endpoint $u_j$ in $u\cro{1,k}$, we have
\ben\label{eq:Fcondrec}
{\bf F}_{u\cro{1,k}}(G)={\bf F}_{u\cro{1,k}}(G \setminus e)+x_j{\bf F}_{u\cro{1,k}}(G. e).
\een
\end{pro}
Very related to these considerations, is the problem of counting of forests of a graph, a forest being just a subset of the edge set, with no cycle (compared to what is said above, it corresponds to the non rooted case, somehow). The number of forests is equals to the specialization $P(2,1)$ of the standard Tutte polynomial; the generating function of forests counted according to the number of edges can also be expressed in terms of the standard Tutte polynomial (see Welsh \& Merino \cite[Formula (18) p.1135]{W-M}). However, note that the general complexity in the evaluation of the Tutte polynomial is $\#P$-hard, even its evaluation $P(2,1)$ (Jaeger et al. \cite{jaeger_vertigan_welsh_1990}); however, in the case of dense graphs, Annan \cite{Annan} provides a ``fully polynomial randomized approximation scheme'' allowing to compute the number of forests, up to a factor $1+\varepsilon$, which in principle, permits approximate uniform generation of these objects (Jerrum et al. \cite{JERRUM1986169}).

Again, the number of forests with some prescribed roots (or their total weights in the weighted case)  can be computed using variant of the matrix tree theorem (see e.g. Chaiken \& Kleitman \cite{MR480115}).

\subsection{Uniform sampling using counting formulae.}
\label{sec:usucf}
The expansion formula \eref{eq:Tcondrec} (or \eref{eq:Tcondrec2}) provides a natural decomposition of the set of subtrees of $G$ containing a given edge $e$ or not. To sample a random tree $\mathcal{T}$ under $\uniform(\TTr{r}{G}{n})$:\\
          -- choose an edge $e$ adjacent to $r$,\\
          -- compute $|\TTr{r}{G\setminus e}{n}|$ and $|\TTr{r}{G.e}{n-1}|$ (using the Tutte recursion),\\
          -- with probability $|\TTr{r}{G\setminus e}{n}|/|\TTr{r}{G}{n}|$, the tree $\mathcal{T}$  is chosen uniformly in $\TTr{r}{G\setminus e}{n}$, otherwise define $\mathcal{T}$ as the tree having as edge set $\{e\}$ union the edge set of a uniform random tree taken in $\TTr{r}{G.e}{n-1}$.
          
        This procedure can be modified to sample in the whole universe $\TTTr{r}{G}$ with probability proportional to $x^{|E(t)|}$ for some fixed $x>0$ (à la Boltzmann)         i.e. $\P(\mathcal{T} = t) = x^{|E(T)|}/\textbf{T}_r(G)$.
        In this case, it suffices to retain $e$ as an edge of the final returned subtree with probability  $x{\bf T}_r(G.e)/{\bf T}_r(G)$, and to go on the construction in $G.e$, or to decide that $e$ is not in the returned subtree with the complementary probability, and to go on the construction in $G\setminus e$. Notice that as we consider/discard edges in the construction of the tree, the consecutive products telescope up to the point where one has $x^{|E(t)|}\textbf{T}_r(G')/\textbf{T}_r(G)$, where $G'$ satisfies $\textbf{T}_r(G')=1$ as explained in \Cref{trivialG}. 

 In this case, when conditioning on the size being $n$, the sample is uniform in $\TTr{r}{G}{n}$. 
        For more on this method see \cite[Section 3]{Du04}.
See Jerrum et al. \cite{JERRUM1986169} for more general facts, concerning the links between the problem of counting and random generation of combinatorial structures. 

The probability of presence of a bunch of edges $e_1,\cdots,e_k$ of $E$ in the random spanning tree can also be determined using the fact that the edge set is a determinantal process, see Burton \& Pemantle \cite{Burton-Pemantle} and Lyons \& Peres \cite[Section 4]{lyons_peres_2017}. The original proof by Kirchhoff uses considerations coming from electrical networks; this method can also be used to prove the negative correlation of presence of two given edge $e$ and $e'$ in  a random (weighted) spanning tree; see also Chap.4 in Lyons \& Peres book \cite{lyons_peres_2017}. More generally then variables $1_e \in T$ for $e\in E(G)$ are negatively associated, as a consequence of the fact that weighted spanning trees edges $(1_e \in T,e \in E)$ form a determinantal process (Burton \& Pemantle \cite{Burton-Pemantle}).

\section{Generation of uniform random trees using Markov chains}
\label{sec:MURT}

\subsection{Algorithmic considerations}

 A graph $G$ can be represented in various ways in a computer. For example, if $E$ is not too large we can use $V=\{1,\cdots,n\}$ and a triangular array $(m_{\{a,b\}}(E), 1\leq a<b\leq n)$ where $m_{\{a,b\}}$ is the multiplicity of the edge $\{a,b\}$ in $E$.
For regular graphs as $\Torus{N}$ or as the complete graph, the edges do not need to be stored, since they can be recovered online.\par
Explicit programming of Markov chains $(X_i,i\geq 0)$ taking their values in $\TTr{r}{G}{n}$, will often imply that, to construct $X_{i+1}$, some (set of) edges and (set of) vertices will be removed or added to $X_i$. In many cases, a ``sub-routine'' devoted to checking if these modifications give a tree is needed to finally accept or reject a modification of $X_i$, and then, to define  $X_{i+1}$.

\paragraph{Checking the tree property is feasible, and has a cost.}

There are some classical algorithms devoted to checking if a subgraph $g$ of a given graph $G$ is a tree: in practice, they have a non-negligible cost (however, at most linear in the size of $g$ if one neglects the access cost to the data). \\
$\bullet$ In all generality, if $g$ is given ``from scratch'', checking if this graph is a tree can be done by performing the breadth-first or depth-first traversal \cite[Sec. 22.2 and 22.3]{Cor09}. \\
$\bullet$  If $g$ has been obtained from a tree by the addition of a single edge and the removal of another one, then checking the tree property can be done as follows: if the edge from $a$ to $b$ has been removed, do the breadth first search from $a$ and check if $b$ is still accessible.\\
$\bullet$ When possible, it is preferable to work with rooted trees instead of unrooted ones. For the canonical orientation in which edges are directed toward the root, all nodes but the root have exactly one outgoing edge (and so, the identity of the edge endpoint can be stored in a 1D array). Assume that we want to add an oriented edge   $(u,v)$ (taken in $\vec{E}$) in the tree and remove say an edge $(a, b)$. Adding $(u,v)$ in $(t,r)$ may:\\
-- either make of $u$ a new leaf, in which case it is easy to see if removing $(a,b)$ preserves the tree property (in words, $a$ or $b$ must be a leaf, and $v$ must be different from $a$);\\
-- or, adding $(u, v)$  produces a (non-oriented) cycle. In this case, $u$ will have two outgoing edges that can be followed  to find the cycle efficiently. From here, it is easy to check if the edge $(a,b)$ is on this cycle, which is a necessary and sufficient condition for the preservation of the tree property upon removal of $(a,b)$ (if the root $r$ is involved in the modifications, the possible choice of a new root may provide some additional details to deal with). The orientations of the edges lying on the cycle have to be modified to get the right orientation of the resulting rooted tree.

\subsection{Three ergodic Markov chains converging to $\uniform(\TT{G}{n})$}
\label{sec:RV}

In what follows we will make use of the following property: if a transition matrix $K$ is symmetric, i.e. $K_{i,j}=K_{j,i}$ for all $i,j\in E$, then the Markov chain is reversible and the uniform measure on $E$ is invariant. 

We present here some dynamics on trees, where each tree being implicitly defined by its edge set. 
In the sequel $t$ and $t'$ are two trees taken in $\TT{G}{n}$, for some $n\geq 2$, and $G=(V,E)$ is a connected graph. The number of edges of both $t$ and $t'$ is $n-1$.

We introduce \textbf{the edge-exchange map} for $G=(V,E)$, as the map defined as
\[\app{\Ex}{\TT{G}{n}\times E\times E}{\TT{G}{n}}{(t,e,e')}{t'=\Ex(t,e,e')}\]
where:\\
$\bullet$ $t'$ is defined from $E(t')=(E(t)\cup\{e\}) \setminus \{e'\}$ if this set of edges defines a tree,\\
$\bullet$ $t'=t$ otherwise. \medbreak

\NewKernel{Exchange the status of two edges of $G$}{labela}{
Suppose $X_0 \in \TT{G}{n}$ is given. To get $X_1\sim K^{\Kerref{labela}}(X_0,.)$, just set $X_1\eqd \Ex(X_0,{\bf e}_1,{\bf e}_2)$ where ${\bf e}_1$ and ${\bf e}_2$ are two edges taken uniformly and independently in $E$.}
\Analysis: The chain is clearly aperiodic, irreducible and symmetric; $\uniform(\TT{G}{n})$ is its unique invariant distribution and ergodicity is ensured by the Perron-Frobeni\"us theorem.\\
\Drawback: If $|E|$ is big compared to $n$, most of the transitions will leave $t$ unchanged, which results in a very long mixing time. When $t$ is changed, checking the tree property is expensive for large $n$.

\NewKernel{Exchange the status of two edges adjacent to the current tree}{labelc}{ Assume that $X_0=t \in \TT{G}{n}$. To get $X_1\sim K^{\Kerref{labelc}}(X_0,.)$, construct two edges  $\overrightarrow{{\bf e}_1}=(\mathbf{u},\mathbf{u}')$ and $\overrightarrow{{\bf e}_2}= (\mathbf{v},\mathbf{v}')$ such that $(\mathbf{u},\mathbf{v})$ are two i.i.d. uniform random nodes of $t$, $\mathbf{u}'$ and $\mathbf{v}'$ are respectively, a uniform neighbour of  $\mathbf{u}$ and of $\mathbf{v}$ (independent). \\
  If ($\mathbf{v}'$ is a  leaf and $\mathbf{u}'$ is outside $t$) then set $X_1=\Ex(t,{{\bf e}_1},{\bf e}_2)$.} 
Rensburg \& Madras \cite{RM} gave this algorithm (Algorithm A in their paper) for lattice trees (and here, we made a small modification to take into account the non-constancy of the node degrees).\\
\Analysis: A simple check shows that this kernel is also aperiodic and irreducible. The probability of a transition from $t$ to $t'\neq t$ is $1/(n^2\, \degree_G(\mathbf{u})\, \degree_G(\mathbf{v}))$ if it can be attained from $\Ex$. Observe that the tree obtained $t'$ has also $n$ nodes, and still $\mathbf{u}$ and $\mathbf{v}$ are some of them. We then get the same probability from $t'$ to choose $(\mathbf{v},\mathbf{u})$ (instead of $(\mathbf{u},\mathbf{v})$) and then $(\mathbf{v}',\mathbf{u}')$ as neighbours from what we see that $K^{\Kerref{labelc}}(t,t')=K^{\Kerref{labelc}}(t',t)$ and therefore its unique invariant is $\uniform(\TT{G}{n})$.\\
\Drawback: Checking the tree property is expensive for large $n$.

\begin{rem}
 Variants are available for all these transition matrices. For example, in $K^{\Kerref{labela}}$ one can consider $(\vec{e}_1, \vec{e}_2)$ drawn from many symmetric distribution with full support over $\vec{E}(G)^2$. In $K^{\Kerref{labelc}}$ one can take $(\mathbf{u},\mathbf{v})$ chosen with any symmetric distribution with full support over $V(G)^2$. 
\end{rem}

\subsection*{The fastest Markov chain}
Using $K^{\Kerref{labela}}$, when $n$ is a bit large, it is unlikely that both edges belong to the same cycle, so that $K^{\Kerref{labela}}$ is slow to mix, because, it mainly changes the ``peripheral edges''. 
$K^{\Kerref{labelc}}$ is somehow worst, since modifications exchange leaves and perimeter edges. \par
The main idea of the next kernel is the following: when the first added edge creates a cycle, then force the second edge to be in this cycle!

On a general graph, there is basically a single way to design such a reversible kernel, when, on a regular graph (on which the degree vertices are constant), several methods can be proposed.

For any simple cycle $(c_0,\cdots,c_{m-1})$ where $c_{i}$ and $c_{i+1 \mod m}$ are neighbours for $i\in\{0,\cdots,m-1\}$, denote by $p_c$ the following distribution on the set of (non-oriented) edges of $c$:
\ben\label{eq:PC}
p_c(\{c_i,c_{i+1 \mod m}\})= \alpha_c\l(\frac{1}{\degree_G(c_i)}+\frac{1}{\degree_G(c_{i+1 \mod m})}\r)
\een
where $\alpha_c$ is the single constant making of $p_c$ a probability on the set of  edges of $c$.
The distribution $p_c$ does not depend on the cyclic order chosen on $c$, nor on the orientation of $c$.

\NewKernel{If the added edge forms a cycle, then break the cycle}{labeldd}{
  Assume that $X_0=t$, to get $X_1\sim K^{\Kerref{labeldd}}(t,.)$ do the following. Take the random oriented edge  $\overrightarrow{\bf e} = (\textbf{u},\textbf{u}')$, where $\textbf{u}$ is a uniform vertex of $t$, and conditional on $\textbf{u}$, $\textbf{u}'$ is uniform among the neighbours of $\textbf{u}$.

\noindent$\bullet$ If the addition of $\ar{\bf e}=({\bf u},{\bf u'})$ to $t$ creates a new leaf, then pick a second independent edge $\ar{\bf e}'=({\bf v},{\bf v'})$ (with the same law as $\ar{\bf e}$). If $\ar{\bf e}'$ is a leaf of $t$ and the removal of $\ar{\bf e}'$ in $t \cup \ar{\bf e}$ produces a tree $t'$ then take $X_1=t'$ else take $X_1=t$.\\
$\bullet$ otherwise 
${\bf u'}$ already belongs to $t$ so that adding $\ar{\bf e}$ creates a cycle $c=(c_0,\cdots,c_{m-1})$ (taken in an arbitrary cyclic order, where $m$ is the cycle length). Take $e$ an unoriented edge of $c$ according to $p_c$ defined in \eref{eq:PC}; we then define $X_1$ as the tree obtained by the addition of the edge ${\bf e}$ followed by the removal of $e$.}

\noindent\Analysis: It is irreducible and aperiodic. The chain is reversible: the (not-so) delicate point to check, is when the addition of $\ar{\bf e}$ creates a cycle.
In this case, the probability that the unoriented edge ${\bf e}$ (to be added) is $\{u,v\}$ is
\[q_t(\{u,v\}):=\frac{1}{|t|} \l(\frac{1}{\degree_G(u)}+\frac{1}{\degree_G(v)}\r)\]
since this occurs if $({\bf u},{\bf u}')$ is $(u,v)$ or $(v,u)$. If adding this edge creates a cycle $c$, then an edge $\{u',v'\}$ of the cycle, will be removed with probability $p_c(\{u',v'\})$, so that globally, the probability to insert $\{u,v\}$ and then to remove $\{u',v'\}$ is $q_t(\{u,v\}) p_c(\{u',v'\})$. Now, the probability to instead, insert $\{u',v'\}$ and then remove $\{u,v\}$ is $q_t(\{u',v'\}) p_c(\{u,v\})=q_t(\{u,v\}) p_c(\{u',v'\})$ (the main point is that the same cycle is then created).\\
\Drawback: Checking the tree property in the rooted case is fast. Again, we did not succeed to provide a coupling from the past for this Markov chain, nor to get some bounds on the mixing time. Simulations show that it is much faster than the other transition matrices in practice.

\begin{rem}[Possible modifications on  regular graphs]
On a regular graph, since $q_t(\{u,v\})$ is constant, many distribution $p'_c(.)$ can be considered instead of $p_c(.)$. For example, one can take a distribution $p_{c',{\bf e}}(.)$ on $c$ depending on the position of ${\bf e}$ on this cycle (for example, $p'_{c,{\bf e}}(.)$ can weight the edges of $c$ according to their distance to ${\bf e}$ on $c$).
\end{rem}

\section{Simulations of uniform subtree of the grid with a given size}
\label{sec:SUBG}
The torus $\Torus{N}$ as well as $\mathbb{Z}^2$ are highly regular graphs for which some methods of generation and of exact enumerations can be designed for this particular case (as well as for $D$-dimensional analogue). We provide some references after \Cref{conj:rthgfqds} (below).

\subsection{Subtrees of the torus up to translation}

\label{sec:STT}

We say that $t$ and $t'$ in $\TT{\Torus{N}}{n}$ are $N$-equivalent if they are equal up to a translation in $\Torus{N}$, and let $\overline{\TT{\Torus{N}}{n}}$ be the set of equivalent classes.
The push-forward measure of $\uniform(\TT{\Torus{N}}{n})$ by the canonical projection $\pi_N$ is $\uniform(\overline{\TT{\Torus{N}}{n}})$  since all classes have  cardinality $N^2$.
  Since the diameter of any tree with $n$ nodes is smaller than $n-1$, the previous discussion shows that the uniform distribution on $\TT{\Torus{N}}{n}$ and on  $\TT{\Torus{N'}}{n}$ can be identified up to random uniform translation, if $N$ and $N'$ are both bigger or equal than $n$. When one wants to sample uniformly in $\TT{\Torus{N}}{n}$ it is then reasonable to work in $\TT{\Torus{n}}{n}$ (the smallest valid torus), or to work up to translation. Indeed,  when one works under the kernel $K^{\Kerref{labeldd}}$, the mixing time of the chain depends on the size of the torus since the larger is the torus, the longer it takes to forget ``not only the shape of the initial tree'', but also its position.

Observe also that sampling in $\TTr{(0,0)}{\Torus{n}}{n}$ and in $\TT{\Torus{n}}{n}$ are basically equivalent, since it is easy to sample one, from the other.

\subsection{Some pictures}
We programmed and ran the chain $K^{\Kerref{labeldd}}$. We made some statistics and videos to show the power and limits of this kernel; in few words, it can be used to sample a random tree with a distribution close to $\uniform(\TTr{(0,0)}{\Torus{n}}{n}$, for $n$ up to say 8000 nodes in few minutes, and $n=10000$ in few hours using a program written in C on a standard computer, starting from any distribution.

Our program starts from a rectangle tree, see Fig. \ref{fig:rec} which is a highly structured tree; we tried many Markov transition matrices with this kind of starting point and only efficient Markov chains ``forget'' the initial distribution in a reasonable time.
\begin{figure}[h!]
	\centerline{\includegraphics[scale=1.2]{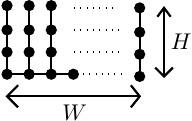}}
 \captionn{A rectangle-tree with width $W$, and height $H$. \label{fig:rec}} 
\end{figure}

\begin{figure}[h!]
  \begin{center}
    \includegraphics[width = 5cm]{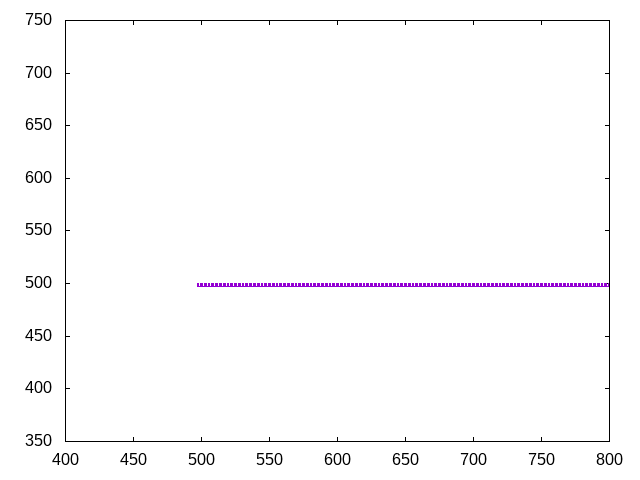}
    \includegraphics[width = 5cm]{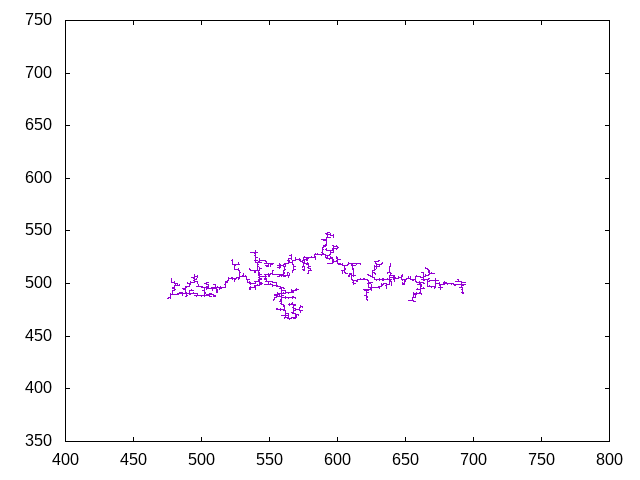}
    \includegraphics[width = 5cm]{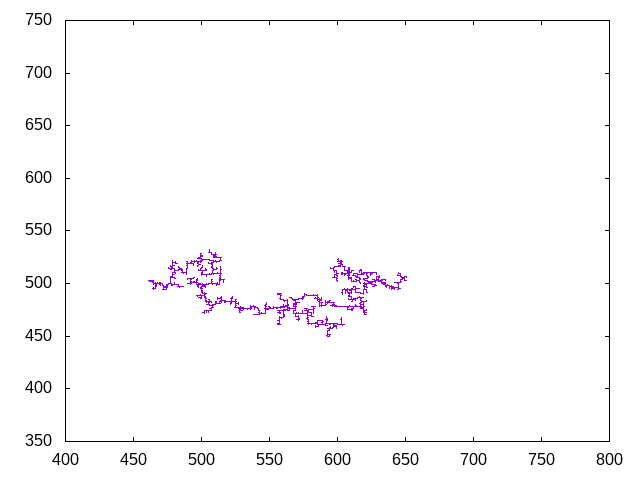}\\
     \includegraphics[width = 5cm]{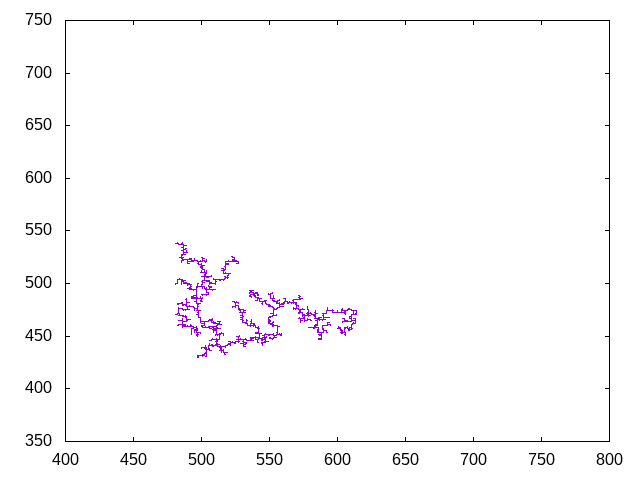}
    \includegraphics[width = 5cm]{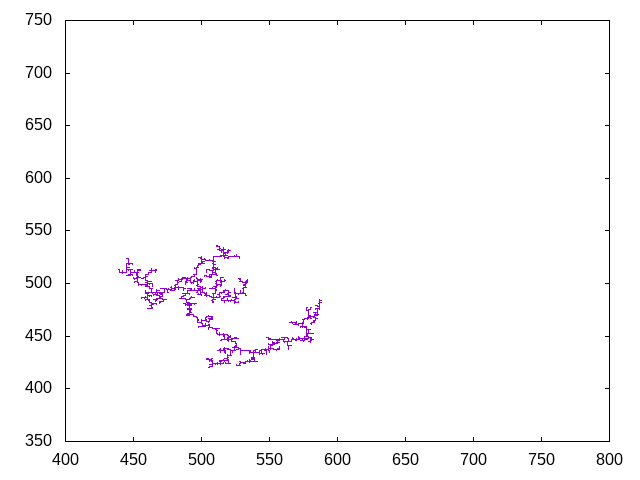}
    \includegraphics[width = 5cm]{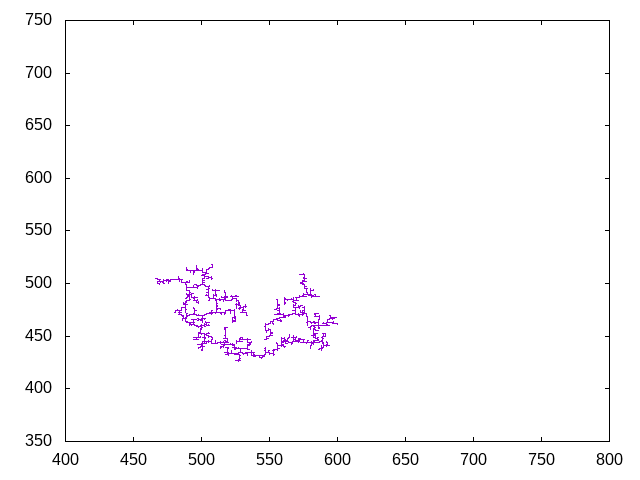}\\
    \includegraphics[width = 8cm]{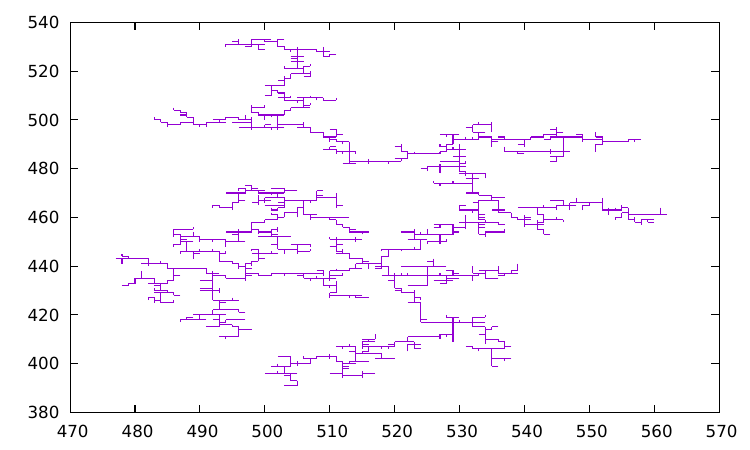}
    \end{center}
    \captionn{\label{fig:MC}Markov chain started from a rectangle tree $400\times 4$, with 1600 nodes, run on $\Torus{1000}$, and observed at time $k\times$ 200 millions, for the $k$th picture. The total execution time is around 1 minute. The last tree is the result after 1.6G iterations. A film with 800 images of the 1.6G steps of the chain (2M steps between successive images) is available at \cite{FM}. }
  \end{figure}

  \begin{figure}[h!]
  \begin{center}
    \includegraphics[width = 5cm]{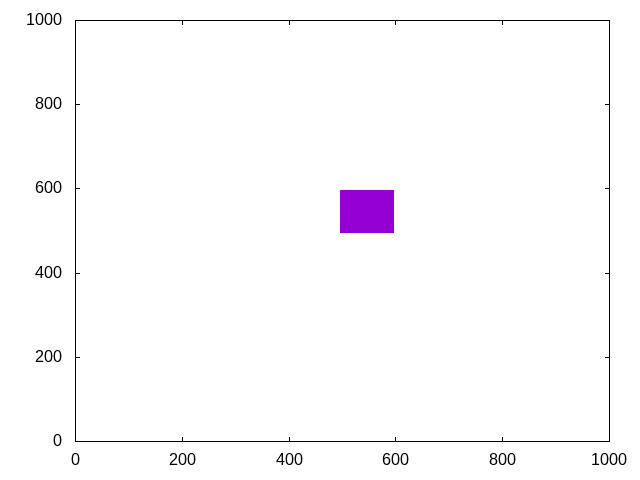}
    \includegraphics[width = 5cm]{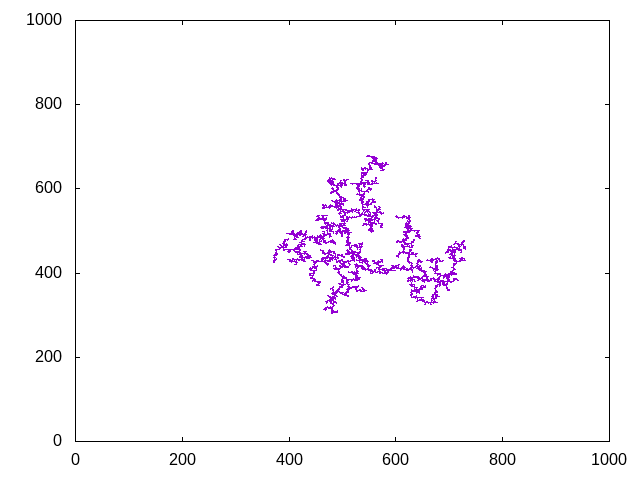}
    \includegraphics[width = 5cm]{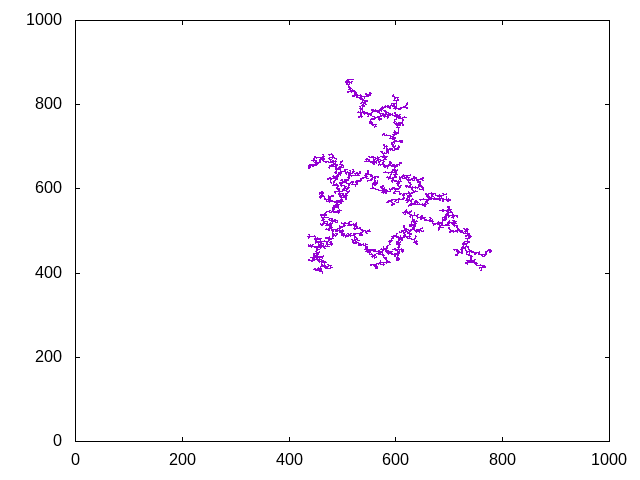}\\
     \includegraphics[width = 5cm]{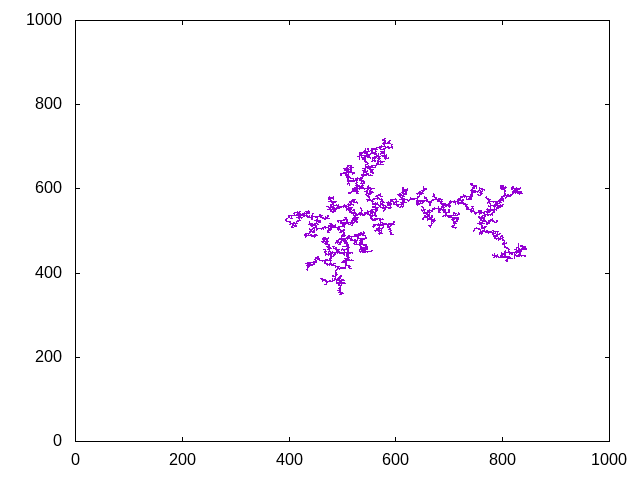}
    \includegraphics[width = 5cm]{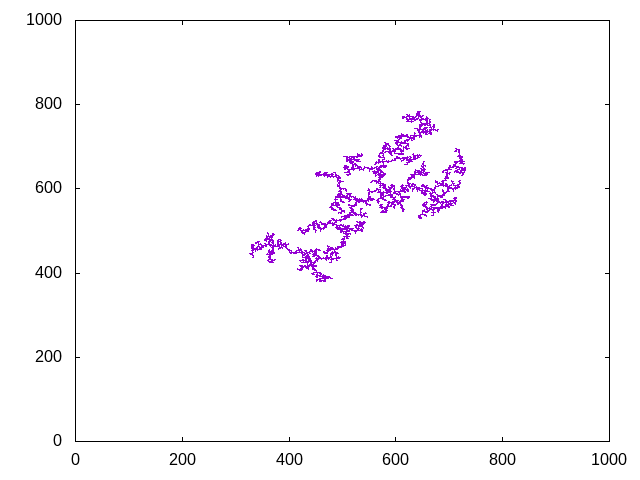}
    \includegraphics[width = 5cm]{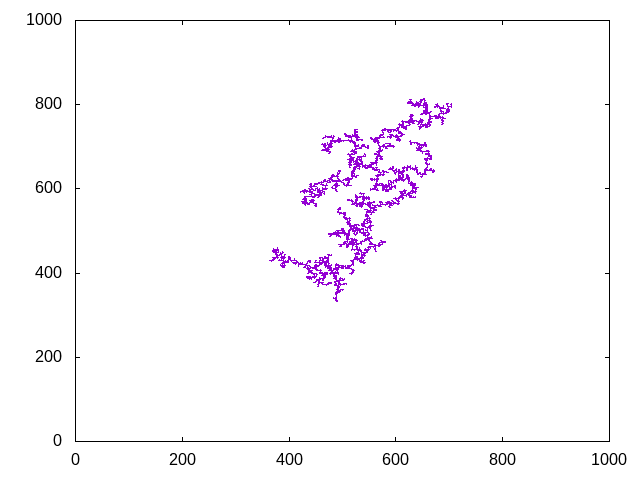}\\
    \includegraphics[width = 8cm]{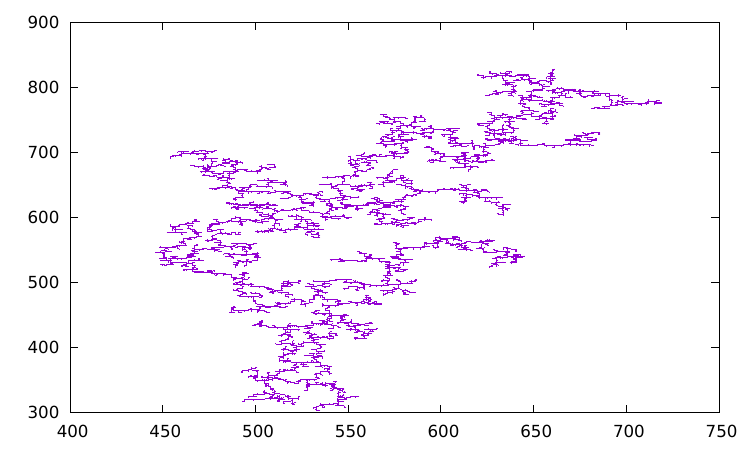}
    \end{center}
    \captionn{\label{fig:MC2} Simulation as in Fig. \ref{fig:MC} starting from a rectangle tree $100\times 100$, ran on $\Torus{1000}$, and with $10^{10}$ steps of the chain between successive pictures. The last tree is the value of the chain after $80G$ steps. The total execution time is around 1 hour. A film with 800 images of the $80G$ steps of the chain ($10^8$ steps between successive images) is available at \cite{FM}. In this picture, macroscopic portion of the pictures are really close to loops.}
  \end{figure}
  
\subsection{Statistics and conjectures}\label{sc:SC}
For any tree $t$ in  $\TT{ {\sf Torus}(N(n))}{n}$, define the \textbf{Euclidean width and height} $w(t)$ and $h(t)$ as respectively the number of columns and rows of the torus containing at least one vertex of $t$. 
The second variable of interest is the random graph distance $ {\bf D}( t) = d_{ t}({\bf u},{\bf v})$ between two i.i.d. uniform nodes ${\bf u}$ and ${\bf v}$ of a (deterministic or random) tree ${t}$.\\
The proportion of nodes in $t$ with degree $j$ is
\[ q_j(t) = {|\{u \in t~: {\sf Degree}(u)=j\}|}\,/\,{|V(t)|}.\]
We conjecture the following (recall the discussion at the beginning of Section \ref{sec:STT}).
\begin{conj}\label{conj:rthgfqds}
  For ${\bf t}_n$ taken uniformly in $\TT{ {\sf Torus}(n)}{n}$ 
  \bir
  \itr there exists $\alpha \in [0.63,0.67]$ such that,
    ${\l(w({\bf t}_n),h({\bf t}_n)\r)}\,/\,{n^{\alpha}}\dd({\bf w},{\bf h})$, where ${\bf w}$ and ${\bf h}$ are almost surely non zero.
\itr there exists $\beta \in [3/4 -0.01,3/4+0.01]$ such that, 
  ${{\bf D}({\bf t}_n)}\,/\,{n^{\beta}}\dd {\bf D}$ 
where ${\bf D}$ is a real random variable, almost surely non-zero.
\itr \label{degree} $(q_j({\bf t}_n), 1\leq j \leq 4)\proba (q_1,\cdots,q_4)$ a constant vector satisfying $q_1\in [0.2585 \pm 0.001]$,  $q_2\in [0.506 \pm 0.001]$, $q_3\in [0.214\pm 0.001]$, $q_4\in [0.02185 \pm 0.001]$.
\eir 
\end{conj}
Rensburg \& Madras \cite{RM} (1992) proposed mainly two Markov chains to produce lattice trees; the first one is (Algorithm A) and coincides with kernel $K^{\Kerref{labelc}}$.\par
Their second algorithm \textbf{(Algorithm B)}, which produces an irreducible and reversible Markov chain, is valid on lattices ($\Z^d$, for $d\geq 2$) consists in the following stages: choose uniformly an edge $e$ in the current tree $X_0=t$ and try to do the following:\\
  -- remove $e$,\\
  -- apply a randomly chosen element of the octohedral group to the smallest connected component,\\
  -- choose one random node $u$ and $v$ uniformly on each of the connected components, (assuming that $u$ is on the smallest connected one)\\
  -- translate $u$ to $u'$ (together with the smallest connect component) such that $u'$ is a uniform neighbour of $v$. Add the edge $u'v$.\\
  If the resulting graph is a tree $t'$, then set $X_1=t'$ else set $X_1=t$.\medskip
  
  Using Monte Carlo methods, they estimated the order of the radius of gyration (which is the mean Euclidean distance between two points taken uniformly in the tree) to $n^{a}$ where $a=0.6374$, and the longest graph distance between two points at $0.7358$. Below table 9 in their paper, Rensburg \& Madras \cite{RM} (1992) provide a survey of the results available at this time concerning simulation of lattice trees, as well as ``guesses'' using methods of statistical physics of the value of $a$.

Jensen \cite[Section 3]{Jensen} (2000) using exact enumerations of ``lattice trees'' up to size 42, conjectured that the order of the radius of gyration of $\bt_n$ is $n^{a}$ with  $a=0.64115(5)$. This conjecture is built using some exact partial generating functions (relying on the  exact enumerations up to size 42) together with some regularity assumptions on the generating functions. It is reasonable to conjecture that $a$ and $\alpha$ (of our conjecture) are equal. Jensen \cite[Section 3]{Jensen} produces also some exact values of the number of elements in  $\TTr{(0,0)}{\Z^2}{n}$ for $n\leq 42$ and conjectured that $\log(\TTr{(0,0)}{\Z^2}{n})/\log(n) \to 3.795254...$.

Rensburg \& Rechnitzer \cite{RR} (2003), using Monte Carlo method estimated the metric exponent to $\nu=0.6437\pm 0.0035$, and the longest path exponent (for the graph distance) to $\rho = 0.74000\pm 0.00062$ (one can conjecture that $\rho$ and $\beta$ are equal).

 Hsu et al. \cite{Hsu_2005} (2005) (see also references therein) discuss a simulation of lattice trees (and branching polymers) constructed on the pruned-enriched Rosenbluth method (PERM) (in dimension $2\leq d \leq 9$). They estimated $\nu$ at 0.6412(5) (many more statistics are studied; they provide an important survey on the result available at this time).

   Finally, we would like to mention, that Botet and Jullien \cite{Bot} in 1985, discussed a model of diffusion-limited aggregation with disaggregation; it was an attempt to define a Markov chain on a DLA like cluster (see Section \ref{sec:dla} for definitions and statistics), having the DLA distribution as invariant distribution. To be precise, their Markov chain $(X_t,t\geq 0)$ is a tree valued Markov chain, with state space $\TTr{(0,0)}{\Z^2}{n}$, and their hope was that the vertex set $V(X_t)$ of $X_t$, would be distributed as the DLA, when $X_t$ was taken under its invariant distribution. \par
   They noticed that the Markov chain they defined does not reach this aim, since the mean gyration radius is around $n^{c}$ with $c\simeq 0.65$, which is not compatible with the DLA statistics, but, as explained just above, this value is compatible with the statistic presented above for a uniform element of $\TTr{(0,0)}{\Z^2}{n}$. \par
   This has possibly been unnoticed, but a very small (time) modification of their Markov chain admits indeed, the uniform distribution on  $\TTr{(0,0)}{\Z^2}{n}$ as invariant distribution.\par
   The Markov kernel of their chain is defined as follows; assume that at time $t$, the current state $X_t$ is a tree $T$ of $\TTr{(0,0)}{\Z^2}{n}$. In order to define $X_{t+1}$, proceed as follows: choose a leaf $v$ of $T$, uniformly at random (the root is never considered as a leaf). Then, erase $v$ and its incident edge from $T$, and starts a random walk $(W(k),k\geq 0)$, starting at $v$, and stopped at its hitting time $\tau$ of $V(T \setminus\{v\})$ (the tree $T$ deprived of $v$). To define $X_{t+1}$, remove $v$ from $T$ and its incident edge, and add the edge corresponding to the last step of this walk, $w=W(\tau-1)\to w' =W(\tau)$.\par
   This chain is not reversible, because $X_t$ and $X_{t+1}$ may have a different number of leaves. However, we may propose the following modification: if instead of choosing a leaf, one chooses a uniform node $v\in T$, and decide to set $X_{t+1}=X_t$ if $v$ is not a leaf, then the Markov chain is reversible (on a regular graph), so that it preserves the uniform distribution on $\TTr{(0,0)}{\Z^2}{n}$. \par
   The invariant distribution of Botet and Jullien \cite{Bot} is then the probability distribution on the set $\TTr{(0,0)}{\Z^2}{n}$ giving to each tree $t$ a probability proportional to $|\partial t|$ (number of leaves, different from the root), since it stays a time 1 on each configuration (before launching a random walk), when the modification we propose, stays a mean time $n/|\partial t|$ on a tree with $|\partial t|$ leaves, before starting the random walk.\par
 However, Botet and Jullien \cite{Bot} Markov chain is slow compared to the three ones presented at the beginning of \Cref{sec:RV}, since these three avoid to performing random walks to choose the new destination of a moving edge.
\begin{conj}\label{conj:etgf} Consider $\bt_n$ a uniform subtree of $\Z^2$ with $n$ nodes, containing $(0,0)$. Denote by ${\sf Drawing}(\bt_n)$ the drawing to $\bt_n$ in the plane (the vertices are points, the edges are segments). There exist $\alpha \in [0.63,0.67]$ and $\beta \in [3/4 -0.01,3/4+0.01]$ such that
  \bir \itr The sequence of compact sets ${\sf Drawing}(\bt_n)/n^{\alpha}$ converges in distribution, for the Hausdorff metric topology on compact sets of the plane, to a non-trivial path connected random compact set $K$ of $\R^2$, with empty interior. Moreover, $K$ is almost surely not a tree: almost surely, there exist some pairs of points $(x,y)$ with two different injective paths from $x$ to $y$ (meaning the set of points of these paths are different).
  \itr The sequence of trees $(\bt_n,d_{\bt_n}/n^{\beta})$ seen as a sequence of compact spaces equipped with their graph distance normalized by $n^{\beta}$ converges in distribution, for the Gromov-Hausdorff distance to a (non-trivial) continuum random tree.\eir 
\end{conj}
\begin{rem} The simulation of ``approximately uniform'' lattice trees with $n$ nodes (and $n$ large) shows the ``appearance'' of macroscopic cycles. The word ``appearance'' is here to express the fact that there is no cycle, since the drawing of a tree has no cycle but the normalization needed to draw the tree creates this appearance (see Fig. \ref{fig:MC2}) (more precisely, it seems that, for $\varepsilon>0$ small,  the drawing of a large tree, normalized by $n^{\alpha}$, is at Hausdorff distance $\leq \varepsilon$ to a compact set having some cycles with a significant perimeter $\gg 2\varepsilon$. This is the reason for \Cref{conj:etgf}$(i)$. If it is indeed the case, infinitely many cycles are likely to be present.
\end{rem}

\begin{rem}
The conjectured limiting proportions of nodes of each degree (iii) are different from those of the UST in $\Z^2$ (see \cite[P. 112]{lyons_peres_2017}). 
\end{rem}

 Now, we add that in large dimension (notably in the case $D>8$), the asymptotic behaviour of lattice trees is well understood (see e.g Hara \& Slade \cite{MR1063208}, Derbez \& Slade \cite{MR1620301}, Slade \cite{Slade}, Holmes \cite{MR2399294,Holmes_2016}, Cabezas et al. \cite{CFHP} and references therein).

\subsection{Simulations}\label{sec:simu}
We made thousands of simulations of this chain (on a multicore PC), each of them running for many steps; 
\ben
\begin{array}{|c|c|c|c|c|}
	\hline
	\textrm{Tree size}  & 1000 & 2500 & 5000 & 8100\pass\hline
	\textrm{Number of simulations}  & 5039 & 5486 & 6111 & 5232\pass\hline
	\textrm{Initial rectangle tree shape}  & 40\times 25 & 50\times 50 & 50\times 100 & 90 \times 90\pass\hline
	\textrm{Nb Steps of the chain}  & 150M & 1G & 25G & 200G\pass\hline
\end{array}
\een
hence, we made 5486 simulations of  trees of size 2500 starting initially with a rectangle tree $50\times 50$, 1G steps of the Markov chain for each tree simulated.
These numbers of steps were decided ``empirically'': starting from a rectangle tree, for example, with size $1000\times 1$ or $40\times 25$, and performing hundreds of simulations with $s$ steps, suffices to compare some statistics as the width and the height, which are asymptotically the same (independently of the initial tree)~: in case of discordance of these statistics, $s$ must be taken larger. The videos (available at \cite{FM}) give some clues that the mixing time should have been reached (if one considers the trees up to translation), even if these simulations do not constitute a formal proof, of course.

To make the estimates associated to the width, both the width and the height of each tree has been used (two numbers by simulations), and for the graph distance, for each tree, 10 independent pairs of vertices $[(u_{2i-1},u_{2i}),1\leq i \leq 10]$ were chosen to compute the graph distance $d_{t}(u_{2i-1},u_{2i})$, where $u_{2i-1},u_{2i}$ are independent and uniform in the vertex set of the tree $t$; this provides 10 numbers for each tree. These 10 values are dependent, as are the width and the height. \\
Now, for each of the sampled trees, the exact number of nodes of each degree has been computed, which provides for each tree a proportion vector $(q_i(t), 1\leq i \leq t)$. 
\paragraph{Distance statistics}
\ben
\begin{array}{|c|c|c|c|c|}
	\hline
	\textrm{Number of nodes}  & 1000 & 2500 & 5000 & 8100\pass\hline
	\textrm{Empirical mean of the width}  & 96.41 & 173.58 & 273.63 & 372.25\pass\hline
	\textrm{Empirical median of the width}  & 95.00 & 171.00 & 269.00 & 367.00\pass\hline
	\textrm{Empirical mean of $d({\bf u},{\bf v})$}  & 95.68 & 189.60 & 317.92 & 457.48\pass\hline
	\textrm{Empirical median of $d({\bf u},{\bf v})$}  & 88.00 & 176.00 & 293.00 & 421.00\pass\hline
\end{array}
\een

Suppose that a sequence of real random variables $(Y_n)$ satisfies $Y_n/n^{\gamma}\dd Z$ for some $\gamma>0$ and non-trivial $Z$, then it is expected that for $n$ and $m$ both large, ${\sf median}(Y_n)/{\sf median}(Y_m)$ should be close to $(n/m)^{\gamma}$.
Assuming that we have a sample from i.i.d. copies of $Y_n$,  $(Y^{(i)}_n,1\leq i \leq N)$, then we can define the empirical mean $\widehat{Y_n}= (Y_n^{(1)}+\cdots+Y_n^{(N)})/N$, and the empirical median ($\widehat{\sf median}(Y_n)=\inf\{x : |\{j:Y_n^{(j)}\leq x\}|\geq N/2$).
This provides the following estimator for $\gamma$, where samples for two different values of $n$ and $m$ are needed:
\ben
{\sf Est}_{\sf median}(\gamma)= \log\l(\widehat{\sf median}(X_n)/ \widehat{\sf median}(X_{m}) \r) /\log(n/m).
\een
By the same method, a second estimator using the empirical mean is
\ben
{\sf Est}_{\sf mean}(\gamma) = \log\l(\widehat{X_n}/\widehat{X_m} \r) /\log(n/m).
\een
Finally, we introduce a last estimator of the exponent $\gamma$ using the  9 empirical deciles $({\sf Dec}_i(Y_n),1\leq i \leq 9)$ where ${\sf Dec}_i(Y_n)=\min\{x :  |\{j:Y_n^{(j)}\leq x\}|\geq N i/10\}$.  We then take $\gamma$ as the values that minimises the $L^2$ distance between the vectors $m^x({\sf Dec}_i(Y_n),1\leq i \leq 9)$ and $n^{x}({\sf Dec}_i(Y_m),1\leq i \leq 9)$:
\[{\sf Est}_{\sf best\ fit\ decile}(\gamma)=\argmin\l( x\mapsto \sum_{i=1}^9 \l|{\sf Dec}_i(Y_n)m^x-{\sf Dec}_i(Y_m)n^x    \r|^2\r), \]
(for $x\in [1/2,1]$)
which we expect to be better than the median, since it takes into account the other deciles.\footnote{the estimator $\argmin\l( x\mapsto \sum_{i=1}^9 \l|{\sf Dec}_i(Y_n)/n^x-{\sf Dec}_i(Y_m)/m^x    \r|^2\r)$ is not good, since it is often reached for $x=1$, for which all the terms inside the absolute value are small.}

Using $(n,m)$ gives the following estimate:

\ben
\begin{array}{|c|c|c|c|c|}
	\hline
	\textrm{$(n,m)$}  & (1000, 2500) & (2500, 5000) & (5000, 8100)\pass\hline
	\textrm{Estimation of $\alpha$ (mean) }  & 0.642 & 0.657 & 0.638\pass\hline
	\textrm{Estimation of $\alpha$ (median) }  & 0.641 & 0.654 & 0.644\pass\hline
	\textrm{Best fit decile $\alpha$  }  & 0.640 & 0.656 & 0.635\pass\hline
	\textrm{Estimation of $\beta$ (mean) }  & 0.746 & 0.746 & 0.754\pass\hline
	\textrm{Estimation of $\beta$ (median) }  & 0.756 & 0.735 & 0.751\pass\hline
	\textrm{Best fit decile $\beta$}  & 0.744 & 0.748 & 0.753\pass\hline
\end{array}
\een

\begin{figure}[h!] \centerline{\includegraphics[width = 6cm]{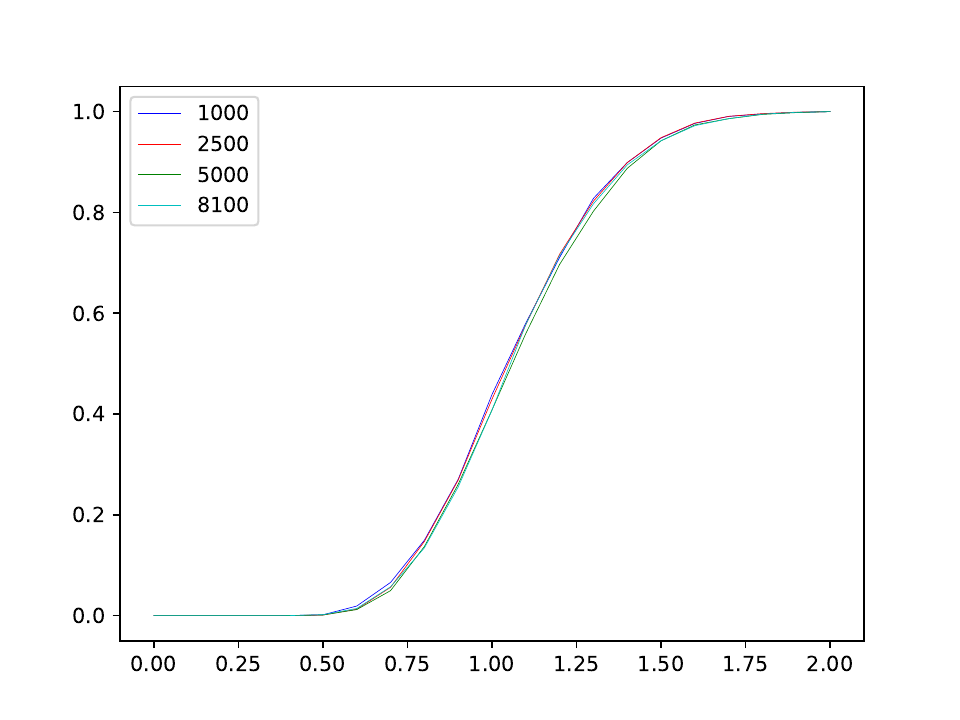}\includegraphics[width = 6cm]{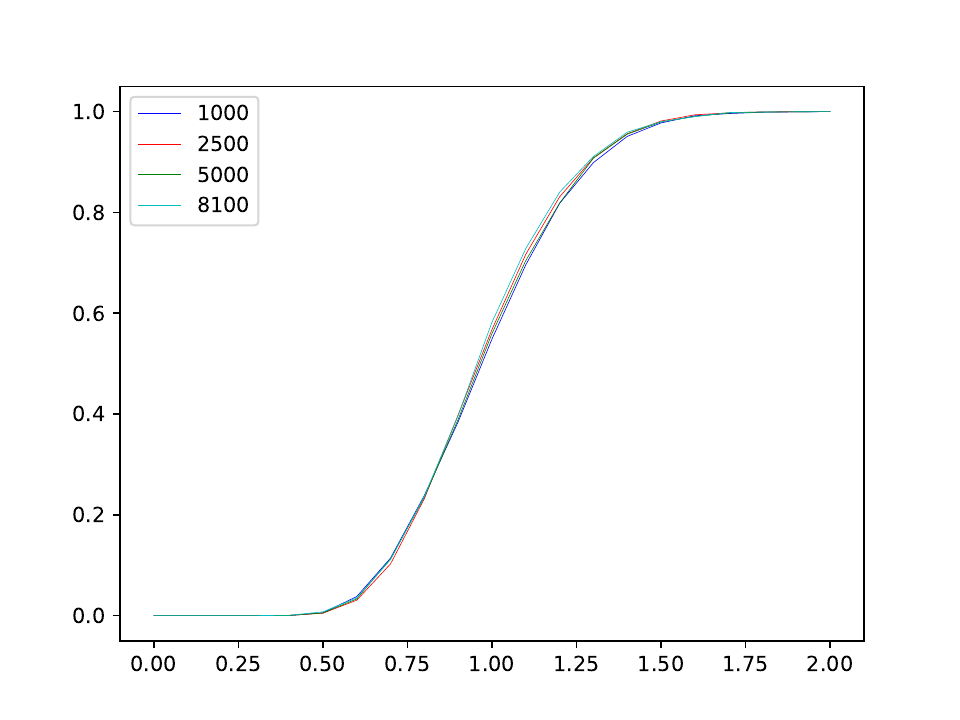}\includegraphics[width = 6cm]{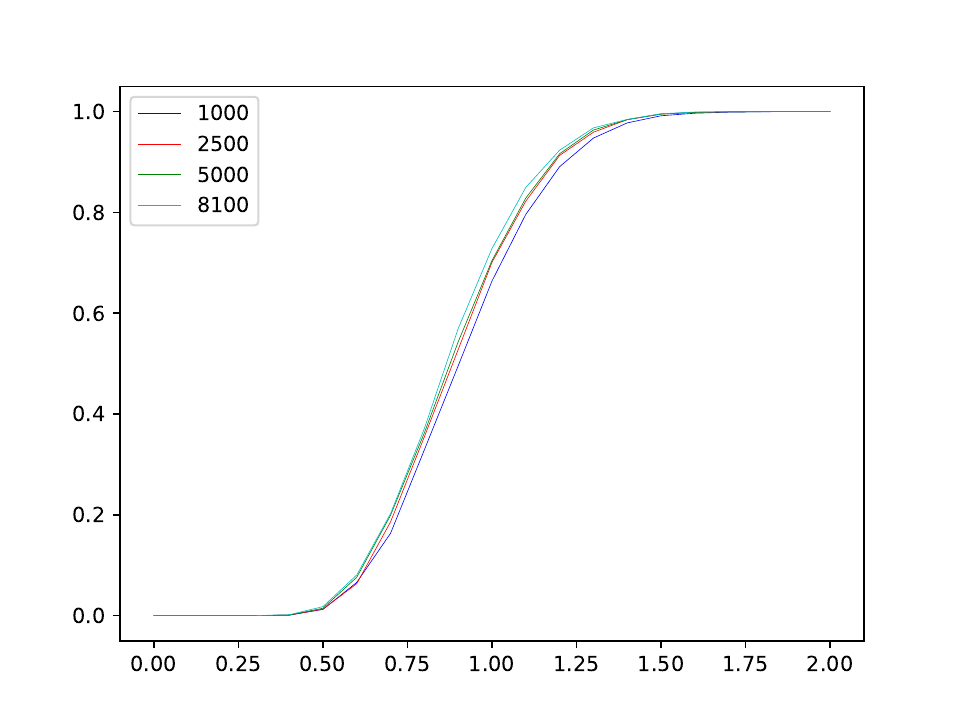}} 
  \centerline{\includegraphics[width = 6cm]{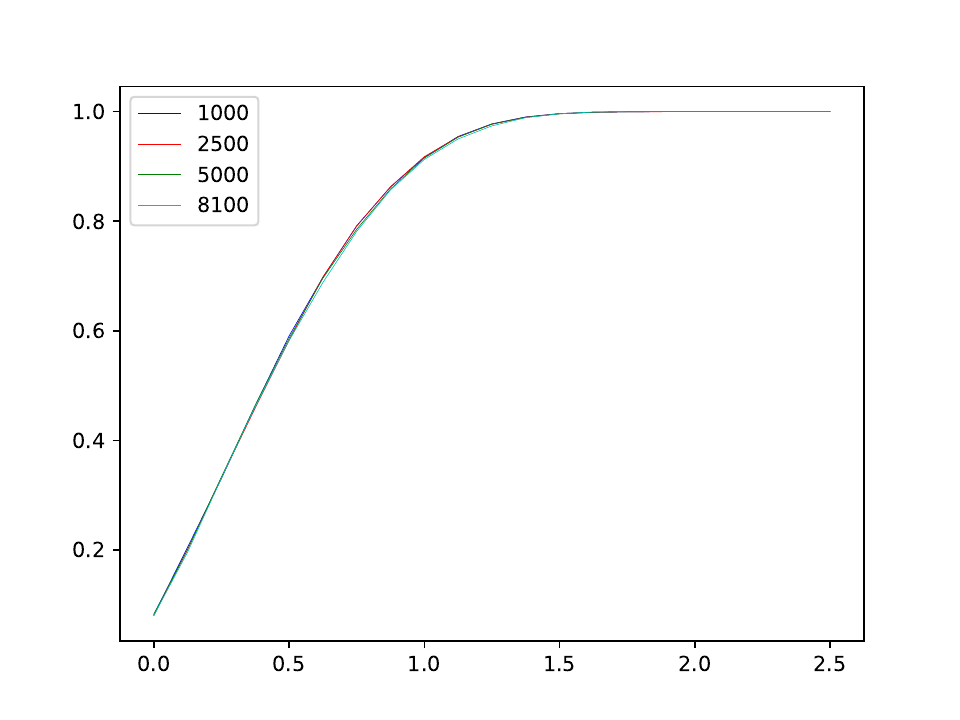}\includegraphics[width = 6cm]{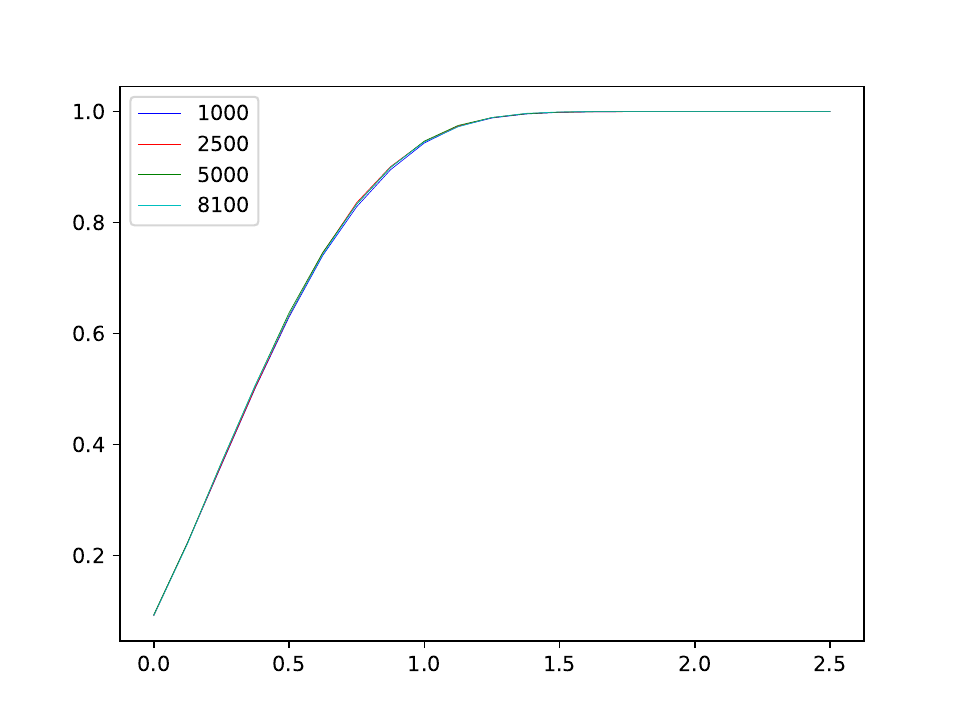}\includegraphics[width = 6cm]{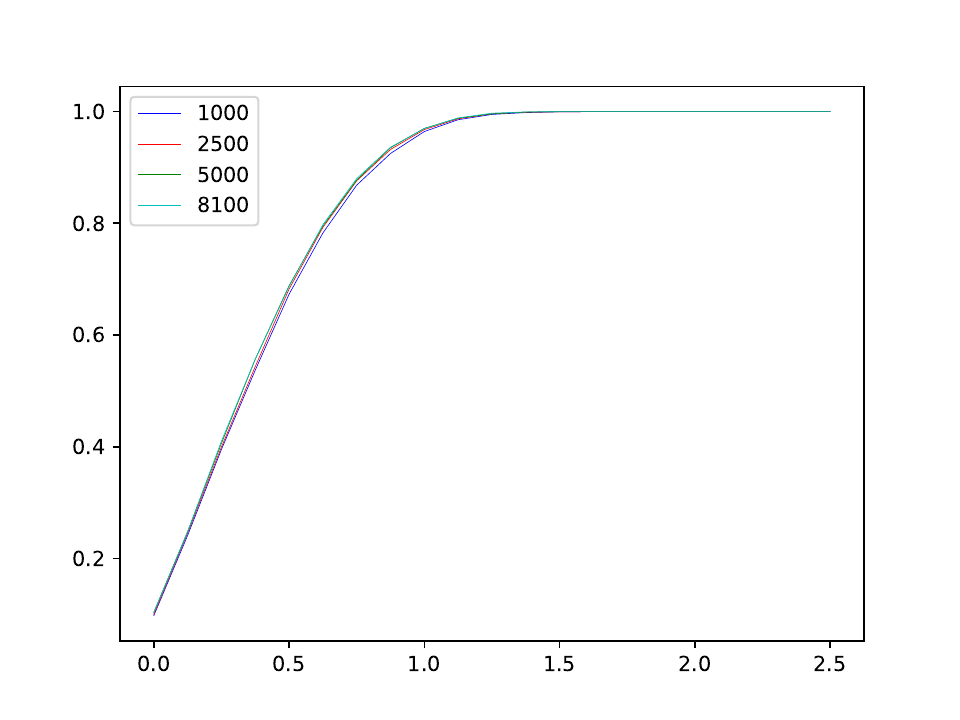}}
  \captionn{\label{fig:dqdqtf}On the first line, (interpolated) empirical cumulative function of $w(t_n)/n^{\alpha}$ for $\alpha$ being respectively 0.64, 0.65 and 0.66. On the second, (interpolated) empirical cumulative distribution function of $d_{t_n}({\bf u},{\bf v})/n^{\beta}$ for $\beta$ being respectively 0.74, 0.75 and 0.76. The 4 curves are so close that they are almost indistinguishable (they are of course far from each other for other exponents) }
\end{figure}

  \begin{rem} Given the results of the estimates and the similarity of empirical cumulative functions of $d_{t_n}({\bf u},{\bf v})/n^{\beta}$ as presented in Fig. \ref{fig:dqdqtf} it is tempting to conjecture that $\beta=3/4$ (but notice that in Rensburg \& Rechnitzer \cite{RR} estimates $\beta$ to be 0.74000. For $\alpha$, we thought that it could be $2/3$ and we used a lot of computer work to produce large trees (of size $8100$) to test this, but finally larger sizes did not change much the outcome and it seems that $\alpha$ should be smaller than $2/3$ (again in \cite{RR}, $\alpha$ is estimated to be 0.6437..
    \end{rem}
  
\paragraph{Degree statistics}

For a sample from $X=(X_1,\cdots,X_n)$ denote by $m(X)$ the empirical mean and sample variance: $m(X)=(X_1+\cdots+X_n)/n$, $s^2(X)= (\sum_{i=1}^n (X_i-m(X))^2)/(n-1)$.
\ben
\begin{array}{|c|c|c|c|c|}
	\hline
	\textrm{Nb of nodes}  & 1000 & 2500 & 5000 & 8100\pass\hline
	m(q_1) & 0.25918 & 0.25858 & 0.25819 & 0.25807\pass\hline
	s(q_1) & 7.416E-05 & 3.155E-05 & 1.476E-05 & 9.092E-06\pass\hline
	m(q_2) & 0.50543 & 0.50550 & 0.50585 & 0.50594\pass\hline
	s(q_2) & 2.507E-04 & 1.072E-04 & 5.067E-05 & 3.085E-05\pass\hline
	m(q_3) & 0.21361 & 0.21408 & 0.21412 & 0.21414\pass\hline
	s(q_3) & 8.471E-05 & 3.474E-05 & 1.752E-05 & 1.036E-05\pass\hline
	m(q_4) & 0.02179 & 0.02185 & 0.02184 & 0.02185\pass\hline
	s(q_4) & 1.882E-05 & 7.398E-06 & 3.705E-06 & 2.260E-06\pass\hline
\end{array}
\een

Observe that the standard deviation is small and seems to go fast to zero.

\section{Relaxation of the subtree sizes:  Transition matrices on $\TTT{G}$}

\label{sec:SST}
Here we will study some Markov chains having some explicit invariant distributions typically with support $\TTT{G}=\cup_n \TT{G}{n}$ (recall Section  \ref{sec:rs}). In \Cref{subtreetree} we will turn our attention to the case where $G$ is itself a tree, in which case a coupling from the past is possible.

\subsection{Mechanisms}

We define two versions of the functions $\Rem$ aiming at removing an edge $e$ of a tree $t$ depending on whether we are dealing with rooted trees or not. 
For an oriented edge $\ar{e}$, we denote by $e$ its unoriented version.  

\noindent{\bf Unrooted version of the $\Rem$ function:}
\[\app{\Rem}{\TTT{G}\times \ar{E}}{\TTT{G}}{(t,\ar{e})}{t'=\Rem(t,\ar{e})}\]
$\bullet$ if $E(t)=\{e\}$ and $\ar{e}=(v_1,v_2)$ then $t' = \{v_1\}$, the tree reduced to the single node $\{v_1\}$,\\
  $\bullet$ else (if $|E(t)|>1$), if $E(t)\setminus\{e\}$ is the edge set of a tree $t^\star$, set $t'=t^{\star}$,\\
  $\bullet$ otherwise, $t'=t$.\\
We stress the fact that the direction of $\ar{e}$ is used only when $t$ has a single edge.\\
\medbreak
\noindent{\bf Rooted version $\Rem_r$:} it aims at removing an edge in a rooted tree $(t,r)\in\TTTr{r}{G}$, while preserving $r$. Here, since the tree is rooted at $r$, $r$ is never considered as a leaf.
If $(t,r)\in \TTTr{r}{G}$ and $e=\{e_1,e_2\}\in E(t),$ then up to renaming the vertices, one may suppose that $e_1$ is the parent of $e_2$ in $(t,r)$ (is closer to $r$):\\
-- if $e_2$ is not a leaf, then do nothing, and set $t'=t$,\\
  -- if $e_2$ is a leaf, then $t'$ is the tree with vertex set $V(t')=V(t)\setminus \{e_2\}$ and edge set $E(t')=E(t)\setminus\{e\}$ (so that the root $r$ is preserved).
\\
Define the function $\Add$ as
\[\app{\Add}{\TTT{G}\times E}{\GGG{G}}{(t,e)}{g=\Add(t,e)}\]
where the graph $g$ has set of edges $E(g)=E(t) \cup\{e\}$ if $e$ is adjacent to $t$, and $g=t$ otherwise. Hence $g$ is connected and may have at most one cycle, and in this case, this cycle contains $e$.\medbreak

When the function $\Add$ has been used, a correction of the obtained graph is sometimes needed if one needs to output a tree (the cycle has to be destroyed as done for the kernel $K^{\Kerref{labeldd}}$).
  
\subsection{Two ergodic Markov chains with computable invariant distribution on $\TTT{G}$}
\label{sec:twoergo}
We present here a  reminiscent of the discrete time birth and death process, which is a general model of Markov chains $(Y_j,j\geq 0)$ taking their values in $\mathbb{N}$, and whose transition matrices are parameterized by a sequence of triplets $[(a_k,b_k,c_k),k\geq 0]$ as follows:
\[ ~\P(X_1= k+1 ~| X_0=k)=a_k,~~\P(X_1= k ~| X_0=k)=b_k,~~ \P(X_1= k-1 ~| X_0=k)=c_k,~\] with  $c_0=0$. It is known (see e.g. Karlin \& McGregor \cite{KmG,KmG2}, or \cite{LF-JFM3}), that such a chain is positive recurrent if $\sum_k \prod_{j=1}^k \frac{a_{j-1}}{c_j}<+\infty$ in which case the invariant distribution is proportional to $\pi_k=\prod_{j=1}^k \frac{a_{j-1}}{c_j}$.
~ \\
Consider a sequence of triplets  $[(\pp_i,\qq_i,\rr_i), 1\leq i \leq |V|]$, indexed by the possible subtree sizes of $G=(V,E)$, which will be used to try to ``add'', ``do nothing'' and ``remove'' one edge of the current tree. As above, for all $i$, $ \pp_i+\qq_i+\rr_i =1$. 
  For the moment we  assume that 
 \be\bpar{ccl}
 \rr_{i}&>&0,\textrm{  for all }i\in \cro{2,|V|},\label{eq:14}\\
  \pp_{i} &>& 0,\textrm{  for all }i\in \cro{1,|V|-1}.\label{eq:15}
   \epar\ee
We will need to consider a ``cycle breaking strategy'' as introduced in the definition of $K^{\Kerref{labeldd}}$ in Section \ref{sec:RV} (recall the definition of $p_c$ in \eref{eq:PC}).
\NewKernel{}{labelf}{Assume $X_0=t\in\TTT{G}$ (with any size). To define $X_1\sim K^{\Kerref{labelf}}(t,.)$, proceed as follows. Pick independently, a random oriented edge $\ar{\bf e} \sim \uniform(\ar{E}(G))$, and ``a random choice ${\bf c}$'' where
  \[\P({\bf c}=+1)= \pp_{|t|},~~ \P({\bf c}=0)=\qq_{|t|},~~ \P({\bf c}=-1)=\rr_{|t|},\]
  which will be the respective probability to ``try'' to add ${\bf e}$, to do nothing, and to remove $\ar{\bf e}$. Do\\ 
 $\bullet$ if ${\bf c}=+1$ then ``try to add {\bf e}'': consider  $g=\Add(t,{\bf e})$. If $g$ is a tree, set $X_1=g$. If $g$ has a cycle $c$, then pick an edge $e$ with probability $p_c$. Define $X_1$ as the tree obtained by the addition of ${\bf e}$ to $t$ followed by the removal of the edge $e$. \\
 $\bullet$ if ${\bf c}=0$, do nothing, and set $X_1=t$,\\ 
 $\bullet$ if ${\bf c}=-1$, then  ``try to remove   $\ar{\bf e}$'': set $X_1 = \Rem(t,\ar{\bf e})$.}
 {\Analysis:}
  $K^{\Kerref{labelf}}$ is  aperiodic and irreducible.
If $t'$ and $t$ have the same number of edges and $t'\neq t$, then, one can pass from $t$ to $t'$ by picking first ${\bf c}=+1$, followed by a transition which is, conditional to this value, the same as for  $K^{\Kerref{labeldd}}$. The proof of $\P( X_1=t'~|~X_0=t)= \P( X_1=t~|~X_0=t')$ for two trees $t$ and $t'$ of the same size is then the same as that of the reversibility of the kernel $K^{\Kerref{labeldd}}$ (the proof is given below the description of  $K^{\Kerref{labeldd}}$).
\par
 Consider $t\in\TTT{G}$ such that $3\leq \cv{t}<|V|$ and suppose that $e\in E(t)$ such that one endpoint of $e$ is a leaf in $t$. Therefore, the transition matrix satisfies
\ben\label{eq:dqfqd1}
K^{\Kerref{labelf}}_{t,t\setminus\{e\}} & = (1/|E|)\,\rr_{\cv{t}},\quad\quad
K^{\Kerref{labelf}}_{t\setminus\{e\},t}  = (1/|E|)\,\pp_{\cv{t}-1}
\een
and again the case $|t|=2$ provides a slight complication, in which case,
\ben\label{eq:dqfqd2}
K^{\Kerref{labelf}}_{t,t\setminus\{e\}} & = (1/(2|E|))\,\rr_{\cv{t}},\quad\quad
K^{\Kerref{labelf}}_{t\setminus\{e\},t}  = (1/|E|)\,\pp_{\cv{t}-1}.
\een
\begin{pro}\label{pro:qsdz} The Markov chain with kernel $K^{\Kerref{labelf}}$ is reversible and its unique invariant measure $\rho$ on $\TTT{G}$ gives the same weight $\nu_n := \nu_n(G)$ to each element of $\TT{G}{n}$, for all $1\leq n \leq |V|$, that is $\rho_t= \nu_{\cv{t}}$, for all $t\in \TTT{G}$. The sequence  $(\nu_k,1\leq k \leq |V|)$ satisfies
\ben\label{inv:eq}
\nu_{m} =2\nu_{1}\,\prod_{i=2}^{m}\left( \frac{\pp_{i-1}}{\rr_{i}} \right),~~\textrm{ for all } m \in\cro{2,|V|}.
\een
and
\ben\label{eq:sffq22}
\sum_{n=1}^{|V|}\, \nu_n |\TT{G}{n}|=1.
\een
Hence, if ${\bf t}\sim \rho$, ${\cal L}( \bt ~|~ \cv{\bt}=n)$ is the uniform distribution on $\TT{G}{n}$.
\end{pro}
\begin{rem}\label{eq:nu} In the Proposition, the sequence $(\nu_i)$ depends on $G$, and then, it should have been written $(\nu_i(G))$ to make this dependence clearer.
  \end{rem}

\begin{proof} First, by Perron-Frobeniüs, there is a unique invariant measure. Therefore, it is enough to show that the only measure $\rho$ on $\TTT{G}$, described in the proposition, satisfies the detailed balance equations \eref{reveqs}.    For $t$ and $t'=t\setminus\{e\}$ and $\cv{t}\geq 3$,
  \ben\label{eq:dqsdqr}
  \nu_{\cv{t}}\, K^{\Kerref{labelf}}(t,t\setminus\{e\})=  \nu_{\cv{t\setminus\{e\}}}\, K^{\Kerref{labelf}}(t\setminus\{e\},t).\een
  From \eref{eq:dqfqd1} one sees that $\nu_{\cv{t}}=\nu_{\cv{t}-1} \pp_{\cv{t}-1}/\rr_t$ when $\cv{t}\geq 3$. Plugging \eref{eq:dqfqd2} in \eref{eq:dqsdqr}, in the case where $\cv{t}=2$, gives: 
  \[\nu_2  (1/(2|E|))\,\rr_{2} = \nu_1 (1/|E|) \pp_1 \equi \nu_2 = 2\nu_1 \frac{\pp_1}{\rr_2}. \]
\end{proof}
\begin{rem}
  \bls Tuning the sequence $(\pp,\qq,\rr)$ allows one to favour a tree size, or an approximate tree size.\\
  \bls If $\qq_i=0$, $\rr_i=\pp_i=1/2$ for all $i$, then
  $\nu_{|t|}= \frac{1}{1+\1_{\cv{t}=1}}$ so that the distribution is uniform on  $\TTT{G}$ (except for the tree reduced to a single node that has a different weight).
\end{rem}
\paragraph{A variant with a fixed root.}
One can turn $K^{\Kerref{labelf}}$ into a kernel $K^{\Kerref{labelf}}_r$ of a Markov chain taking its values in $\TTTr{r}{G}$ where $r$ is a fixed vertex of $V$. This version will play an important role for the exact sampling of a uniform subtree of a tree in Section \ref{sec:exact_samp}.
\par We define  $K^{\Kerref{labelf}}_r$ by emphasizing its differences with $K^{\Kerref{labelf}}$: to preserve $r$, use $\Rem_r$ instead of $\Rem$, and instead of taking directed edges $\ar{\bf e}$ in $\ar{E}(G)$, we consider the unoriented ones ${\bf e}$ in $E(G)$. In this case, one can prove the following proposition by adapting the proof of \Cref{pro:qsdz}.
\begin{pro}\label{pro:qsdz2}
The Markov chain with kernel $K^{\Kerref{labelf}}_r$ is reversible and its unique invariant measure $\rhor{r}$ on $\TTTr{r}{G}$ gives the same weight $\nu_n$ to each element of $\TTr{r}{G}{n}$, for all $1\leq n \leq |V|$, that is $\rho_t= \nu_{\cv{t}}$, for all $t\in \TTT{G}$. The sequence  $(\nu_k,1\leq k \leq |V|)$ satisfies
\ben\label{inv:eq2}
\nu_{m} =\nu_{1}\prod_{i=2}^{m}\left( \frac{\pp_{i-1}}{\rr_{i}} \right),~~\textrm{ for all } m \in\cro{2,|V|}.
\een
and
\ben\label{eq:sffqtt}
\sum_{n=1}^{|V|}\, \nu_n |\TTr{r}{G}{n}|=1.
\een
Hence, if ${\bf t}\sim \rho$, ${\cal L}( \bt ~|~ \cv{\bt}=n)$ is the uniform distribution on $\TTr{r}{G}{n}$.
\end{pro}
Compared to \eref{inv:eq}, in \eref{inv:eq2} the factor 2 has been suppressed.

\subsection{A fast kernel with computable invariant distribution for regular graphs}

We propose in this part a kernel having a computable invariant distribution when all the vertices of $G=(V,E)$ have the same degree $D$.
This kernel is almost the same as the previous one ($K^{\Kerref{labelf}}$), its analysis is the same, but it mixes much  faster: the idea is to pick edges adjacent to the current tree, instead of uniform edges in $E(G)$.\medbreak
 
\NewKernel{A fast kernel for regular graphs}{labelg}{Keep the same definition as for the kernel $K^{\Kerref{labelf}}$, except for the choice of the random edge $\ar{\bf e}$, do the following instead.  Assume that  $X_0=t$, pick uniformly at random node $\mathbf{u}$ in $V(t)$, and then a random edge $\ar{\bf e}=(\mathbf{u},u')$ uniformly in the set of adjacent edges of $\mathbf{u}$ (so that $\mathbf{u}$ is the origin of this edge). }
\Analysis: Transition between trees with the same size is done as in $K^{\Kerref{labelf}}$. And it is direct to check that for any $t$ such that $|t|\geq 3$, and $e$ an edge such that $t\setminus \{e\}$ is a tree (with one node less)
\be
K^{\Kerref{labelg}}_{t,t\setminus \{e\}} &=& \frac{1}{\cv{t}}\l(\frac{1}{\degree_G(\mathbf{u})}+\frac{1}{\degree_G(u')}\r) \rr_{\cv{t}}\\
K^{\Kerref{labelg}}_{t\setminus \{e\},t} &=& \frac{1}{(\cv{t}-1)} \frac{1}{\degree_G(\mathbf{u})} \pp_{\cv{t}-1}
\ee
again if $\cv{t}=2$, in this case if $t'=\Rem(t,(\mathbf{u},u'))$ is the tree $t'$ reduced to $\mathbf{u}$, so that
\be
K^{\Kerref{labelg}}_{t,t'} &=& \frac{1}{\cv{t}\,\degree_G(\mathbf{u})}\rr_{\cv{t}}=\,\frac{\rr_{|2|}}{2\,\degree_G(\mathbf{u})}\, \\
K^{\Kerref{labelg}}_{t',t} &=& \frac{1}{\cv{t'}\,\degree_G(\mathbf{u})}\, \pp_{\cv{t'}}=\frac{ \pp_1}{\degree_G(\mathbf{u})},
\ee
since, in this transition the directed edge $(\mathbf{u},u')$ needs to have the right direction.
\begin{pro}\label{pro:qfter}
  If the degree of all nodes in $G$ is the same, then the Markov chain with transition matrix $K^{\Kerref{labelg}}$ is reversible and its unique invariant measure $\rho$ on $\TTT{G}$ gives the same weight $\nu_n$ to each element of $\TT{G}{n}$, for all $1\leq n \leq |V|$, that is $\rho_t= \nu_{|t|}$, for all $t\in \TTT{G}$. The sequence  $(\nu_k,1\leq k \leq |V|)$ satisfies
\ben\label{inv:eqqqq}
\nu_m  =  2\nu_{1} \prod_{i=2}^{m}\left( \frac{ \pp_{i-1}/(i-1)}{2\,\rr_{i}/i} \right),~~ \textrm{ for } 2\leq m \leq |V|\een
and
\ben\label{eq:sffq}
\sum_{n=1}^{|V|}\, \nu_n |\TT{G}{n}|=1.
\een
Hence, if ${\bf t}\sim \pi$, ${\cal L}( \bt ~|~ |\bt|=n)$ is the uniform distribution on $\TT{G}{n}$.
\end{pro}

\begin{rem} \label{rem:rel} Recall that the transition matrices $K^{\Kerref{labelf}}$ and $K^{\Kerref{labelg}}$ are defined using $[(\pp_i,\qq_i,\rr_i), 1\leq i \leq |V|]$.  The conditions $(\pp_i>0,1\leq i<|V|)$ and $(\rr_{i}>0, 2\leq i \leq |V|)$ are imposed so that they ensure the irreducibility of these chains on $\TTT{G}$.
  Now, assume that one takes $X_0$ according to some distribution $\nu$ with support in $\cup_{n \in \cro{n_1,n_2}} \TT{G}{n}$ where $1 \leq n_1 < n_2 \leq |V|$. Assume that $\rr_{n_1}=0$ and $\pp_{n_2}=0$, and $\rr_{k}>0$ for $k\in \cro{n_1+1,n_2}$,   $\pp_{k}>0$ for $k\in \cro{n_1,n_2-1}$. In this case, the Markov chain under consideration is irreducible in  $\cup_{n \in \cro{n_1,n_2}} \TT{G}{n}$ (exercise left to the reader). In this case we have the same result for the distribution of the invariant measure as in Proposition \ref{pro:qfter} between $1$ and $n_2$ (instead of $|V|$) when $n_1=1$, and if $n_1>1$, the invariant distribution is given by
\ben\label{inv:eqdqsd}
\nu_m  = \nu_{n_1} \prod_{i=2}^{m}\left( \frac{\pp_{i-1}/(i-1)}{ 2\;\rr_{i}/i} \right)\textrm{ for }m \in\cro{n_1+1,n_2}.
\een
When $n_1<n_2$, the irreducibility of the chain 
and \eref{eq:sffq} is easily adapted to the present case. \\
If $n_1=n_2$, then one can see that the vertex set $V(X_0)$ of the initial tree $X_0$ cannot change: for each $i$, $V(X_i)=V(X_0)$, so that this model is a Markov chain taking its value in the spanning trees of $V(X_0)$ (this setting is treated in Section \ref{sec:STC}).
\end{rem}
 
\color{black}

\section{Survey of models of random subtrees of a graph}
\label{sec:survey}

In this section, we present many distributions (with simulations methods) far from the uniform distribution, but which are interesting on their own (and, marginally, can  be used to design simulation of the uniform distribution by reject for small graphs, or small values of $n$).

\subsection{The pioneer tree}
\label{sec:PT}

We introduce the pioneer tree which is a new random tree model. 
Recall the definition of $\FET(W_0,\cdots,W_{\tau_{|V|}})$ given in \eref{defi:dqsd}. The pioneer tree aims to generalize Aldous--Broder construction: instead of taking all the first entrance edges to all nodes (for a $M$-Markov chain under its stationary regime), which provides a tree with weight $\prod_{e} \ra{M}_e$ as stated in Theorem \ref{the:kgyqsd}, just keep the $n$ first ones. We take the same setting as in Section \ref{sec:ABA}: $G$ is a connected graph, $M$ a positive Markov transition matrix on $G$, and $W$ is a $M$-Markov chain (we drop the condition of reversibility). 

The aim of this section is to present this model, and to show that it shares, as the uniform spanning-tree model does, a strong link with a tree valued Markov chain.\par

\NewModel{The pioneer random tree}{PRT}
{The $n$ pioneer tree ${\sf PRT}_n(W_i, 0\leq i \leq \tau_n$) is the rooted edge-labelled tree $(\FET(W_i, 0\leq i \leq \tau_n), L_n)$, where $L_n$ gives the label $k-1$ to the edge $(W_{\tau_k},W_{-1+\tau_k})$, for all $2\leq k \leq n$.}
Hence, the vertex set of ${\sf PRT}_n(W_i, 0\leq i \leq \tau_n)$ is  $\{W_0,\cdots,W_{\tau_n}\}$, the first $n$ vertices visited by $W$.
\begin{defi}
  Denote by  $\TTTrLd{r}{G,n}$ the set of rooted edge-labelled trees $((t,r),\ell)$  such that $(t,r)$ belongs to $\TTr{r}{G}{n}$, and such that the $n-1$ labels associated with the edges form the set $\{1,\cdots,n-1\}$ and are decreasing on any injective path from a leaf to the root $r$.\footnote{An injective path $w=(w_0,\cdots,w_m)$ is a path such that $i,j \in\cro{0,m}$, $i\neq j \imp w_i\neq w_j$. } 
\end{defi}
A simple consequence of the construction is the following fact:
\begin{lem}  The pioneer tree ${\sf PRT}_n(W_0,\cdots,W_{\tau_n})$ belongs to $\TTTrLd{W_0}{G,n}$ and
\ben\label{eq:edqd}
{\sf PRT}_n \subset {\sf PRT}_{n+1},~~~\textrm{ for any } 1\leq n \leq |V|-1.
\een
Hence for all $n$, ${\sf PRT}_n$ is an edge-labelled subtree of the global spanning-tree  ${\sf PRT}_{|V|}$ equipped with its edge-labels.
\end{lem}
In the same way as Aldous--Broder $\FET(W_0,\cdots,W_{\tau_{|V|}})$  can be seen as the state, at time 0, of a spanning-tree valued Markov chain started at time $-\infty$ (this is the argument at the core of Aldous and Broder proofs), for any $n$, the pioneer tree has a very similar property, for the following Markov chain taking its values in $\cup_{r\in V}\TTTrLd{r}{G,n}$: again, $n$ is any number in $\cro{1,|V|}$, so that the following construction includes the spanning-tree case, but not only.

\paragraph{A Markov chain on pioneer trees driven by a random walk: erase the oldest edge}

\NewKernel{Add a random step and erase the oldest edge}{EOE}{Assume that at time 0, $((T_0,R_0),L_0)$  is an element of $\TTTrLd{R_0}{G,n}$, whose tree $T_0$ is rooted at $R_0$. Under the kernel $K^{\Kerref{EOE}}$, $((T_1,R_1),L_1)$ is defined as follows:\\
   $\bullet$ First, $\P(R_1=v ~|~R_0=u) = \ra{M}_{u,v}$, which means that the roots $(R_k, k\geq 0)$ performs a Markov chain with transition matrix $\ra{M}$ on $G$. \\
  $\bullet$ Consider the oriented edge $e=(R_0,R_1)$ of $G$; $R_1$ will be the new root of the new tree $T_1$.
  \begin{itemize}\compact
   \item[(a)]  If $R_1=R_0$ (possible if there is a loop): in this case set $((T_1,R_1),L_1)=((T_0,R_0),L_0)$, 
  \item[(b)] If $R_1$ is already in $T_0$, then adding the edge $e=(R_0,R_1)$ in $T_0$ creates a cycle (possibly, the small cycle $R_0\to R_1 \to R_0$). To get $T_1$, add $e$ to $T_0$, label $e$ temporarily 0,  record ${\bf m}$ the {\bf maximal} label on the created cycle, and remove the edge with label ${\bf m}$; finally, orient the remaining edges of the cycle toward $R_0$
  \item[(c)] else, $R_1$ was not in $T_0$ so that if one adds the edge $e=(R_0,R_1)$ to $T_0$, then $R_1$ is a new node. To get $T_1$, add the edge $e$ to $T_0$, label $e$ temporarily 0 and remove the edge $e'$ adjacent to the leaf with {\bf maximal} label (the label ${\bf m}$ of $e'$ is  $n-1$). 
  \end{itemize}
 To define  $L_{1}$ in both cases, keep the labels of all edges of $L_0$ that are  $>{\bf m}$, and add 1 to all the other labels (those in $\cro{0,{\bf m}-1}$,  including the new one labelled temporarily 0).}

This chain is a generalization of Aldous-Broder tree Markov chain, but here, in order to keep track of the chronological order of the edges, additional labels are needed. Observe that the performed random walk is done according to the time reversal transition matrix  $\ra{M}$.
\begin{pro}\label{pro:croi1}
 \bir
  \itr The labels $L_1$ are different and decreasing on each path toward the root, and then so that $K^{\Kerref{EOE}}$ defines indeed a transition matrix on $\cup_{r\in V}\TTTrLd{r}{G,n}$.
  \itr If $(X^{(n)}_j,j\geq 0)$ is a Markov chain on $\cup_{r\in V}\TTTrLd{r}{G,n}$ with kernel $K^{\Kerref{EOE}}$, then for $X^{(n-1)}_j$ be the labelled tree obtained by removing the edge with largest label in $X^{(n)}_j
  $, the process $(X^{(n-1)}_j,j\geq 0)$ is a Markov chain on $\cup_{r\in V}\TTTrLd{r}{G,n-1}$ with kernel $K^{\Kerref{EOE}}$.
  \eir
\end{pro}

\begin{proof}[Sketch of proof]
  Giving all the details would be too long. We give the main ideas only.\\
  $(i)$ The proof is done by inspection of both cases $(b)$ and $(c)$ in the definition of $K^{\Kerref{EOE}}$.\\
  $(ii)$ Suppose that $((t_n,\ell_n),r)$ and $((t_{n+1},\ell_{n+1}),r)$ are two edge labelled trees with $n$ and $n+1$ nodes, such that $((t_n,\ell_n),r)$ is obtained from $((t_{n+1},\ell_{n+1}),r)$ by the suppression of the edge with greatest label $n$ (we write ${\sf Proj}((t_{n+1},\ell_{n+1}),r)=(t_n,\ell_n),r)$). When taking a step under the kernel $K^{\Kerref{EOE}}$, a new edge $(r,r')$ is added, $r'$ becomes the new root: this addition gives different possible situations for $t_n$ and for $t_{n+1}$: \\
  -- (A) $r'$ is not in $t_{n+1}$ (nor in $t_n$),\\
  -- (B) $r'$ is in $t_{n+1}$ but not in $t_n$.\\
  In case $(A)$,  after applying $(c)$ of  definition of $K^{\Kerref{EOE}}$, both obtained trees $((t'_n,\ell_n'),r')$ and  $((t'_{n+1},\ell_{n+1}'),r')$ satisfy  ${\sf Proj}((t_{n+1}',\ell_{n+1}'),r)=(t_n',\ell_n'),r)$).\\
  In the case $(B)$, the cycle obtained by adding $(r,r')$ to $t_{n+1}$ contains necessarily the edge with greatest label of $t_{n+1}$ (otherwise a cycle would have been created also by adding $(r,r')$ to $t_n$). From here the conclusion is simple.
\end{proof}
Analogously as Aldous-Broder tree Markov chain preserves the distribution specified in \eref{the:kgyqsd}, the Markov chain with kernel $K^{\Kerref{EOE}}$ has the property to leave the pioneer tree distribution invariant. 
\begin{pro}\label{pro:croi2}
  The Markov chain with kernel $K^{\Kerref{EOE}}$ is ergodic on  $\cup_{r\in V}\TTTrLd{r}{G,n}$, and its invariant distribution is the distribution of the pioneer ${\sf PRT}_n(W_i, 0\leq i \leq \tau_n)$ for $W_0$ following the invariant distribution $\rho$ of $M$ (with full support on $\cup_{r\in V}\TTTrLd{r}{G,n}$).
 \end{pro}
Hence, several points can be noticed: the consistency of the trees $({\sf PRT}_n,1\leq n \leq |V|)$, the fact that a labelling is needed to construct this coupling, the fact that Aldous and Broder scheme to study the $\FET$ can be applied here again using a time-reversal chain under its stationary distribution, and also the fact that, forgetting their labels, all of them are subtrees of the original Aldous--Broder spanning tree. 
\begin{proof}
\noindent   The main idea consists in introducing a time-reversal (as in Aldous and Broder argument), and a second family of trees that we call $\LET$. 

Any finite path $(z_0,\cdots,z_\tm)$ on $G$ can be used to define a rooted tree
$\LET(z_0,\cdots,z_\tm)$, rooted at $z_\tm$ as follows: first $\LET(z_0)$ is the tree reduced to its root $z_0$; from $k=0$ to ${\tm-1}$, construct $\LET(z_0,\cdots,z_{k+1})$ from  $\LET(z_0,\cdots,z_{k})$ by the suppression of the outgoing edge from $z_{k+1}$ (if any), by the addition of the edge $(z_k,z_{k+1})$ and by setting the root at $z_{k+1}$. The set of nodes of $\LET(z_0,\cdots,z_{\tm})$ is $\{z_0,\cdots,z_\tm\}$; if one denotes by
\[\nu_k=\max\{j : |\{z_j,\cdots,z_\tm\}|=k\},\]
the last time $k$ nodes remain to be visited ``in the future'', then, for any $k\in \{1,\cdots,|\{z_0,\cdots,z_\tm\}|\}$, $\nu_k$ is the date of visit of a node for the last time; hence, the tree $\LET(z_0,\cdots,z_\tm)$ has for edges
\ben
(z_{\nu_k}, z_{1+\nu_k}), \textrm{ for } k=|\{z_0,\cdots,z_\tm\}| \textrm{ to }2.
\een
In  Definition \ref{defi:dqsd}, $\FET$ is associated with a covering path; this definition can be extended to any path, covering or not.
It is immediate to check, that, for any path $(w_0,\dots, w_\tm)$ on $G$,
  \ben\label{eq:qsdqs}
  \FET(w_0,\cdots,w_\tm)=\LET(w_\tm,\cdots,w_0).
  \een

\begin{figure}[h!]
  \centerline{\includegraphics[width = 8cm]{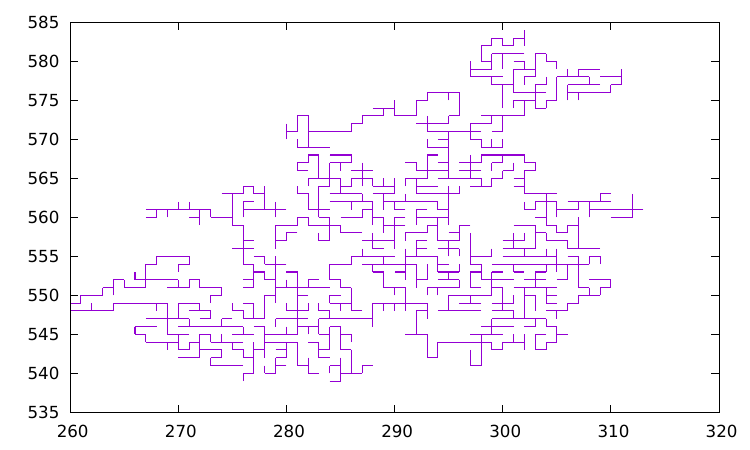}\includegraphics[width = 8cm]{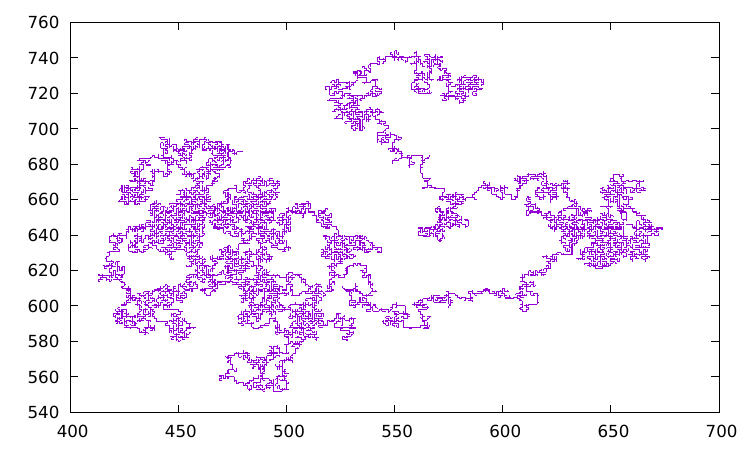}}
  \captionn{ \label{FIG:simuRO} Simulation of  $K^{\Kerref{EOE}}$ on $\Torus{1000}$, of a tree with 1000 and then 10000 edges (in the first case, 25M steps are done starting from a rectangle tree 40$\times 25$, in the second case 200M steps  starting from a rectangle tree 100$\times 100$.}
 \end{figure}

Assume now that $(X_k,k\in Z)$ is a $M$ Markov chain and $(Y_k, k \in \mathbb{Z})$ is a $\ra{M}$ Markov chain, both of them taken under their invariant distribution. \par
We start with the spanning tree case. There are three main ideas:\\
\bls \textbf{Construction of $\LET$ following the ``erase the oldest'' dynamic.}\par $k\mapsto \LET(Y_i,i\leq k)$ is a Markov process such that from time $k$ to $k+1$ a new edge $(Y_k, Y_{k+1})$ is added, and the outgoing edge $e$ from $Y_{k+1}$, if any, is suppressed; and in such a case before suppression, the addition of $(Y_k,Y_{k+1})$ created a cycle $C$. By induction on $k$ one can prove that the edge creation timestamps give an increasing labelling on any injective path to the root. We claim that the edge $e$ was the ``oldest'' edge of $C$. This statement is meaningful since the date of creation of each edge is $\sigma(Y_i, i \leq k)$ measurable: each edge is the last exit edge to a node. Therefore, the further from the root is an edge on the $\LET$, the smaller creation timestamp it has and therefore the older it is. Hence, the edge $(Y_k,Y_{k+1})$ creates a cycle with a path going to $Y_k$, which is then a branch in the tree, so that the outgoing edge from $Y_{k+1}$ is indeed the oldest in the cycle.
Hence, up to the labels, the tree in the ``erase the oldest edge'' dynamics is the same as  $k\mapsto \LET(Y_i,i\leq k)$.\\
\bls \textbf{Adding the ``right'' labels to the analysis.}\par Label the edges of $\LET(Y_i, i\leq k)$ by $\ell_k$ the relative order in $\cro{1,|\{Y_0,\dots, Y_k\}|-1}$ of their creation timestamps as in the preceding part, this is an increasing labelling on any injective path towards the root. We produce a reverse labelling $\ell_k^\downarrow$ of $\ell_k$ as follows\\
\[\ell_k^{\downarrow}(e)=|\{Y_0,\dots, Y_k\}|-j,~~ \textrm{ for } j \in \cro{1,|\{Y_0,\dots, Y_k\}|-1}.\]
Under $\ell_k^\downarrow$, the bigger the label, the smaller its timestamp is and therefore the older the edge.\\
Now, in the spanning tree case the chain ``erase the oldest chain'' and $k\mapsto (\LET(Y_i, i\leq k),\ell_k^{\downarrow})$ (from $k$ large enough) can be identified under their stationary regime (this can be seen more easily by the time-reversal argument that follows).\\
\bls \textbf{time-reversal application to obtain {\sf pioneer} from $\LET$ + labels: }\par
The combinatorial property \eref{eq:qsdqs} allows one to see that if $X$ is a $M$ Markov chain and $Y$ a $\ra{M}$ Markov chain under their common stationary distribution $\rho$
\[{\sf PRT}(X_i, 0\leq i \leq \tau_{V})\eqd (\LET(Y_i, i\leq 0),\ell_0^{\downarrow}).\]
To complete the proof for $n<|V|$, it suffices to use \eref{eq:edqd} and its counterpart for $(\LET(Y_i, i\leq 0),\ell_0^{\downarrow})$: in words, keeping from the spanning tree process the $n-1$ edges with the smallest labels, provides on the left-hand ${\sf PRT}(X_i, 0\leq i \leq \tau_{n})$, and in the right hand the tree $(\LET(Y_i, i\leq 0),\ell_0^{\downarrow})_{e: \ell_0^{\downarrow}(e)<n}$ restricted to the $n-1$ edges with smallest labels; the coupling argument given also allows one to compare the process $k\mapsto(\LET(Y_i, i\leq 0),\ell_0^{\downarrow})_{e: \ell_k(e)<n}$ with the ``erase the oldest edge'' chain is still valid.
\end{proof}
The distribution of the vertices of the tree $\{W_0,\cdots,W_{\tau_n}\}$ is the track of the Markov chain till it visits $n$ different points. It is possible to give some combinatorial formulas for the distribution of this support, but they are not enlightening. For the asymptotics on some graphs (as on $\Z^2$ or $\Torus{N}$), Brownian limit of $(W_k,k\geq 0)$ suitably normalized shows that from a probabilistic perspective, the question is the following.
\begin{Ques} Describe the distribution of $\FET(W_0,\cdots,W_{\tau_n})$ conditionally on vertex set $\{W_0,\cdots,W_{\tau_n}\}$.
  \end{Ques}
For more information on the combinatorics behind this model, we send the reader to \cite{LFJFM}.

\subsubsection{Erase the youngest edge, a degenerate variant of the kernel $K^{\Kerref{EOE}}$}
\label{sec:hdgfds}

It seems natural to ask if erasing the youngest edge gives an exploitable model, to define this mechanism just replace {\bf maximal} by {\bf minimal} in the description of the ``erase the oldest edge transition matrix'' $K^{\Kerref{EOE}}$. This process tends to destroy almost all leaves and to provide a poor model of random trees, even if, as a model of weakly branching ``self avoiding random walk'', it could be thrilling to study (see Fig. \ref{FIG:simuRY}).
\begin{figure}[htbp]
  \centerline{\includegraphics[width = 8cm]{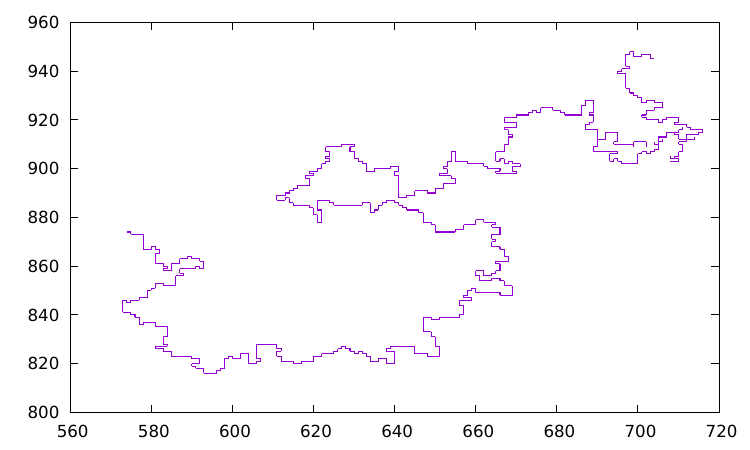}\includegraphics[width = 8cm]{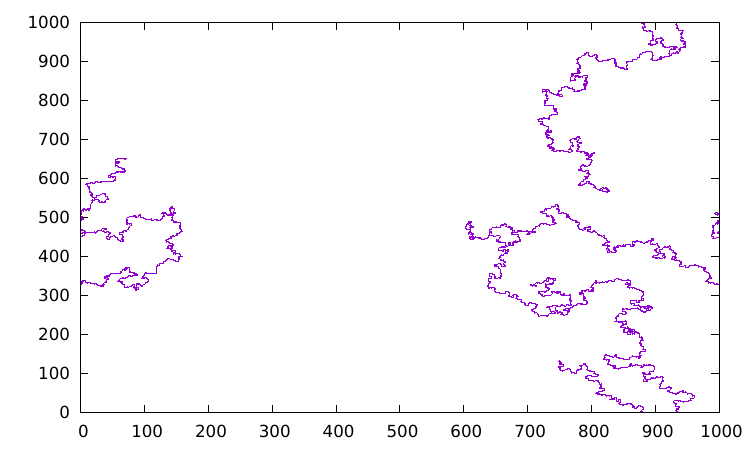}}
  \captionn{ \label{FIG:simuRY} Remove youngest edge, simulation on $\Torus{1000}$, of a tree with 1000 and then 10000 edges (in the first case, 1M steps starting from a rectangle tree 40$\times 25$, in the second case 10M steps  starting from a rectangle tree 100$\times 100$.}
 \end{figure}

\subsection{No ``local construction'' of a uniform element of $\TT{G}{n}$}\label{sec:qdegppzr}

In this section, we present a generic argument allowing one to prove that it is not possible to sample a uniform element of $\TT{G}{n}$ using few steps a random walk, when $n$ is small compared to $|V|$.
This argument can be used to reject many constructions one may imagine.
\begin{theo}\label{theo:nolocal} Consider a simple random walk $W=(W_k,k \in  \mathbb{Z})$ on a graph $G=(V,E)$ under its invariant distribution (meaning that knowing $W_i$, $W_{i+1}$ is uniform among the neighbours of $W_i$).   Denote by $\ar{\tau_n}:=\inf\{k\geq 0: \#\{W_0,\cdots,W_k\}=n\}$ the first time the random walk visits $n$ points, or, ``the same thing'', backward,  $\ra{\tau_n}:=\max\{k \leq 0: \#\{W_{k},\cdots,W_0\}=n\}$.   
In general, there does not exist any map $F$ taking its values on the set of trees with $n$ nodes, such that
  $F(W_0,\cdots,W_{\ar{\tau_n}})$ is uniform on $\TT{n}{G}$ or on $\TTr{r}{n}{G}$ (with $r$ random or not), and such that the vertex set of $F(W_0,\cdots,W_{\ar{\tau_n}})$ is included in $\{W_0,\cdots,W_{\ar{\tau_n}}\}$. The same statement holds for $F(W_{\ra{\tau_n}},\cdots,W_0)$ instead.
\end{theo}
  \begin{rem}
    \bir
    \itr The ``In general'' in the statement is important.  Aldous--Broder theorem asserts that when $n=|V|$ the map $F$ exists: it is $\FET$! The proof of \Cref{theo:nolocal} consists in exhibiting a family of graphs on which, for $n$ small compared to $|V|$, it is not possible to extract from  $(W_0,\cdots,W_{a_n})$ a uniform  element of $\TT{n}{G}$, even for $a_n$ large compared to $\tau_n$, as long as $a_n$ is negligible in front of $|V|$
    \itr The hypothesis that the vertex set of the resulting tree is included in the trace of $\{W_0,\cdots,W_{\ar{\tau_n}}\}$ is needed since, without this condition, the randomness of the trajectory could be used in a ``non-natural way'' to sample a uniform element of $\TT{n}{G}$. 
    
    For example, a path with size $k$ on ${\sf Torus}(n)$ provides a uniform random word of $\{0,1,2,3\}^k$ (the possible directions of each step numbered from 0 to 3), and this word can be used to sample in a set with a smaller size (using reject, if needed), for example in  $\TT{j}{{\sf Torus}(n)}$ for any $j$ such that $|\TT{j}{{\sf Torus}(n)}|\leq 4^k$ (an algorithm which would associate a tree to a word would be needed). However, the produced tree would be far to be included in the track of the chain. This is what we want to avoid here. 
\eir
\end{rem}
\begin{proof}
The main idea is the following: a simple random walk has a simple stationary distribution $\rho$ which is $\rho_u =\degree(u)/\sum_{v\in V} \degree(v)$. Hence, a simple random walk taken under its invariant distribution, is localized in a graph ``proportionally to the degree of the starting node''.
The probability that a uniform tree in $\TT{G}{n}$ has vertex set $V'\subset V$ is proportional to the number of spanning trees in ${\sf Induced_G}(V')$, which roughly, can be thought to depend on the product of the nodes degree in ${\sf Induced_G}(V')$ rather than their sums.  Hence, the distribution of the support $V'$ has somehow nothing to do with $\rho$. \par
For the non convinced reader, let us take an example of graph in which this phenomenon is evident. Take the graph on the set of vertices $\{1,..,n^3\}$ whose edges are described by the fact, that the graph induced by $\{1,\cdots,n\}$ is the complete graph $K_n$, and the vertices $(n,n+1,...,n^3)$ forms a path (going from vertex $n$ to $n+1$ to $\cdots$ to  $n^3$).

The invariant distribution $\rho$ of the simple random walk on this graph is $\rho_i=(n-1)c_n$ for each vertex $i \in\{1,\cdots,n-1\}$, $\rho_n=nc_n$, for $i\in \{n+1,\cdots,n^3-1\}$, $\rho_i=2c_n$ and $\rho_{n^3}=c_n$, for $c_n=1/(2n^3+n^2-3n)$. Hence, the starting point of the random walk will be in $\{2n+1,\cdots, n^3\}$ with probability close to 1, so that a random walk stopped when it touched $n$ points, starting under this invariant distribution will see only the vertices of the path with probability going to 1 when $n\to\infty$.

But, the total number of spanning trees of the graph $K_n$ is $n^{n-1}$ (all of these trees have size $n$) which is far greater than the number of size $n$ subtrees of the path which is $O(n^3)$.
\end{proof}

\subsection{A model inspired by Wilson's algorithm}
\label{Ver:notfull}

\NewModel{The connected component of $r$ in the model with one outgoing edge per node in $V\setminus r$}{ZM}
{Let  $r\in V$ be a distinguished vertex; consider $(\bfe_u, u \in V \setminus\{r\})$ a family of independent random directed edges, where  $\bfe_u=(u,\bf{u'})$ and $\bf{u'}$  is a uniform neighbour of $u$. Denote by ${\bf t}(r)$ the connected component of $r$: it is a tree rooted at $r$.}

For a general graph $G$, the support of the distribution of ${\bf t}(r)$ is (included but) different of $\TTTr{r}{G}$. For example, if $G=\Torus{n}$, each connected component of the complement of ${\bf t}(r)$ contains oriented cycles, and then, these components cannot be reduced to a single vertex (see a simulation in Fig. \ref{fig:qsdqs}).

Given a tree $t$, recall $V_p(t)=\{ w\in V: d_t(w,V(t))=1\}$ the set of perimeter sites of $t$. For each $w\in V_p(t)$, let $p_t(w)= |\{ (w,u) \in E, u \not\in t\}|/\degree_G(w)$ the probability that the outgoing edge from $w$ does not touch $t$.
    For any $t\in \TTTr{r}{G}$
\ben\label{eq:qqfegsdqs}
\P({\bf t}(r)=t)=\Big(\prod_{v\in t\atop{v\neq r}}\frac{1}{\degree_G(v)} \Big)\Big(\prod_{w\in V_p(t)}p_t(w)\Big).
\een

\begin{figure}[h!]
	\centerline{\includegraphics[width=7cm]{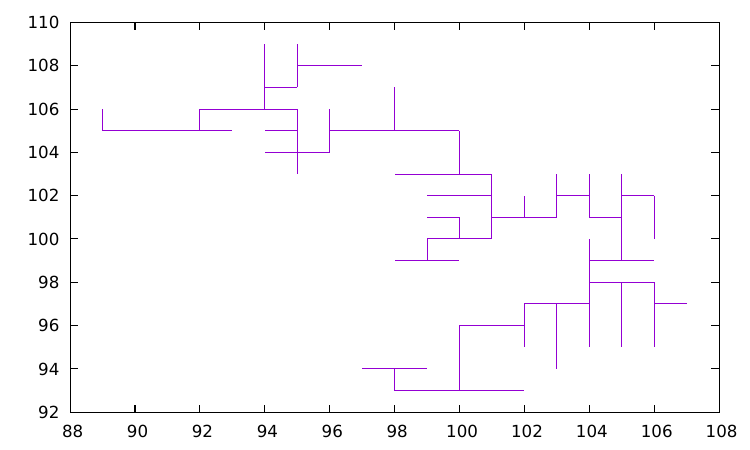}}
	\captionn{\label{fig:qsdqs} Simulation  of \textbf{Model \Modref{ZM}} on $\Torus{200}$,  3536949 simulations were needed to get a tree of size at least 100. In fact, by chance, the tree had exactly size 100. It seems that a mean of around 5 millions simulations are needed to get this size at least. Simulating big trees by this method seems out of reach.)}
      \end{figure}
      
\noindent To get a model having full support in  $\TTTr{r}{G}$, it suffices to modify a bit the model to allow nodes to have zero outgoing edge:

\NewModel{At most one outgoing edge per node}{AM}
{Take a parameter $q\in(0,1]$ and consider a collection  $\l[B_v(q),v \in V\setminus\{r\}\r]$ of i.i.d. Bernoulli$(q)$ random variables to label the vertices. Consider for each vertex $u$ in $V\setminus \{r\}$ with $B_u(q)=1$ a uniform random outgoing edge $\bfe_u=(u,\bf{u'})$, independent of the others (defined as in \textbf{Model \Modref{ZM}}).  Again take ${\bf t}_q(r)$ the connected component of $r$.}

It is simple to see that for $t\in \TTTr{r}{G}$,
\be
\P({\bf t}_q(r)=t)=\Big(\prod_{v\in t\atop{v\neq r}}\frac{q}{\degree_G(v)} \Big)\Big(\prod_{w\in V_p(t)}(1-q)+qp_t(w)\Big).
\ee
When $q=1$ we recover the model \eref{eq:qqfegsdqs} above, but for $q\in(0,1)$, the support of the random variable ${\bf t}_q(r)$ is the complete set $\TTTr{r}{G}$.
However, in practice, on $\Torus{n}$, this model produces  very small trees, even smaller than in the case $q=1$ for which getting a large tree is rare (see Fig. \ref{fig:qsdqs}).

\subsection{Subtree of a size biased forest}
\label{sec:biased}
Recall the definition of forest given in Section \ref{sec:enum}. In the literature, the term ``spanning forest'' is often used to denote a collection of trees $(t_1,\cdots,t_k)$, each of them spanning a connected component $(c_1,\cdots,c_k)$ of a graph (having $k$ connected components). Here, the underlying graph $G=(V,E)$ is connected, and we call spanning forest, a subgraph of $G$ with no cycle, spanning $V$, or equivalently, a collection of subtrees of $G$ whose vertex sets form a partition of $V$. Any total order on $V$ can be used to order the trees $t_1,\cdots,t_k$ in any spanning forest, for example, by sorting the subtrees according to their least vertex; denote by $t <t'$ the corresponding order between disjoint trees; in the sequel the set of spanning forest
  \[{\sf Spanning Forests(G)}:=\cup_k \{(t_1,\cdots,t_k) \textrm{ spanning forest }, t_1 < \cdots <t_{k}\}.\]  A distribution on ${\sf Spanning Forests(G)}$ is said to be size biased, if for ${\bf f}$ taken under this distribution
$`P({\bf f}=(f_1,\cdots,f_k))$ is proportional to $\prod_{j=1}^k |f_j|$ for any $k$, and $(f_1,\cdots,f_k)\in{\sf Spanning Forest(G) }$ (and zero, otherwise): roughly, this distribution favours the multiplicity of components of small sizes $\geq 2$. \par
The size bias is equivalent to the rooted case model, in which each tree is rooted at one of its vertex, since the number of possible roots of a given tree is given by its size.
To build a size biased spanning forest of the graph $G$, it suffices to add a point to the vertex set, that is to take $V'=V \cup \{z\}$, and to add an edge between $z$ and all the elements of $V$, that is to define $E'=E \cup \{\{z,v\},v \in V\}$. Set $G'=(V',E')$.

\NewModel{The tree containing $r$ in a size biased forest}{SBF}
{Let ${\bf T'}$ be a UST of $G'=(V',E')$, and consider the spanning forest ${\bf f}=({\bf f}_1,\cdots,{\bf f}_k)$ (for some $k\geq 1$), with vertex set $V$ and edge set $E({\bf f})=E({\bf T'})\cap E$, that is, the edges of ${\bf T'}$ not adjacent to $v$.}
The forest ${\bf f}$ is a size biased spanning forest since each ${\bf f}_i$ can be connected by $|{\bf f}_i|$ different edges to $z$.   
Let $\bt$ be the connected component of ${\bf f}$ containing $r$. For all $t\in\TTTr{r}{G}$,
\[`P({\bf t}=t) = |t|\times\frac{\l|\SP( G'\setminus t)\r|}{| \SP(G')|} \] where $| \SP(G')|$ is the number of spanning trees of $G'$, and $|\SP( G'\setminus t)|$ the number of spanning trees of $G'$ deprived of all the vertices of $t$.

Notice that here $|\SP(G'\setminus t)|$ can be computed using the matrix tree theorem and then, if a bound on  $|\SP(G'\setminus t)|$ is known for all subtrees $t$ of size $n$, a rejection method can be used to sample a uniform  element of $\TTr{r}{G}{n}$.\\
\noindent{\Analysis:} The computation of a uniform spanning tree of $G'$ is fast, and can be done on huge graphs. \\
{\Drawback:} This distribution can not be used in general to sample uniformly in $\TTr{r}{G}{n}$; indeed, the rejection method here is unlikely to work if the desired size $n$ is far from 0 and $|V|$: in most graphs $G$, it produces some huge ratios between the weights $\l|\SP( G'\setminus t)\r|$ and $\l|\SP( G'\setminus t')\r|$ for $t,t'\in \TTTr{r}{G,n}$. Besides, the evaluation of $\l|\SP( G'\setminus t)\r|$  by the matrix tree theorem produces also some difficulties if the graph is large, since manipulation of huge integers is an issue. 

\paragraph{\underline{Variant}}
A method to favour larger components is to use Wilson's algorithm with some random walks $(W^b_i,i\geq 0)$ less likely to visit  $z$. When $W^b_i=z$, the node $W^b_{i+1}$ is uniform on $V$; otherwise, if $W^b_i=v \in V$, then $W^b_{i+1}=z$ with probability $p$, and with probability $1-p$, $W^b_{i+1}$ is a uniform neighbour of $v$ in $G$. This construction induces a distribution on $\SP(G')$ proportional to
\[ p^{{\sf Indegree}(z)} \prod_{u\neq r, f(u)\neq z} \frac{1}{\degree_G(u)}\]
where ${\sf Indegree}(z)$ counts the number of steps with destination $z$ 
in the construction and $f(u)$ denotes the father of $u$ in the final spanning tree of $G'$. This is valid when $z$ is not chosen as the first point in Wilson's algorithm (otherwise some minor adaptations are needed). Hence, for a $d$-regular graph $G$, this is proportional to
$ \l(d\,p\r)^{{\sf Indegree}(z)}$.
\begin{lem} Assume that $G$ is $d$-regular, let ${\bf T}''$ be the spanning tree of $G'$ constructed by the variant presented above, and the spanning forest ${\bf f}=({\bf f}_1,\cdots,{\bf f}_k)$ of $G$ (for some $k\geq 1$), with vertex set $V$ and edge set $E({\bf f})=E({\bf T''})\cap E$.
  For any spanning forest $f=(f_1,\cdots,f_{D+1})$ of $G$,  $\P({\bf f}=f~|~\{{\sf Indegree}(z)=D\})$ is proportional to $\prod_{j=1}^{D+1}|f_j|$.
\end{lem}

In practice, on the graph $\Torus{N}$, it is possible to adjust $p$ so that the probability of the event $\{{\sf Indegree}(z)=1\}$ is far from 0; by acceptance/rejection, it then gives a procedure to simulate ${\bf f}$ with ${\sf Indegree}(z)=1$, in words, a spanning forest containing two trees. It is also possible to condition by $\{{\sf Indegree}(z)=1$,$|{\bf t}(r)|=n \}$  (see Fig. \ref{fig:SBF}).\par
For a tree $\bt$ rooted at $r$ with diameter $d<N$ on  $\Torus{N}$, call canonical embedding of $({\bf t},r)$, denoted ${\sf Canonical}(\bt)$, the tree in $\Z^2$, rooted at 0,  obtained by taking the translated tree ${\bf t}-r$, and projected in $\mathbb{Z}^2$ (in the only reasonable way which preserves the orientation of the edges). 
\begin{conj}Conditionally on $|{\bf t}(r)|= n$, the rescaled vertex sets, ${\sf Canonical}({\bf t}^{<n}(r))/\sqrt{n}$, converges in distribution for the Hausdorff metric on compact subsets of $\R^2$ to a limiting compact set $K$, with Lebesgue measure 1, simply connected. 
\end{conj}
One could further conjecture that the contour process possesses a limiting distribution, probably having some common features with $SLE_8$ (the contour has to be thought as a path that turns around the tree at constant speed, at distance equals to the lattice mesh divided by 3, so that it is a close curve that characterizes the tree).  However, the fact that the imposed condition provides an object with area 1, and since this property is not conformal invariant, the connections with SLE seems not trivial, and the conjecture difficult to state. The interface of $K$ seems to have also to be a SLE type trajectory, which seems to be simple, and could be conjectured to, still at the limit, surrounds a domain with area 1. Again, this area condition implies that even stating a conjecture is not a simple task.

Another link, maybe a bit more speculative, would concern some possible relations with the massive version of $SLE_2$.  The global construction of our tree has some similarities with the model of Makarov \& Smirnov \cite{MaSm} who studied loop-erased random walk with killing: at each time the walk has a small probability proportional to $m^2$ to be absorbed by a cemetery point; such a random walk  conditioned to start and finish at some given points of a domain, converges towards $SLE_2^{(m)}$, a massive version of $SLE_2$. This construction is similar to the construction of the current variant, which uses loop erase random walks that can reach, at each step, the additional point $z$ with a small probability. It may then be expected that some asymptotic characteristics of our model could be related to $SLE_2^{(m)}$ (for example, the limit path from a vertex conditioned to be in the tree, to the root).

\begin{figure}[h!]
\begin{center}\includegraphics[width= 8cm]{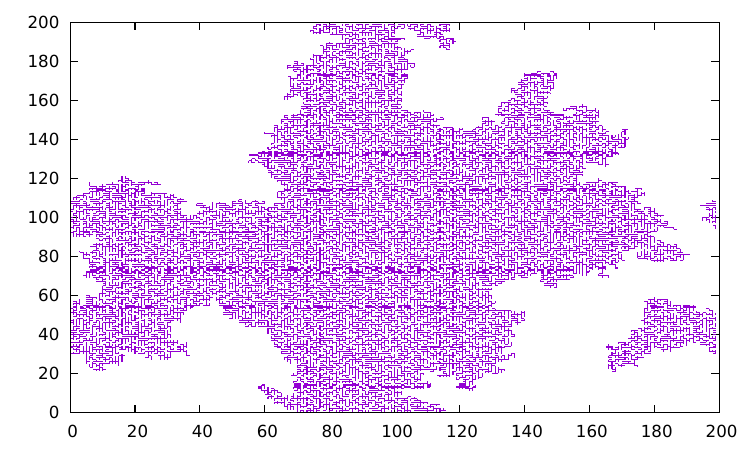} \includegraphics[width=8cm]{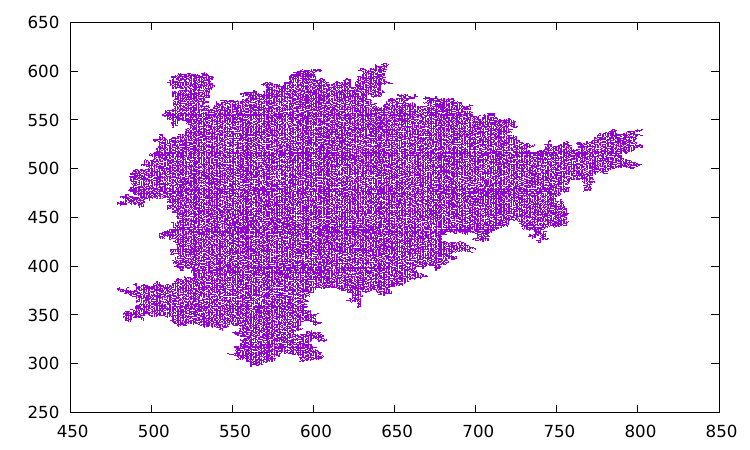}
\end{center}
\captionn{\label{fig:SBF} Left: Simulation with $p= 10^{-5}$ on the $\Torus{200}$, conditioned on $D=1$, and by the fact that the tree attached to $r$ ``the center of the torus'' has size between $[19000,21000]$, that is approximately half of the total size (240 simulations were needed, the output size of ${\bf t}(r)$ is 20852). Right: $p = 2\times 10^{-7}$,  on the $\Torus{1000}$, conditioned on $D=1$, and $|{\bf t}(r)|\in [45000,55000]$ (2553 simulations were needed before satisfying these conditions, with output $|{\bf t}(r)|= 52106$).  }
\end{figure}
\subsection{Subtree extraction of the uniform spanning tree}
\label{sec:dqskpd}
A method that seems promising to obtain an element of $\TT{G}{n}$ with a prescribed distribution, is a two steps procedure: first, sample a UST $\bt$ of $G$, and then, extract by a second (random) procedure, a subtree $\bt'$ of $\bt$. \\
 S. Wagner \cite{Wagner2019}, gives a lower bound on the probability that a randomly chosen uniformly in $\TTT{G}$, is spanning (depending on a linear lower bound on the minimum degree of the nodes).
Chin et al. \cite{Chin18} obtained that if ${\bt}(G)$ is a uniform random unrooted tree in  $\TTT{G}$, then
\[\P( {\bt}(K_n) \text{ is spanning }) \to e^{-1/e},~~\P({\bt}(K_{n,n}) \text{ is spanning }) \to e^{-1/e^2}.\]

It turns out that getting a uniform element $\bt'$ of $\TT{G}{n}$ by such a two step procedure seems really difficult except when $n$ is very small (and maybe, an obstruction comes from the fact that the edges of the uniform spanning trees form a determinantal process, as shown by Burton \& Pemantle \cite{Burton-Pemantle}).

However, extraction of subtrees of UST allows us to obtain some interesting models; we review some of them here, but additional ways to extract random subtrees from a tree are examined in Section \ref{sec:exact_samp}.

 \subsubsection{Uniform random subtree of the UST}
 
  In Section \ref{sec:egu}, we will provide an algorithm to sample a uniform subtree of a given tree (or uniform conditionally on the size, with some adjustable parameter to favour a given mean size); it is tempting to use these algorithms on a uniform spanning tree ${\bf UST}$ of $G$.
  Here we make explicit the distribution of $\bt^{(n)}\sim \uniform(\TT{{\bf UST}}{n})$, whose support is $\TT{G}{n}$. Since any tree of size $n$ is a subtree of at least one spanning tree (see Simulation in Fig. \ref{fig:simduksd}).
  For all $t\in \TT{G}{n}$
\[\P(\bt^{(n)}=t)  \propto \sum_{T \in \SP(G)} \frac{ 1_{t \in \TTTr{v}{T}}}{|\TTTr{v}{T}|}.\] 
\begin{figure}[h]
	\centerline{\includegraphics[width = 8cm]{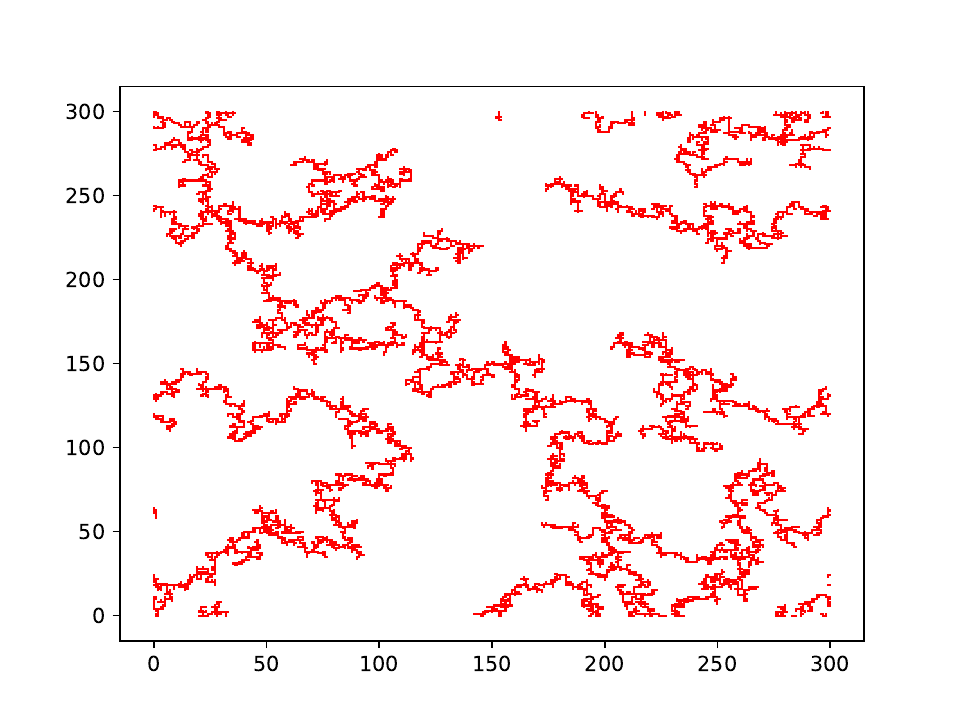}\includegraphics[width = 8cm]{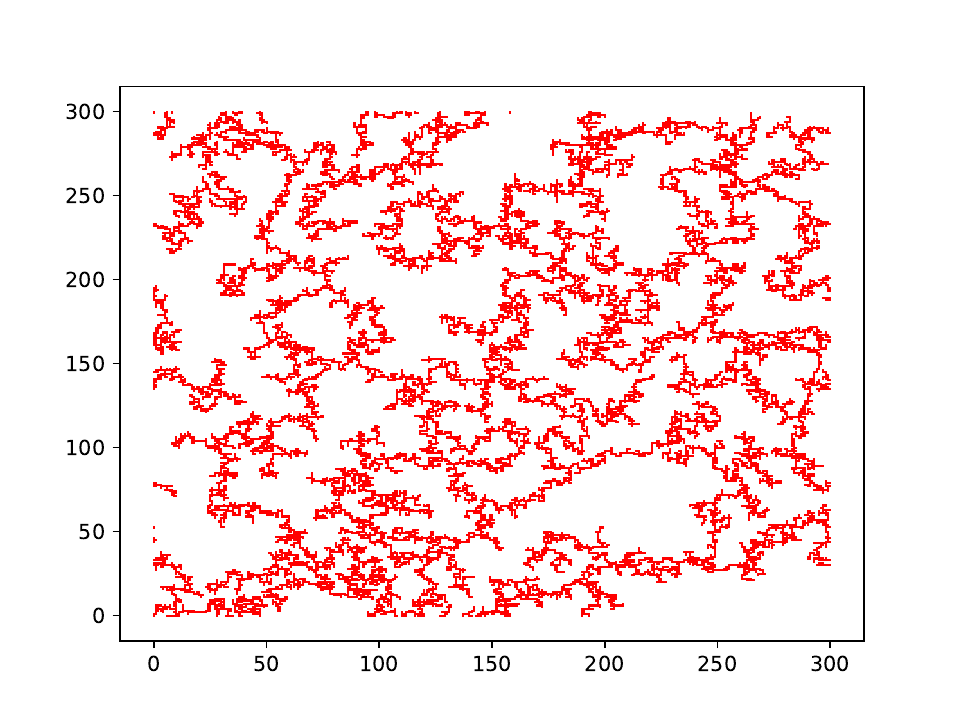}}
	\captionn{\label{fig:simduksd} Simulation of a uniform subtree of a rooted UST of $\Torus{300}$. In each picture, the UST is sampled using Wilson's Algorithm. \textbf{Left:} $\pp_i= 0.35 = 1-\rr_i$ for all $i\in \cro{2,|T|-1}$, the output is a tree of size $9589$ . \textbf{Right:} $\pp_i= 0.36= 1-\rr_i$ for all $i\in \cro{2,|T|-1}$, the output is a tree of size $18626$. \label{FIG:simu2}}
\end{figure}

\subsubsection{Model of evaporation of the edges of a UST}
\label{edgeremoval}

 Take a rooted UST $({\bf T},r)$ of $G$, with as usual its edges directed toward $r$. Consider a sequence $({\bf u}_i,i\geq 1)$ of i.i.d. uniform nodes on $V\setminus \{r\}$. Define the sequence of forests $({\bf F}_i: i\geq 0)$ by, ${\bf F}_0=\{{\bf T}\}$,  and for $i\geq 1$, ${\bf F}_i$ is obtained from the removal of the outgoing edge of ${\bf u}_i$ from  ${\bf F}_{i-1}$ (which increases the number of trees by $1$ if this edge is removed). Let ${\bf t}_i(r)$ be the connected component of $r$ in ${\bf F}_i$, and set ${\bf t}^{<n}(r)$ be the first element in the sequence $({\bf t}_i(r): i\in \N)$ such that ${\bf t}_i(r)<n$, a target size.
\begin{figure}[h]
 \centerline{ \includegraphics[scale=0.71]{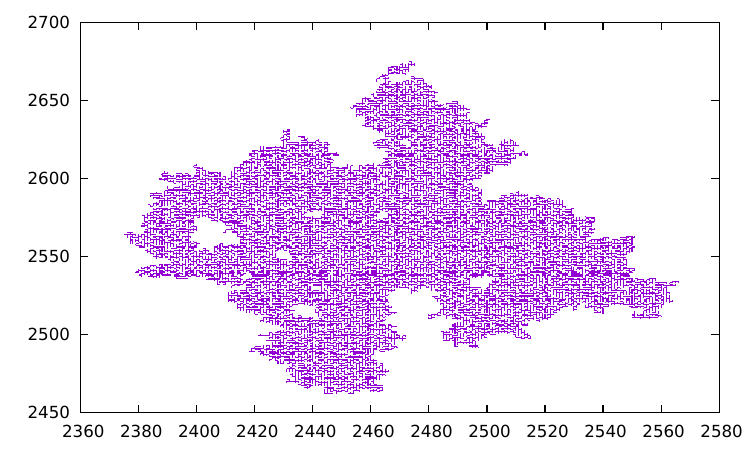}}
\captionn{Simulation of a UST on $\Torus{5000}$ seen as directed toward its root; removal of edges is done until ${\bf t}^{<n}$ has a size smaller than $20000$. In this example, 43707 removals were needed and the size of ${\bf t}^{<n}$ is 18159.} 
\end{figure}\par
In the literature, the removal of a single random edge $e$ of a tree $T$ gives rise to two connected components, and the connected component which does not contain the root is called a fringe subtree; it has been studied for numerous models of random (non embedded) trees (see e.g. Aldous \cite{Aldous_fringe}, Holmgren \& Janson \cite{H-J} and references therein). 
Here when a single uniform edge is removed, we are interested in the connected component which contains the root $r$, that we call ${\bf t}_1(r)$; therefore, for a fixed $t\in \cup_{n=1}^{|V|-1}\TTTr{r}{G,n}$
\[\P({\bf t}_1(r)=t) = \frac{|\SP({\sf Induced}_G(V\setminus V(t))|\times |\{ \{u,v\}\in E, u \in t, v \notin t\}|}{|\SP(G)|}.\]
This comes from the fact that, before the edge removal, both connected components were connected by one of the edges between them in $G$, and the connected component not containing the root is any spanning tree of ${\sf Induced}_G(V\setminus V(t))$. In general, ${\bf t}_1(r)$ is not uniform, nor uniform conditionally on its size (one notable exception, is when $G$ is the complete graph $K_n$).
\begin{conj} Let $(N(n))$ be a sequence of integers such that $\limsup n/N(n) <1$. Suppose $G=G_{N(n)}=\Torus{N(n)}$.
  The rescaled vertex sets, ${\sf Canonical}({\bf t}^{<n}(r))/\sqrt{n}$, converges in distribution for the Hausdorff metric on compact subsets of $\R^2$ to a limiting compact set $K$, with Lebesgue measure 1, simply connected.  
\end{conj}
The condition  $\limsup n/N(n) <1$ ensures that the diameter of  ${\bf t}^{<n}$ is smaller than ``the torus side''.

We conjecture that there is no loss of area at the limit, because, in the simulations it can be seen that most of the edge removals discard very small parts of the connected component of $r$.\par
The spanning case corresponds to the case where no edge is removed, which is equivalent to $\frac{n}{N(n)^2}=1$. In this case, on the analogue configuration on a square $\cro{1,n}\times\cro{1,n}$ of the square lattice the scaling limit of the interface is described by $SLE_8$ (See \cite{LSW04}).
  Due to the area constraint, which is not preserved by conformal transformations, if a scaling limit exists it would not be conformal invariant.

\subsection{DLA type model} \label{sec:dla}
The DLA has been introduced by Witten \& Sander \cite{W-S}, in 1983; very little is known about it, see e.g. Eberz-Wagner \cite{EW}.

The common definition of the DLA on the lattice $\Z^2$ is as follows: at time 0, the set of occupied vertices is $S_0:=\{(0,0)\}$. Then some particles are launched, successively, and performs a simple random walk on $\Z^2$, meaning that each step is equally likely, $(-1,0), (1,0),(0,-1),(0,+1)$, independently of all steps of all random walks. When the $i$th particle reaches a vertex $x_i$ which is at distance 1 (for the $L^1$ distance) to the set of occupied vertex $S_{i-1}$,  it is somehow frozen in that position, and one sets $S_i:=S_{i-1} \cup \{x_i\}$.
The cluster obtained $S_n$, depends on the launching points of the random walks. The DLA is the cluster $S_n$ obtained by letting the launching  points go to $+\infty$.

In what follows, we propose a small variation of this construction (which is already present in the literature, see e.g. \cite{Bot}), which can be defined on any graph $G=(V,E)$, and which allows one to define a new model of random subtree of $G$.

We call this model $\TDLA$, where the prefix {\sf T} stands for tree.
\begin{rem} In the lattice case, the vertex set of our $\TDLA$ is not distributed as the DLA, because the stopping rule of the random walk we adopt is not exactly the same.
\end{rem}

To define it, consider a sequence of simple random walks $(W^{k}=(W^k_i,i\geq 0), k \in \N)$ starting from $W_0^{k} = \infty$ for all $k\in \N$.  
The TDLA is a sequence of subtrees $({\bf tdla}_i,i\geq 0)$, where ${\bf tdla}_i=(\bD_i,E(\bD_i))$, which is defined recursively as follows. Set $\bD_0=\{r\}$, $E(\bD_0)=\varnothing$. Assume ${\bf tdla}_{k-1}=(\bD_{k-1}, E(\bD_{k-1}))$ has been defined for $k-1\geq 0$.
 Instead of waiting for $W_0^{k}$ to be at distance 1 from the vertex set $\bD_{k-1}$ of the current tree ${\bf dla}_{k-1}$, wait till the hitting time
of this vertex set \[\tau_k=\inf\l\{m : W_m^k \in \bD_{k-1}\r\},\]
so that $\bfe_k=(W_{\tau_k-1}^k,W_{\tau_k}^k)$ is the step allowing to reach ${\bf dla}_{k-1}$.

~\medskip
This construction can be performed on any graph at the price of two modifications: replace $(0,0)$ by a marked vertex, and the starting point $+\infty$ of the random walks, by another choice of distribution, for the launching points.
\NewModel{The (finite graph) DLA tree}{DLT}
{On a finite connected graph $G=(V,E)$ with $r\in V$, the TDLA sequence $(\TDLA_r(j,G), 1\leq j \leq |V|)$ is defined as explained above for $({\bf dla}_k,k\geq 0)$ with two simple modifications: the random walks  are independent simple random walks on $G$ which start at i.i.d. points $(w_k, k \in \N)$ chosen uniformly in $V$, and if a random walk $W^{k}$ has its starting point $W^k_0$ in the current tree $\TDLA_r(k-1,G)$, then  (do nothing and) set $\TDLA_r(k-1,G)=\TDLA_r(k,G)$.}
\begin{figure}[h!]
	\centerline{\includegraphics[height=7cm,width=7cm]{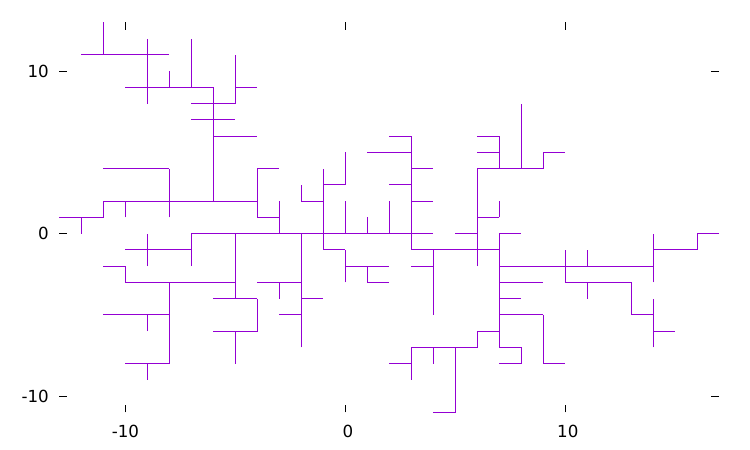} \includegraphics[height=7cm,width=7cm]{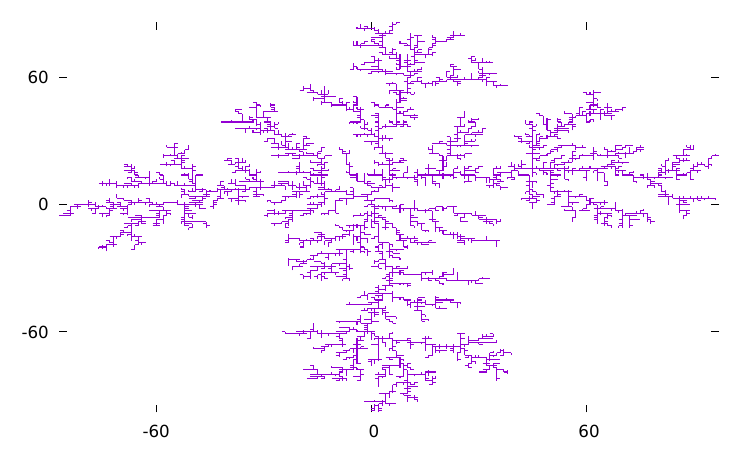} }
	\captionn{DLA tree with $k=250$ and $k=5000$ vertices, built on $\Torus{1000}$. The initial particle is at $(0,0)$.}
\end{figure}
This way of defining the TDLA seems efficient to us in the sense that it allows us to define the TDLA on all graphs: for example, on the complete graph, it allows us to construct uniform increasing trees (the edges from any node to the root are increasing, and the node labels are exchangeable).
The standard TDLA would be defined on $\Z^2$ using random walks starting from $\infty$ as explained above.
\begin{conj}\label{eq:qfdq} There exists $C \in (1/2,1)$ such that, for any $c>C$, \[D_{\sf var}\l(\TDLA_{0}(n,\Torus{n^c}),{\bf dla}_n\r)\to_{n\to+\infty} 0\]
 where $D_{Var}$ is the total variation distance.
\end{conj}
The natural model of TDLA on $\Z^2$ (with particles coming from $\infty$) appears to be a kind of limit of  $\TDLA_0(n,\Torus{N})$ when $N\to +\infty$  (or of $\TDLA_0(n,[-N,N]^2)$), with the initial particle placed at 0, since, for $N\to+\infty$, the $n$ starting points of the $n$ random walks goes to $+\infty$ with $N$, and the topology of the graph far from 0 should not play an important role. \\
If one works on the square $[0,N]^2$ with an initial point at $r=(0,0)$, performing a reflected simple random walk on the square, then, one gets an object which has a (single) diagonal symmetry in distribution (see Fig. \ref{fig:squareDLA}):
\begin{figure}[h!]
  \centerline{\includegraphics[height=7cm,width=7cm]{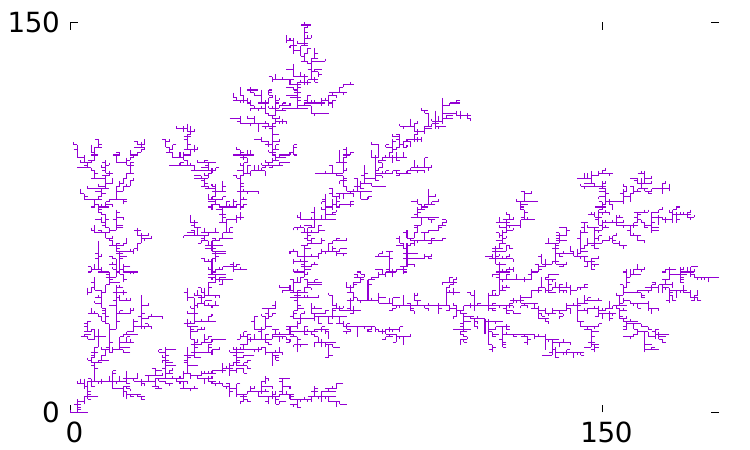}}
  \captionn{\label{fig:squareDLA} Corner DLA tree $\TDLA_{0}(5000,[0,999]^2)$ with initial vertex at $(0,0)$ (that is defined on the square $[0,999]^2$, with root at $(0,0)$ and 5000 vertices).}
\end{figure}
{We call this DLA, \textbf{the corner DLA}. The initial vertex is at a corner, and there are two parameters: the square side, and the number of particles.

Of course, as everyone who has seen these kinds of pictures, it is tempting to conjecture that for a sequence $N(n)\to\infty$, there exists a sequence $a(n)\to\infty$ such that
\be
\frac{\TDLA_0(n,\Torus{N(n)}}{a(n)}\dd \TDLA^{\infty}\ee
for the Hausdorff metric topology on compact subset of $\R^2$, where $\TDLA^{\infty}$ is a.s. a (non-trivial) continuum random tree embedded in $\R^2$ (that is a connected subset of $\R^2$, where between any two points $x,y\in \TDLA^{\infty}$, there is a single injective path $\gamma$ (up to the time parametrization), such that $\gamma(0)=x$, $\gamma(1)=y$, and $\gamma\in[0,1]\in \TDLA^{\infty}$). As can be guessed from Fig. \ref{fig:squareDLA}, a convergence can still be conjectured for $\TDLA_{0}(n,[0,N(n)]^2)/a(n)$ (probably for the same normalization) to another continuum random tree $\TDLA_{0}^{\infty}$.
  \begin{rem} One finds in the literature many random growth models of DLA type, aims at modelling various physical, electrical, biological or chemical real phenomenons. Many of them provides tree like structures embedded in $\R^2$ or $\R^3$. In a lot of cases, aggregations of new particles depend on the complete current structure, and their study are the most often, if not always, complex. We refer to  Vicsek \cite{Vic} for an overview of these questions, results and simulations.
\end{rem}
\subsubsection{A few statistics on square DLA}
We made some simulations and statistics to try to guess the critical exponents in the case of square DLA starting with a single vertex in a corner.
\ben
\begin{array}{|c|c|c|c|c|c|}
	\hline
	\textrm{Tree size}  & 5000 & 6000 & 7000 & 8000\pass\hline
	\textrm{Number of simulations}  & 13671 & 13659 & 13645 & 13635\pass\hline
\end{array}
\een

Again, we made two types of distance statistics, as in Section \ref{sc:SC}:  the \textbf{Euclidean width and height} $w(t)$ and $h(t)$ (number of vertical resp. horizontal row occupied), and random graph distance $ {\bf D}( t) = d_{ t}({\bf c},{\bf v})$ between this time, the ``root corner'' and a random node in the tree. To make the statistics, for each simulation, we used both values $w(t)$ and $h(t)$, and sample 10 random nodes ${\bf v}$ for each DLA. We use the same methods as in  Section \ref{sc:SC} to evaluate the more plausible values of $\alpha$ and $\beta$ for which
$w(t_n)/n^{\alpha}$ and $ d_{ t}({\bf c},{\bf v})/n^{\beta}$ would converge in distribution, given our samples. The square size is the same for all simulations ($1000\times 1000$).
\ben
\begin{array}{|c|c|c|c|c|c|}
	\hline
	\textrm{Number of nodes}  & 5000 & 6000 & 7000 & 8000\pass\hline
	\textrm{Empirical mean of the width}  & 170.93 & 190.31 & 208.36 & 225.31\pass\hline
	\textrm{Empirical median of the width}  & 171.00 & 190.00 & 208.00 & 225.00\pass\hline
	\textrm{Empirical mean of $d({\bf c},{\bf v})$}  & 160.62 & 178.69 & 195.43 & 211.48\pass\hline
	\textrm{Empirical median of $d({\bf c},{\bf v})$}  & 166.00 & 185.00 & 202.00 & 218.00\pass\hline
\end{array}
\een

\ben
\begin{array}{|c|c|c|c|c|}
	\hline
	\textrm{$(n,m)$}  & (5000, 6000) & (6000, 7000) & (7000, 8000)\pass\hline
	\textrm{Estimation of $\alpha$ (mean) }  & 0.589 & 0.588 & 0.586\pass\hline
	\textrm{Estimation of $\alpha$ (median) }  & 0.578 & 0.587 & 0.588\pass\hline
	\textrm{Best fit decile $\alpha$  }  & 0.579 & 0.581 & 0.590\pass\hline
	\textrm{Estimation of $\beta$ (mean) }  & 0.585 & 0.581 & 0.591\pass\hline
	\textrm{Estimation of $\beta$ (median) }  & 0.594 & 0.570 & 0.571\pass\hline
	\textrm{Best fit decile $\beta$}  & 0.581 & 0.586 & 0.586\pass\hline
\end{array}
\een

\begin{figure}[h!] \centerline{\includegraphics[width = 6cm]{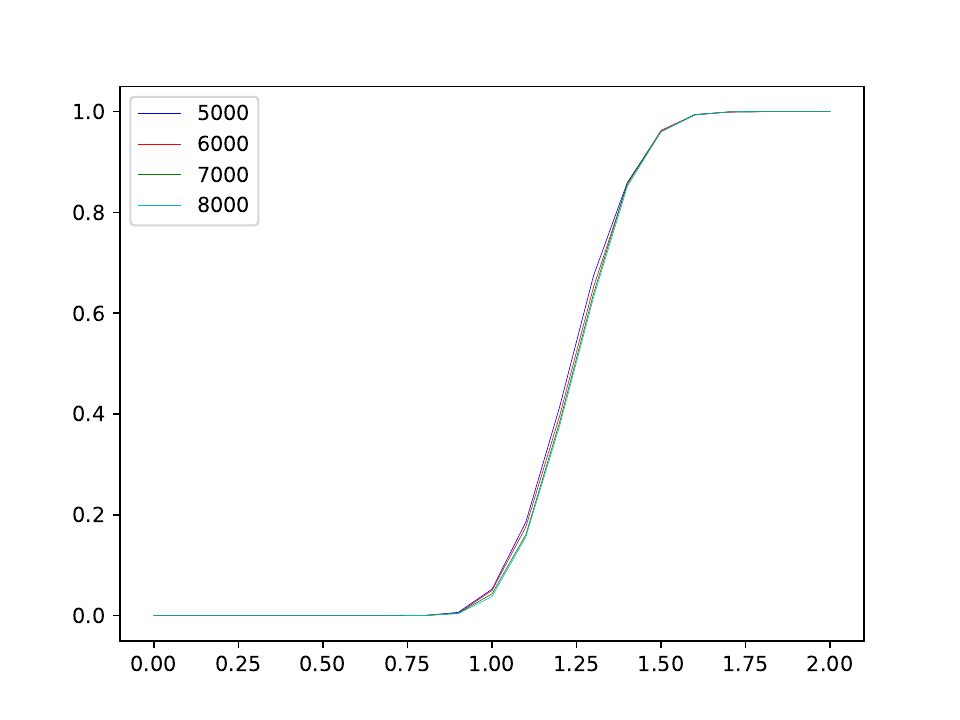}\includegraphics[width = 6cm]{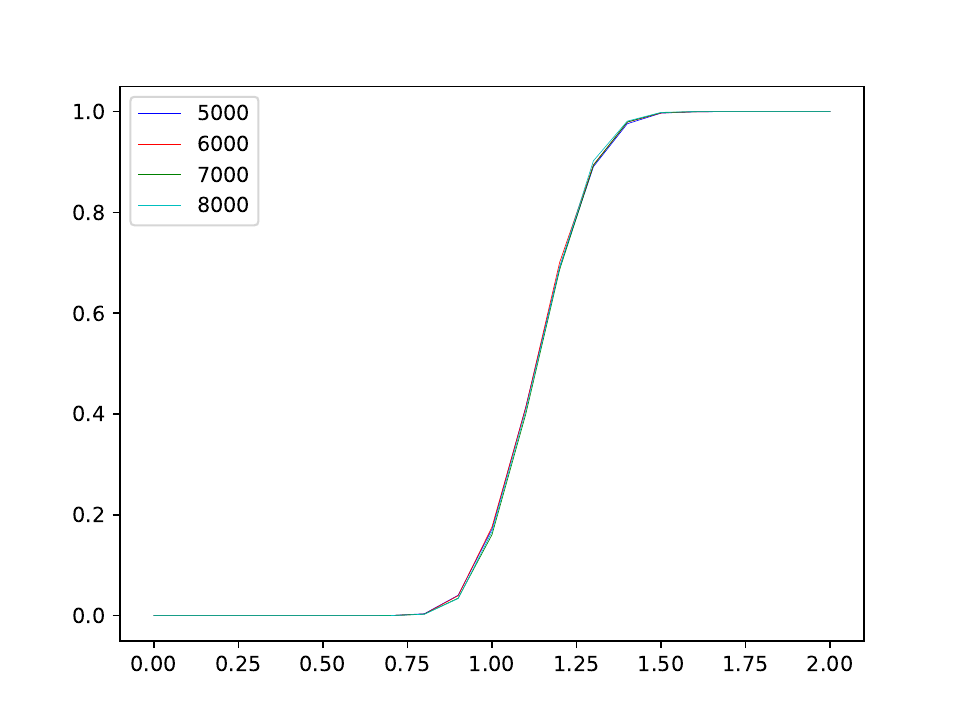}\includegraphics[width = 6cm]{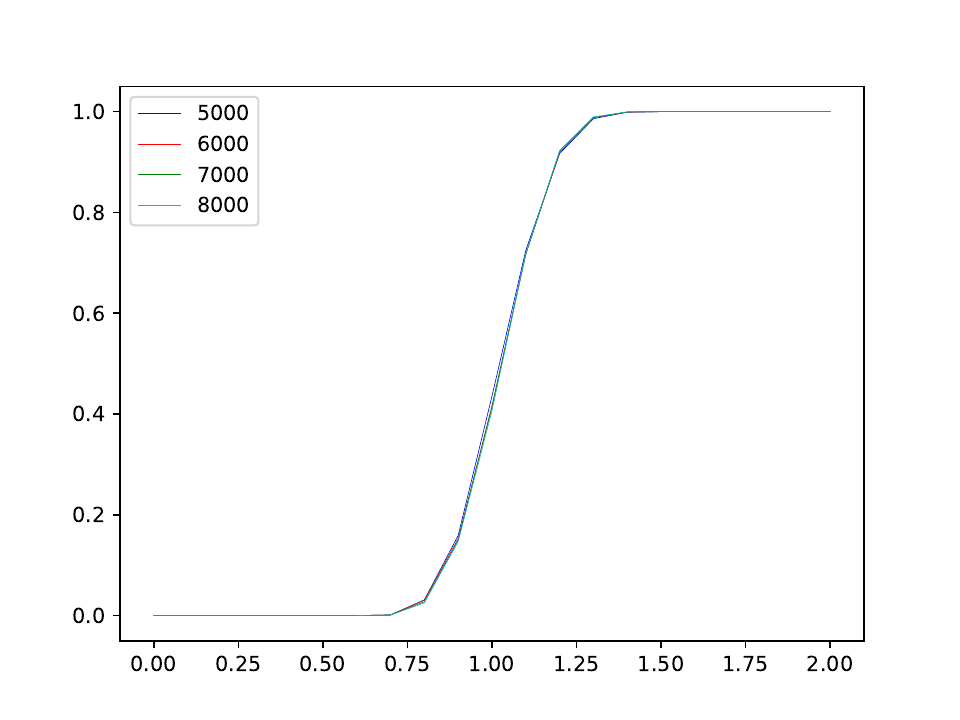}} 
  \centerline{\includegraphics[width = 6cm]{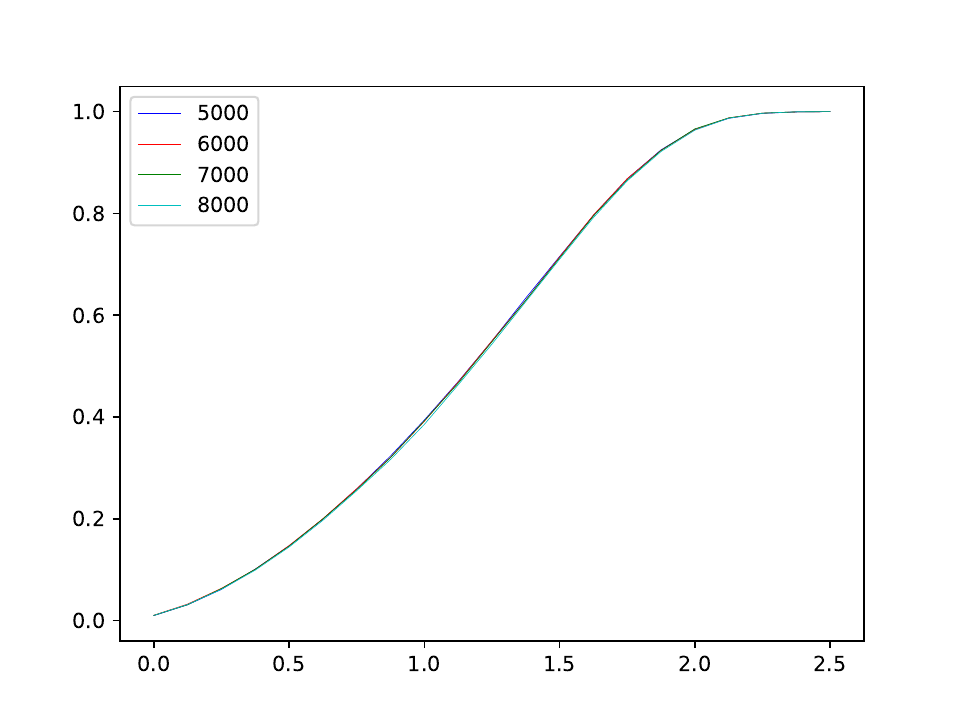}\includegraphics[width = 6cm]{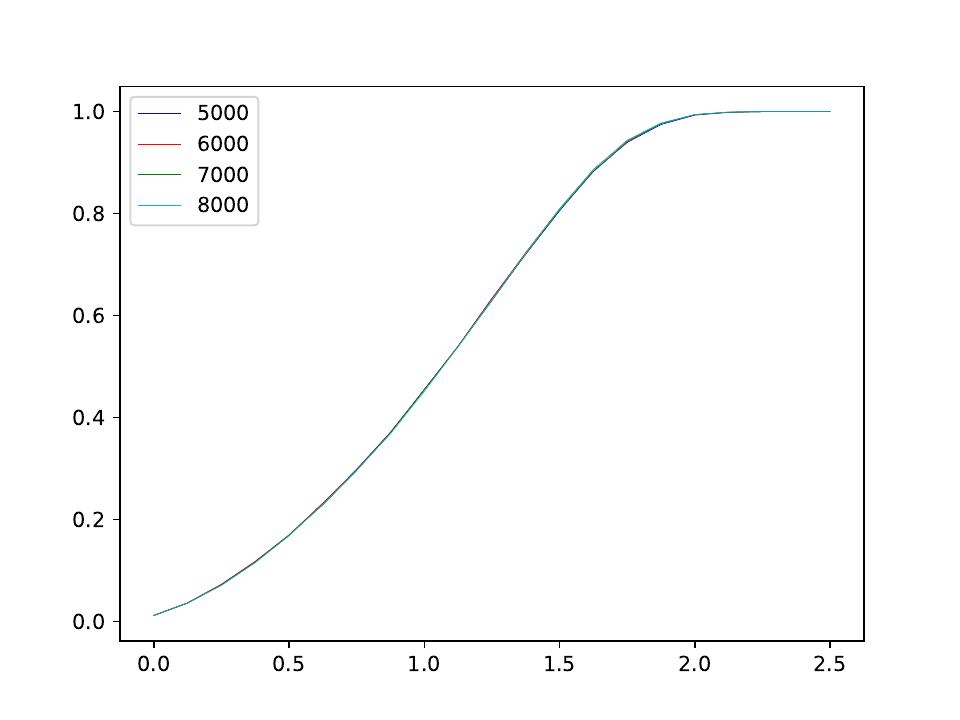}\includegraphics[width = 6cm]{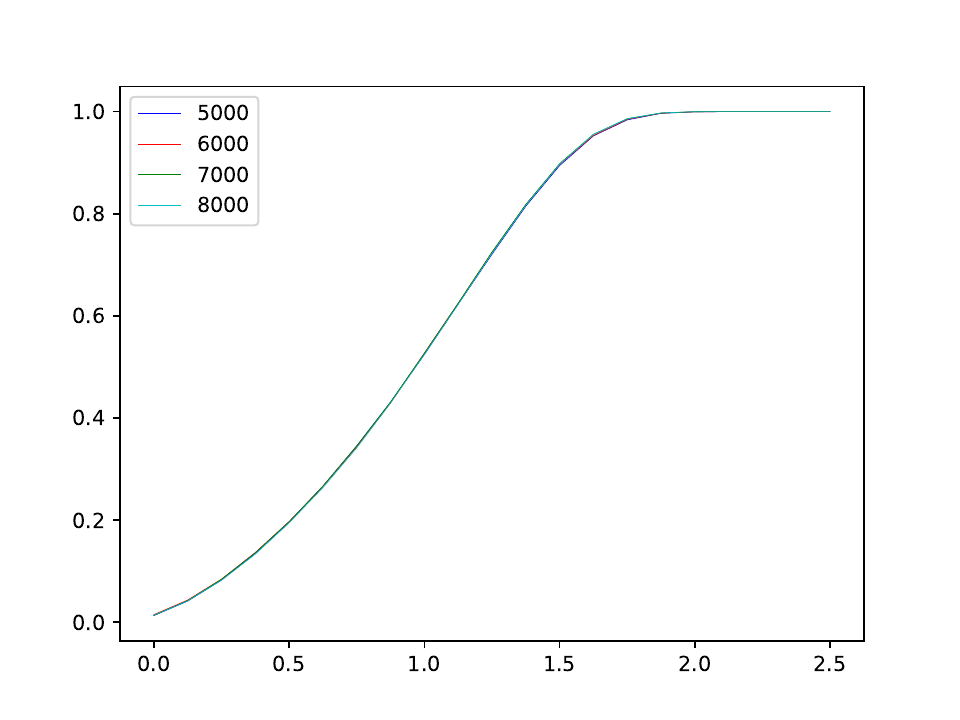}}
  \captionn{\label{fig:dladqdqtf66}On the first line, (interpolated) empirical cumulative function of $w(t_n)/n^{\alpha}$ for $\alpha$ being respectively 0.57, 0.58 and 0.59. On the second, (interpolated)  empirical cumulative function of $d_{t_n}({\bf c},{\bf v})/n^{\beta}$ for $\beta$ being respectively 0.57, 0.58 and 0.59. }
\end{figure}

\begin{conj}In the case of $\TDLA_{0}(n,[0,+\infty)^2)$, both  $w(t_n)/n^{\alpha}$ and $d_{t_n}({\bf c},{\bf v})/n^{\beta}$ converge in distribution for $\alpha=\beta$, for a certain $\alpha\in[0.55,0.61]$.
\end{conj}
Since the points start inside the square, they are more likely to start ``inside'' the current cluster, which implies, probably, that our simulations produce results a bit smaller than the expected limiting values. In Lawler \cite[Sec. 2.6]{Lawl} it is discussed that $\alpha$ is conjectured to be $(d+1)/(d^2+1)$ in dimension $d$.
A version of the DLA on the upper half plane is defined and studied in Procaccia \& Zhang \cite{PZ}.

\subsection{Internal DLA}
\label{sec:dqffsd}

This model has been introduced by Diaconis \& Fulton \cite{DF} and it is defined as follows.  Consider a sequence of i.i.d. simple random walks $(W^{k}, k \in \N)$, all of them starting at the same vertex $W_0^{k} = r$ for $k\in \N$.
The internal DLA is a sequence of clusters of vertices $(\bI_i)_{i=0}^{\infty}$ defined as follows. Set $\bI_0=\{r\}$. Assume $\bI_k$ has been defined and define $\bI_{k+1}=\bI_k\cup \{\bu_{k+1}\}$, where $\bu_{k+1}$ is the first point in the complement of $\bI_{k}$ hit by the random walk $W^{(k+1)}$ (that is let $\tau_{k+1}=\inf\{ m : W_m^{k+1}\notin \bI_{k}\}$, then $\bu_{k+1}=W_{\tau_{k+1}}^{k+1}$) .

\NewModel{The internal DLA tree}{IDT}
{Use the random walks defined above. Define the sequence of trees $(\bT_k, k \in \N)$ as follows: $\bT_0$ is the tree reduced to its root $r$. To define $\bT_{k+1}$ from $\bT_{k}$, add the edge $\bfe_{k+1}=(w_{\tau_{k+1}-1}^{k+1},W_{\tau_{k+1}}^{k+1})$ corresponding to the step of the random walk $W_{k+1}$ reaching a node in the complement of the vertex set of $\bT_k$. Again $\bT_k$ is a tree with $k+1$ vertices (simulations on Fig. \ref{fig:tidla}).}

\begin{figure}[htbp]
  \centerline{\includegraphics[width=6.5cm,height=6.5cm]{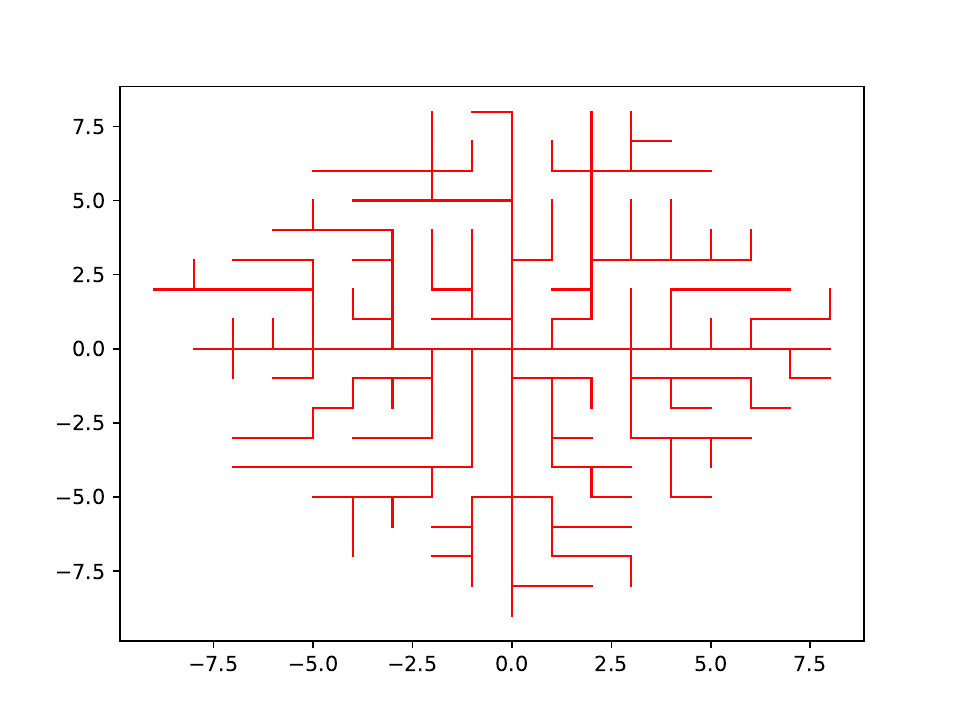} \includegraphics[width=6.5cm,height=6.5cm]{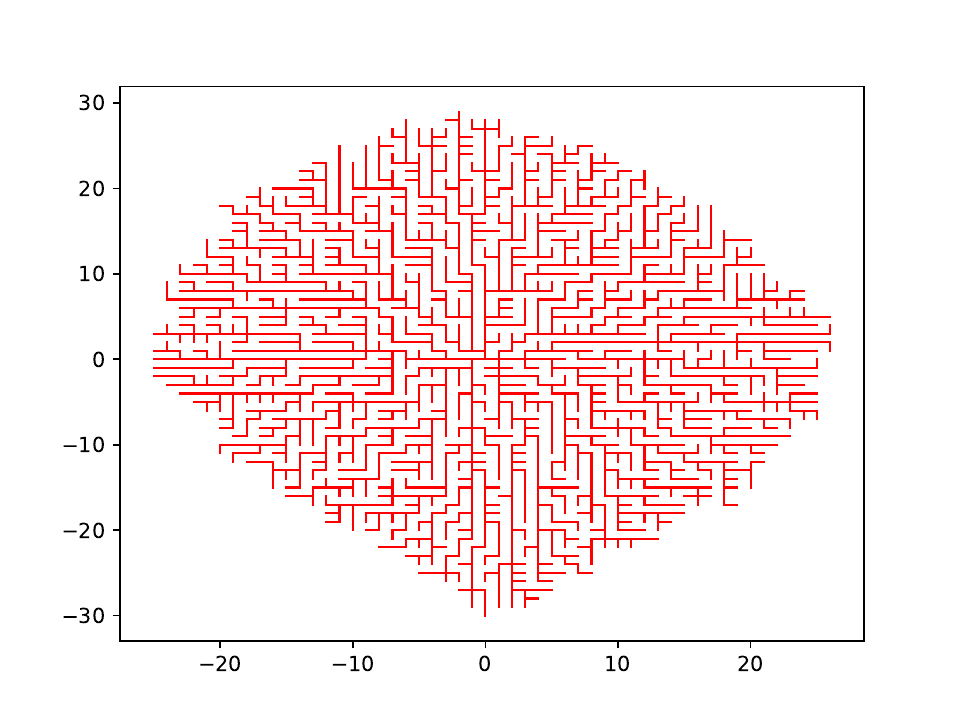} }
  \captionn{\label{fig:tidla} Tree Internal DLA with $k=200$ and $k=2000$ vertices, $v= (0,0)$.}
  \end{figure}
  Much information is known on the cluster (see Lawler et al. \cite{LBG} for a limit shape theorem, Levine and Sheffield \cite{JLS1} for the fluctuations in 2D (see also Lawler \cite{L10.1214}), Jerison et al. \cite{MR3161315}, and Jerison et al. \cite{JLS} in larger dimension).

\subsection{Constructions on weighted graphs}
\label{sec:MST}

In this part, we assume some i.i.d. weights  $\bC=({\bf C}_e :e\in E(G))$ associated with edges picked according to a non-atomic measure $\mu$ on $(0,+\infty)$.
The induced random order $\sigma$ of the edges is the (a.s. well defined) permutation satisfying
\ben\label{eq:sig}
e_{\sigma(1)} < \cdots < e_{\sigma(|E|)}.
\een

The two first models given in this part are built using Prim's \cite{Prim} and Kruskal's \cite{Krus} algorithms which extract the minimum spanning tree (MST) of a weighted graph $(G,\bC)$. In fact, the MST is a function of the induced random order $\sigma$ (this is a consequence of Prim's, Kruskal's and also Bor\.{u}vka's algorithm  \cite{Boru}; we refer to Nes\u{e}tr\u{\i}l et al. \cite{MR1825599}, Graham and Hell \cite{MR783327} for historical notes on this problem).

\subsubsection{Prim's component of the origin} 

\NewModel{The Prim component of a vertex}{PrimMod}{Build a sequence of trees $(\bP_j, j=1,\cdots,N)$ where $\bP_j$ is a tree with $j$ nodes, as follows: first, take $\bP_1=r$ a fixed node. Assume that $\bP_j$ has been built and set $\bP_{j+1}$ as the tree $\bP_j$ together with the edge $e$ of minimal weight between a node of $\bP_j$ and a node out of $\bP_j$. }
Since the weights are chosen according to an atomless measure, the sequence $(\bP_j, j=1,\cdots,N)$ is a.s. well defined.
\begin{figure}[h!]
  \centerline{\includegraphics[width = 10cm]{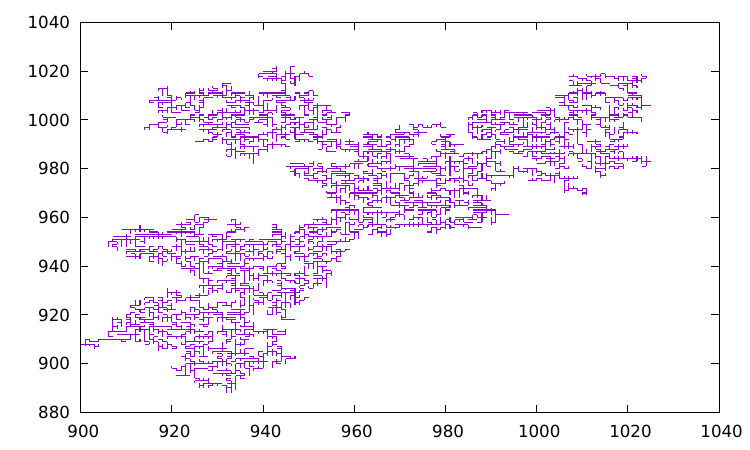}}
   \captionn{Prim's algorithm applied on $\Torus{2000}$ with i.i.d. weights ${\bf C}_e  \sim {\sf Uniform}(0,1)$, stopped when the tree reached $5000$ vertices. \label{FIG:Prim}}
 \end{figure}

\begin{Ques} Take a connected weighted graph $(G,\bC)$ and $r$ a fixed vertex of $V$. What is the distribution of ${\bf P}_n$? Of the process $({\bf P}_n)$?
\end{Ques}

When $K_n$ is the complete graph on $n$ vertices, the minimum spanning tree rescaled by $n^{1/3}$ converges in the Gromov-Hausdorff sense to a binary continuous random tree (Addario-Berry et al. \cite{ABBG}). For the moment, not much is known on the limiting tree. The analysis of this case relies on the fact that the connected components of  $({\bf P}_n)$ are related to the multiplicative coalescent (see Aldous \cite{aldous1997}, Broutin \& Marckert \cite{BroutinMarckert}).
\\
\indent Under some hypothesis, the total length of the edges in a minimal spanning tree admits a deterministic limit: it is $\zeta(3)$ (notably, as shown by Frieze \cite{Frieze}, when $G$ is the complete graph, and the edge weights are uniformly distributed on $[0,1]$), and a similar result occurs for regular enough weight distributions (Steele \cite{Steele}). Additional related results are numerous (see eg. Cooper et al. \cite{Cooper} and Janson \cite{Janson2,Janson3}).
\begin{rem}The probability $\P({\bf P}_n=t)$ is proportional to the number of induced permutation orders $\sigma$ giving $t$. It is possible to find a description of these permutations $\sigma$, by fixing first the relative order of the edges of $t$ (their Prim order); this provides a way to describe the $\sigma_e$ of the perimeter edges $e$ of $t$; however, the formula $\P({\bf P}_n=t)$ thus obtained, has a summation form which seems to be intractable.
\end{rem}

\subsubsection{Kruskal's algorithm:} Define a sequence of graphs $(\bK_j, j\geq 1)$. Take $\bK_1=(V,\varnothing)$ the graph with no edges and vertex set $V$, and $\bK_i=(V,E_i)$ the graph with edge set $E_i=E_{i-1}\cup \{e_{\sigma(i)}\}$ if this set of edges does not contain any cycle, or $E_{i}= E_{i-1}$ otherwise. Stop the construction at the MST $\bK$, which is the first graph $\bK_i$ containing $N-1$ edges.

At any fixed step $m$, $\bK_m$ is a forest; as time passes by, its connected components merge. In particular the connected component containing a given fixed vertex $r\in V$ has a non-decreasing size.

\NewModel{The Kruskal's component of a vertex}{Krusk}{Stop the construction in Kruskal's algorithm when the cluster containing $r$ has at least $n$ edges. Denote by $\bK_{{\sf size}\geq n}(r)$ the tree obtained (see Fig. \ref{fig:kru}).}
For any $t\in \TTTr{r}{G}$ with $|t|\geq n$, let $I_t:={\sf Induced}_G(\{u: d_G(u,t)\leq 1\})$ be the induced subgraph of $G$, formed by the nodes at distance $\leq 1$ from $t$. Each edge $e\in E(I_t)$ is either an edge of the tree,  a perimeter edge of $t$ (meaning that $\{e\} \cup E(t)$ is the edge set of a tree) or a ``cyclic edge'' meaning that $\{e\}\cup E(t)$ is the edge set of a graph with a (unique) cycle, denoted $C_t(e)$. Denote by $P_t$ the set of perimeter edges and ${\sf In}_t$ the set of cyclic edges.
\begin{pro} For any tree $t \in \cup_{k\geq n} \TTr{r}{G}k$
  \be
  \P(\bK_{{\sf size}\geq n}(r)=t)&=&\l|S_t\r| / |E(I_t)|!
 \ee
 where $S_t$ is the subset of the symmetric group ${\cal S}(E(t))$ composed of the permutations $\sigma$ that satisfy the following properties:\\
 $(a)$ $\sigma_e \geq \sigma_f, \forall f \in C_t(e), \forall e\in{\sf In}_t$,\\
 $(b)$ $\min\{\sigma_f ,f \in P_t\}\geq \max\{\sigma_e, e\in E(t)\}$.
 \\
 $(c)$ if one removes the edge $e$ of $t$ with largest label $\sigma(e)$, the connected component of $t$ containing the root has size $<n$.
\end{pro}
\begin{proof} It is simple to see that the realisation of the event $\{\bK_{{\sf size}\geq n}(r)=t\}$ depends only on the relative order of $\sigma$ on $E(I_t)$, which is uniform by symmetry.
  Now, by definition, removing the last edge added at $t$ must leave the connected component attached to the origin with a size $<n$ (condition $(c)$). Now, condition $(b)$ is needed: without it, some perimeter edges of $t$ would have been added before $t$ is completed. Condition $(a)$ translates the condition that an edge is not added if it forms a cycle: the other edges need to have been added before.
  \end{proof}
A variant consists in rejecting $\bK_{{\sf size}\geq n}(r)$ as long as its size is not exactly $n$ (meaning that we reassign weights to all edges). Denote by $\bK_{{\sf size}=n}(r)$ the result obtained. \\
\begin{figure}[h!]
  \centerline{
    \includegraphics[width = 7cm]{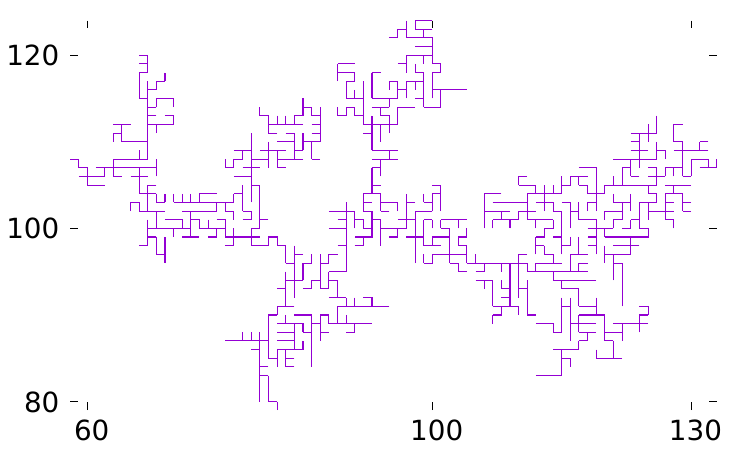}~~~~ \includegraphics[width = 7cm]{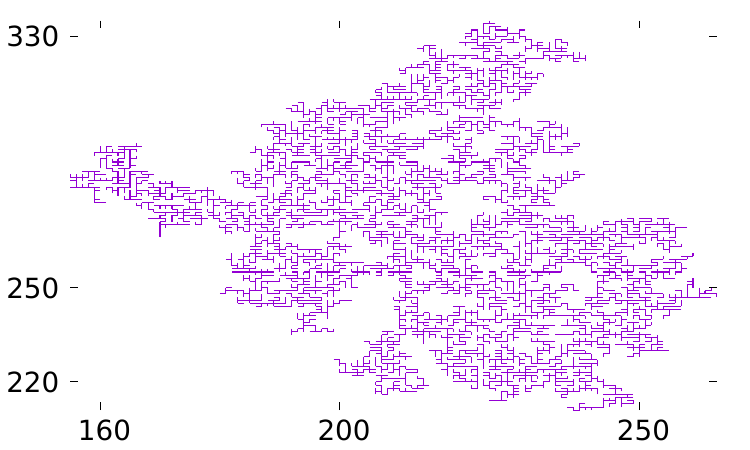}}
  \captionn{\label{fig:kru}Kruskal's algorithm on $\Torus{200}$ with a uniform random order of the edges, stopped when the tree containing the vertex $v=(100,100)$ has size in $[1000,1000(1+0.03)]$ (the output size is 1001). 20 simulations were needed before success (the size can jump over this interval during the construction process). The second construction is done on the $\Torus{500}$, and the algorithm is stopped when a tree with size in $[5000,5000 (1+0.01)]$ containing $v=(250,250)$ is obtained by the same method with uniform random order of the edges (four simulations were needed, with output size equals 5007).}
\end{figure}

\subsubsection{A few statistics on the Kruskal's trees}
We made some statistics to try to guess critical exponents in the case of the Kruskal's trees on the $\Torus{N}$. To get an efficient way to test the creation of cycles, we turned the Kruskal's forest into a forest of rooted trees as follows: at the beginning, all nodes are roots of trees reduced to a single vertex. Each time unit, a uniform vertex $u_k$ and a uniform direction $d_k$ (north, est, west, south) are chosen independently of the other choices. Let $v_k$ be the vertex at distance 1 from $u_k$ on the torus, such that $(u_k,v_k)$ has direction $d_k$. The oriented edge $(u_k,v_k)$ is then added to the ``forest'' if it does not create a cycle. In this case, in the rooted tree $(t,r)$ that contained $u_k$, the edges from $r$ to $u_k$ are oriented toward $u_k$ so that the new root of the new tree after this merging, is the root of the tree that contained $v_k$ beforehand. Hence, the component we are interested in is the rooted tree that contains a given node, chosen before the starting of the simulation. Since the diameter of a tree is at most twice the largest distance to the root, we expect the critical exponent to be independent of the choice of the root.

We fix a value $n$, and to not lose too much waiting time for a realisation of a tree with size exactly $n$, we wait till $|\bK_{{\sf size}\geq n}|\leq n(1+0.03)$, so that finally, this amounts to conditioning by $n\leq |\bK_{{\sf size}\geq n}|\leq n(1+0.03)$. All the simulations are done on $\Torus{700}$ which is, in practice large enough so that none of the thousands simulations we did got a width or a height of this size.
\ben
\begin{array}{|c|c|c|c|c|c|}
	\hline
	\textrm{Tree size}  & 4000 & 6000 & 8000 & 10000\pass\hline
	\textrm{Number of simulations}  & 10758 & 10679 & 10446 & 10365\pass\hline
\end{array}
\een
Again, we made two types of distance statistics, as in Section \ref{sc:SC}:  the \textbf{Euclidean width and height} $w(t)$ and $h(t)$ (number of vertical resp. horizontal row occupied), and random graph distance $ {\bf D}( t) = d_{ t}({\bf r},{\bf v})$ between the root of the tree and  a random node in the tree. To make the statistics, for each simulation, we used both values $w(t)$ and $h(t)$, and sample 10 random nodes ${\bf v}$ for each such tree. 
We use the same methods as in  Section \ref{sc:SC} to evaluate the more plausible values of $\alpha$ and $\beta$ for which
$w(t_n)/n^{\alpha}$ and $ d_{ t_n}({\bf r},{\bf v})/n^\beta$ would converge in distribution, in view of our samples.
\ben
\begin{array}{|c|c|c|c|c|c|}
	\hline
	\textrm{Number of nodes}  & 4000 & 6000 & 8000 & 10000\pass\hline
	\textrm{Empirical mean of the width}  & 112.96 & 140.07 & 163.51 & 183.95\pass\hline
	\textrm{Empirical median of the width}  & 111.00 & 138.00 & 161.00 & 180.00\pass\hline
	\textrm{Empirical mean of $d({\bf r},{\bf v})$}  & 174.77 & 225.08 & 269.44 & 311.91\pass\hline
	\textrm{Empirical median of $d({\bf r},{\bf v})$}  & 165.00 & 213.00 & 254.00 & 295.00\pass\hline
\end{array}
\een

\ben
\begin{array}{|c|c|c|c|c|}
	\hline
	\textrm{$(n,m)$}  & (4000, 6000) & (6000, 8000) & (8000, 10000)\pass\hline
	\textrm{Estimation of $\alpha$ (mean) }  & 0.530 & 0.538 & 0.528\pass\hline
	\textrm{Estimation of $\alpha$ (median) }  & 0.537 & 0.536 & 0.500\pass\hline
	\textrm{Best fit decile $\alpha$  }  & 0.528 & 0.538 & 0.516\pass\hline
	\textrm{Estimation of $\beta$ (mean) }  & 0.624 & 0.625 & 0.656\pass\hline
	\textrm{Estimation of $\beta$ (median) }  & 0.630 & 0.612 & 0.671\pass\hline
	\textrm{Best fit decile $\beta$}  & 0.623 & 0.628 & 0.653\pass\hline
\end{array}
\een

\begin{figure}[h!] \centerline{\includegraphics[width = 6cm]{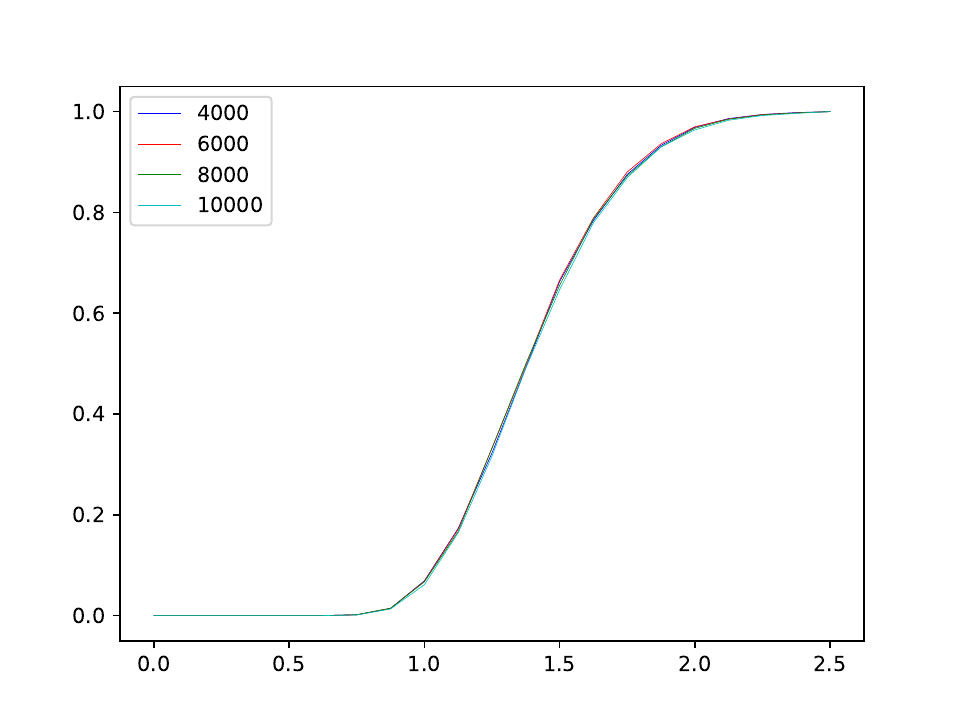}\includegraphics[width = 6cm]{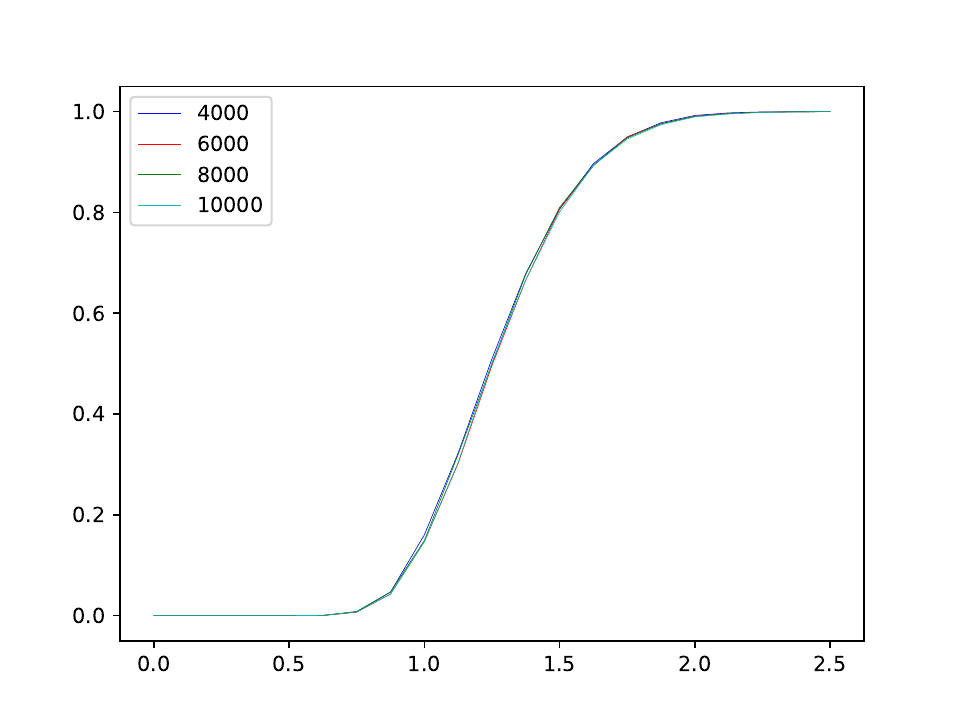}\includegraphics[width = 6cm]{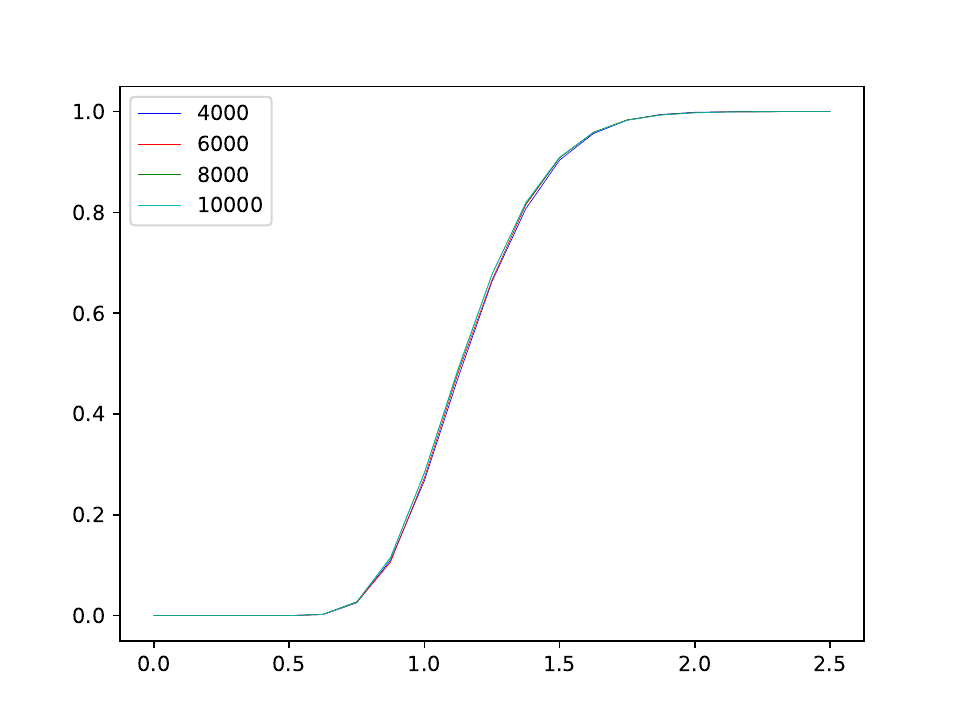}}
  \centerline{\includegraphics[width = 6cm]{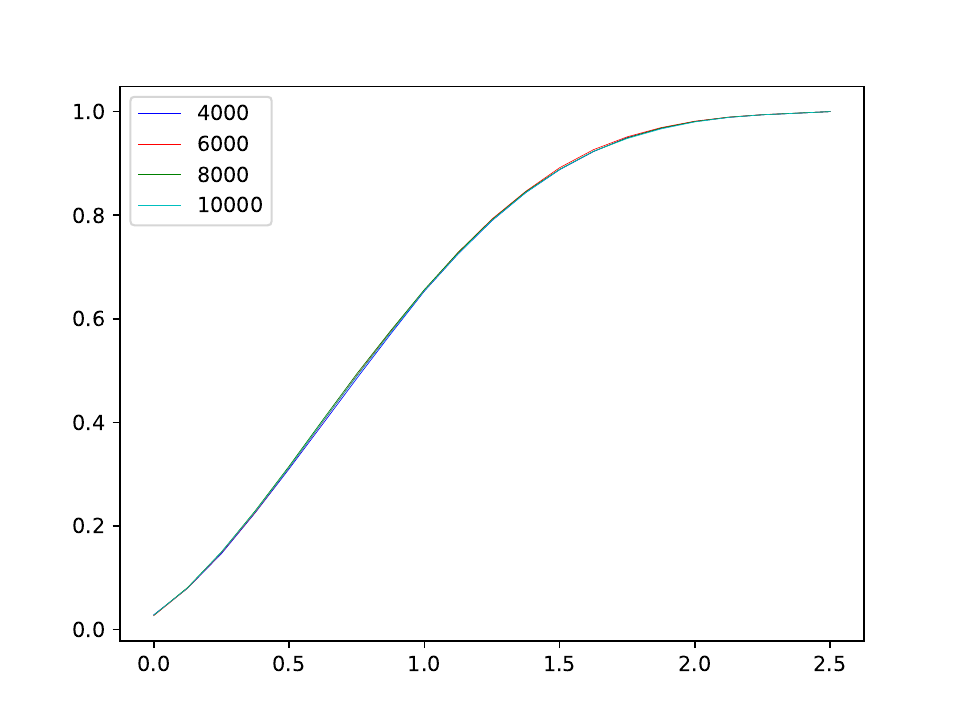}\includegraphics[width = 6cm]{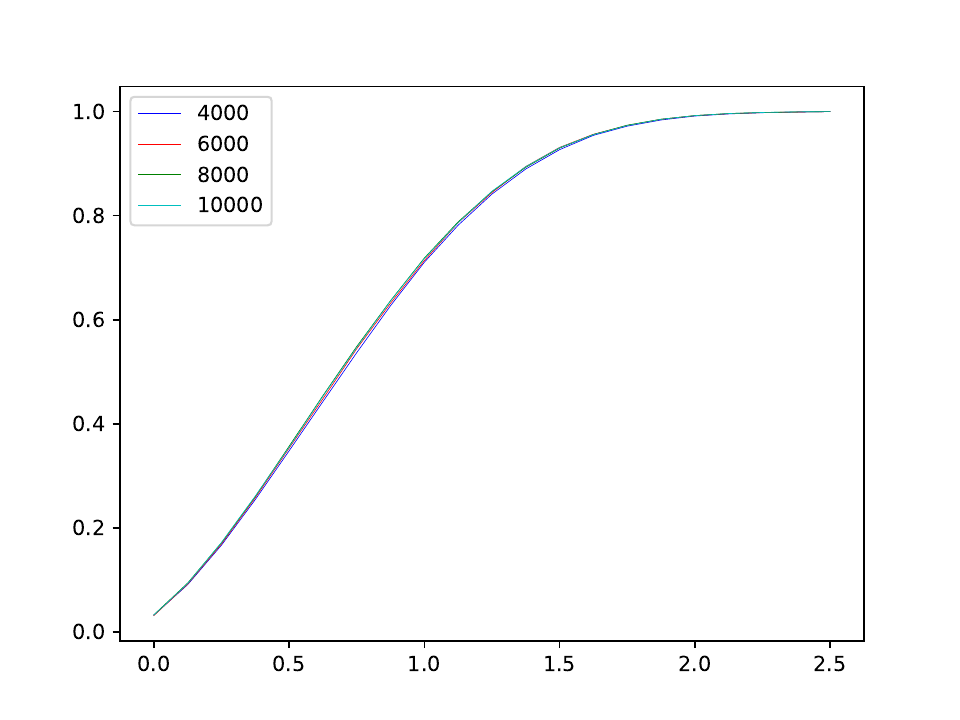}\includegraphics[width = 6cm]{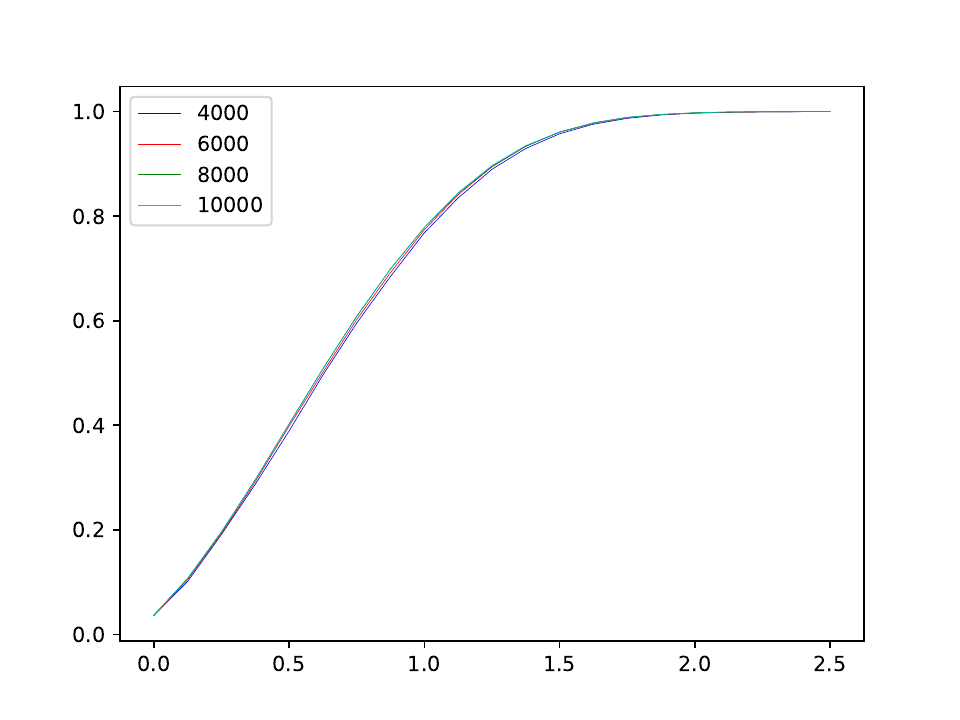}}
  \captionn{\label{fig:dladqdqtf}On the top, (interpolated) empirical cumulative function of $w(t_n)/n^{\alpha}$ for $\alpha$ being respectively 0.52, 0.53 and 0.54. On the bottom, (interpolated)  empirical cumulative function of $d_{t_n}({\bf r},{\bf v})/n^{\beta}$ for $\beta$ being respectively 0.63, 0.64 and 0.65. }
\end{figure}
\begin{conj}  $w(t_n)/n^{\alpha}$  converges in distribution for some $\alpha\in[0.50,0.55]$.
  \end{conj}
The simulations suggest that either $\beta$ exists but the sizes of the simulated trees are not large enough to estimate it, or there does not exist any such $\beta$ (a correction term like $(\log n)^\gamma$ may be needed). However, the curves \eref{fig:dladqdqtf} show that the empirical cumulative function of $ d_{ t_n}({\bf r},{\bf v})/n^\beta$ are really close for the simulated $n$ for some $\beta$, so that it can be guessed that $d_{t_n}({\bf r},{\bf v})/a_n$ should converge for some $(a_n)$ to a non-trivial limit.
\color{black}

\NewModel{The minimal weighted subtree}{MWS} 
{Consider the subtree tree ${\bf t}_n(\mu)$ of $G$ with $n$ edges and minimal weight among those with $n$ edges.}

The problem consisting in finding the minimal weighted subtree of size $n$ is called the $k$-cardinality tree problem (see e.g. Chimani et al. \cite{CKLM} and reference therein): it is a NP-complete problem, and we gave up on the idea of providing pictures for this model.

\begin{rem}
  Here, the distribution of ${\bf t}_n(\mu)$ depends on $\mu$ not only on the relative order of edges.
\end{rem}

Other optimisation problems like this one exist in the literature, for example, the Steiner tree problem (which amounts to finding the tree with minimal weight connecting a subset of nodes $U\subset V$ in a graph) and its numerous variants, for which the nodes are also weighted, for example, the node-weighted Steiner tree problem (Buchanan et al. \cite{BWB}), the edge capacitated Steiner tree problem (see Bentz et al. \cite{BCH} in which additional constraints on the tree are added), the minimum routing cost spanning-trees, which amounts to optimizing the mean distance between pairs of uniform random nodes (Wu et al. \cite{Wu}); a similar type of problem ``Optimum Communication Spanning-Trees'', introduced by Hu \cite{Hu} (see recent developments in Zetina et al. \cite{Zetina}, Luna-Mota \cite{CLM}).

\subsubsection{First passage percolation}

Again consider the same model of weighted graph $(G,\bC)$ and a distinguished vertex $r$.
Now, with each node $w\neq r$ associate the (a.s. well defined) path $\bL_w$ from $r$ to $w$ with minimal weight $\bC(w)$ (sum of the weights of the edges belonging to the path). The union of the paths $\bT(\mu):=\cup_{w\neq r} \bL_w$  forms a.s. a tree (it is connected and acyclic with probability 1, since a cycle implies that two different paths have the same weight).

\NewModel{The first passage percolation tree}{FPP}{Denote by $\bT_n(\mu)$ the tree formed by the union of the paths from $r$ to the $n$ nodes (including $r$) with the smallest weights. }

It it quite simple to find graphs and distributions $\mu$ for which $\bT_n(\mu)$ is not uniform in $\TTr{r}{G}{n}$. In $\Z^2$, there exists a limit shape theorem (Cox-Durrett shape theorem, see Auffinger et al. \cite[Section 2]{ADH} for this theorem, and an overview of last passage percolation problems).

\color{black}
\begin{figure}[h!]
  \centerline{\includegraphics[width = 8cm]{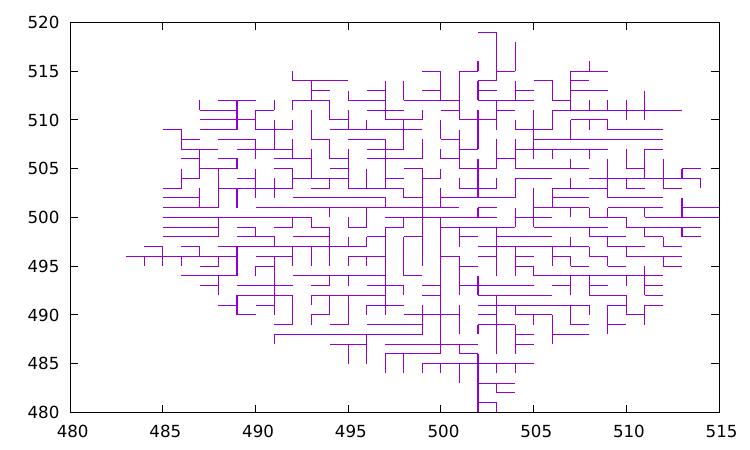}\includegraphics[width = 8cm]{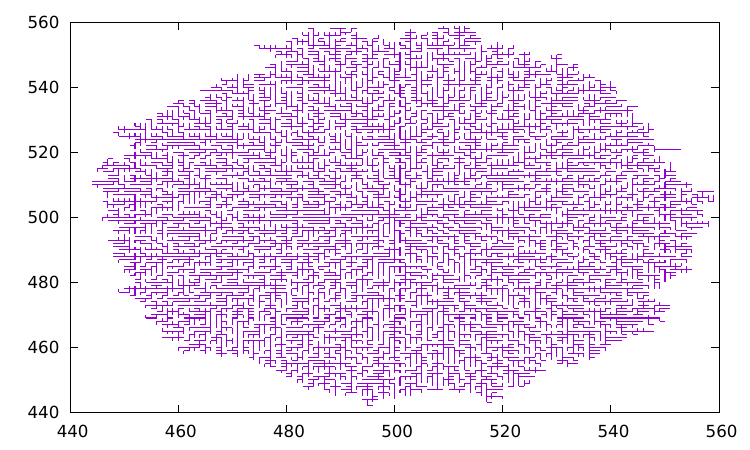}\\}
  \centerline{\includegraphics[width = 8cm]{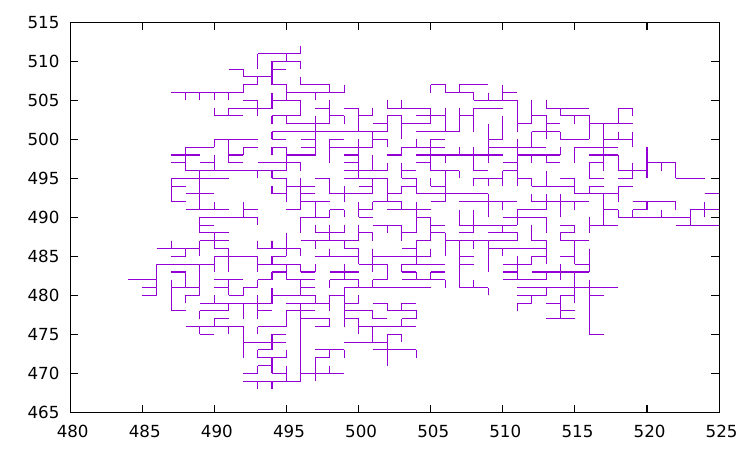}\includegraphics[width = 8cm]{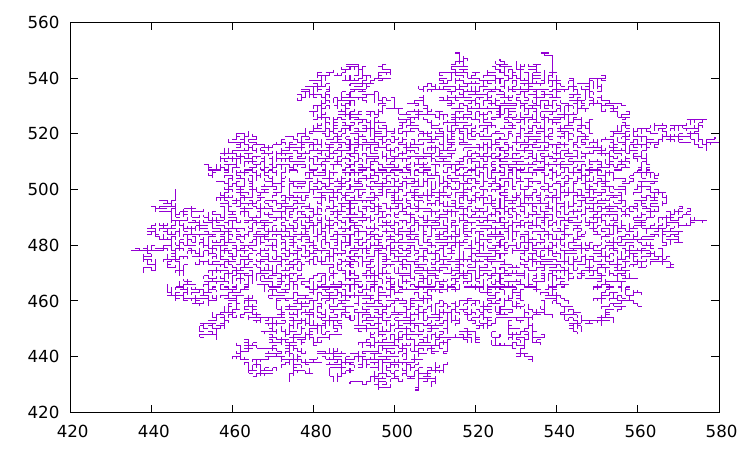}}
  \captionn{First passage percolation on the $\Torus{1000}$. On the first line, the ${\bf C}_e$ are i.i.d. uniform on $[0,1]$. Both trees are done in the same environment and have sizes 1000 and 10000.  On the second line, the weights are distributed as $1/E^5$ where $E$ has the exponential distribution with parameter 1.
  	\label{FIG:simufpp}} 
\end{figure}

\section{Random subtrees of a tree $T$ }\label{subtreetree}
\label{sec:exact_samp}

The case where $T$ is an infinite $d$-ary tree is discussed few lines above \Cref{sec:STC}. Here we focus on the case where  $T$ is a finite tree.
\subsection{Combinatorial considerations}
\label{eq:grfsdf}

For a given finite tree $T$, the polynomial $\Phi_T(x)=\sum_{k\geq 1} x^k |\TT{T}{k}|$ is called the subtree polynomial of $T$. Due to the decomposition of trees at their root, the computation of $\Phi_T$ is much less expensive than in the case of graphs, even in the weighted case (see  Yan \&  Yeh \cite{YAN2006256}). This implies that uniform sampling of subtrees of a given size of a tree can be done exactly, in principle, even on trees of big size, just by counting the number of subtrees with size $k$ containing a given vertex, and making some decomposition (see also Brown \& Mol \cite{BrMo} and reference therein).

\subsection{Exact generation of uniform or conditionally uniform subtrees} \label{sec:egu}

The absence of cycles in $G$ simplifies the implementation of the Markov chains we introduced in \Cref{sec:twoergo}. Moreover, it is easy to define monotone transition matrices for the inclusion order, so that coupling from the past techniques will be possible. In this section,  $G=(T,r)$ is a rooted tree (and we keep the notation $T=(V,E)$).\par
Recall the (rooted) kernel $K^{\Kerref{labelf}}_{r}$ in $\TTTr{r}{G}$ introduced in  Section \ref{sec:RV}, defined using a sequence of parameters $[(\pp_i,\qq_i,\rr_i), 1\leq i \leq |V|]$. 
Here the graph $G=(T,r)$ so that the attempt of addition of a new edge ${\bf e}$ to the current tree $X_0$ will never create any cycle.\\

\begin{pro}\label{pro:dqgyjr}
Consider a sequence $[(\pp_i,\qq_i,\rr_i), 1\leq i \leq |V|]$ such that $\rr_i>0$ for all $i$. For a tree $T$, consider $\bt$ taken under the invariant distribution of $K_{r}^{\Kerref{labelf}}$, for a fixed $r\in T$.  We then have 
\[
	\P(\bt=t) =\nu_1(T)\prod_{i=2}^{|t|} \left( \frac{\pp_{i-1}}{\rr_{i}} \right) 1_{t \in \TTTr{r}{T}}
\]
where $\nu_1(T)$ denotes $\nu_1$ in \Cref{pro:qsdz2} (and Rem. \ref{eq:nu}) applied to  $T$. \end{pro} 
Recall that given $T$ and $n$, ${\cal L}({\bf t}~|~|\bt|=n)$ is uniform in $ \TTr{r}{T}{n}$. This is independent of $(\pp_i)$ and $(\rr_i)$, which leaves a degree of freedom to bias the size of ${\bf t}$.
\Cref{FIG:simu2} shows some simulations obtained using this procedure on a uniform spanning tree of the torus.
\begin{rem}
  A particular case is obtained when $\pp_{i-1}=\rr_i$ for all $i\in \cro{2,|V|}$, since this reduces to the sampling of a uniform subtree of $T$.
\end{rem}

\paragraph{A coupling from the past for $K^{\Kerref{labelf}}_{r}$.}
Consider the following condition:
\be
{\sf Hypothesis~M}:
&&\pp_1\leq \pp_2\leq \cdots\leq \pp_{|V(T)|-1},\\
&& \rr_2\geq \cdots \geq \rr_{|V(T)|}.
\ee
Since $\pp_i+\rr_i+\qq_i=1$, it is also required that $\rr_k+\pp_k\leq 1$.
In other words, the bigger the tree is, the faster it grows, and the smaller the tree is, the faster it shrinks.

We will show that under the {\sf Hypothesis~M}, it is possible to couple the Markov chain under the transition matrix  $K^{\Kerref{labelf}}_{r}$ so that it is monotone for the inclusion partial order, where for $t,t'\in \TTTr{r}{T}$ we say that $t\preceq t'$ if $E(t)\subset E(t')$.  This partial order possesses as least element the tree $\underline{t}=\{r\}$ (reduced to its root), and as greatest element,  the complete tree $\overline{t}=T$. 

For more information on the coupling from the past when the space state possesses a partial order with a unique minimal and a unique maximal element, we refer to \cite{PW96,PW98b}.\\
The realization of the coupling is done according to the following lines.
First, define a function
\[\app{f}{\TTTr{r}{T}\times E(T)\times [0,1]}{\TTTr{r}{T}}{(t,e,{\sf v})}{\begin{cases}
		\Add(t,e)& \text{ if } {\sf v}\leq \pp_{|t|},\\
		\Rem_r(t,e)& \text{ if } {\sf v}\geq 1-\rr_{|t|},\\
		t & \text{ if }\pp_{t}<{\sf v}< 1-\rr_{|t|}.
	\end{cases}}	
    \]
Consider a realization of a sequence of i.i.d. vectors $((\textbf{e}_k,\textbf{v}_k): k\in \Z)$ where $\textbf{e}_k\sim \uniform{(E(T))}$ is independent of $\textbf{v}_k\sim \uniform{[0,1]}$. Now, set  \[f_k(\cdot)= f(\cdot,\textbf{e}_k,\textbf{v}_k)\]
and for every pair of integers $(k_1,k_2)$ such that $k_1<k_2$ we consider \[F_{k_1}^{k_2}(t) = f_{k_2}\circ f_{k_2-1}\circ \dots \circ f_{k_1}(t).\]
For a reader not familiar with this kind of considerations, there are two key points:\\
$\bullet$ firstly, for any $t\in \TTTr{r}{T}$, the process $(F_{0}^k(t),k\geq 0)$ has the distribution of a Markov chain with kernel $K^{\Kerref{labelf}}_{r}$ with initial state, the tree $t$,\\
$\bullet$  and secondly,  a natural coupling is provided since the complete family $[(F_{0}^k(t),k\geq 0),t\in \TTTr{r}{T}]$ can be constructed altogether simultaneously since they are built using the same source of randomness. \\
The ${\sf Hypothesis~M}$ ensures the monotonicity of the chain: a direct consequence of this hypothesis and of the definition of $f$, is that, for every $(e,v)\in E(G) \times [0,1]$ and any $t,t'\in \TTTr{r}{T}$
\[ t\preceq t' \imp f(t,e,v) \preceq f(t',e,v),\] 
and therefore, for every $k_1\leq k_2$, $F_{k_1}^{k_2}(t) \preceq  F_{k_1}^{k_2}(t')$ too. In particular, for each tree $t \in \TTTr{r}{G}$,
\[F_{k_1}^{k_2}(\underline{t}) \preceq  F_{k_1}^{k_2}(t) \preceq F_{k_1}^{k_2}(\bar{t}).\] Hence $F_{t_1}^{t_2} (\underline{t}) = F_{t_1}^{t_2} (\overline{t})$ iff $ F_{t_1}^{t_2} (t)$ is the same for all $t\in \TT{r}{T}$.

We recall the monotone coupling from the past algorithm.\\
\ligne
\underline{{\sf Monotone coupling from the past:}}\\
$\bullet$ $iter := 2$ (or another free parameter $>1$)\\
$\bullet$ $s:=1$\\
$\bullet$ {\bf While $F_{-s}^0(\underline{t}) \not= F_{-s}^0(\overline{t})$ do}\\ 
\phantom{qsdsdqzs}\quad~~~$s := s  \times iter$\\
$\bullet$ {\bf End while}\\
$\bullet$ {\bf Return $F_{-s}^0(\underline{t})$}\\
\ligne

The backward chain $(F_{-s}^{0}(\cdot): s\in \N)$ is  indirectly related to the forward chain $(F_0^{s}(\cdot): s\in \N)$;  set
\[\overrightarrow{\tau} = \inf\{ s\geq  0: F_0^s(\underline{t})=F_0^s(\overline{t}) \} ~~\textrm{ and }~~\overleftarrow{\tau} = \max\{ s\leq 0 : F_{-s}^0(\underline{t})=F_{-s}^0(\overline{t}) \}\]
the so-called forward and backward coupling time, respectively. 
As stated in  \cite[P. 21]{PW96}, $\overrightarrow{\tau}$ and $\overleftarrow{\tau}$ have the same distribution.\\
We will give some bounds on the forward coupling time. 

  For each  ``time'' $s\geq 0$, define a colouring $C_s=(C_s(w), w\in V(T))$ of the vertex set of $T$ as follows: \\
  -- if $w \in V(F_0^s(\underline{t}))$, set $C_s(w)= {\sf red}$,\\
  -- if $w  \in V(F_{0}^s(\overline{t}))\setminus V(F_{0}^s(\underline{t}))$, set $C_s(w)= {\sf white}$,\\
  -- otherwise set $C_s(w)={\sf black}$.\\
  At time 0, $C_0(w)={\sf white}$ for all nodes of $T$, except the root which is ${\sf red}$. For any $s$, the set of {\sf red} vertices are those of the ``minimal tree'', $F_{0}^s(\underline{t})$, while  those of the ``maximal tree'' $F_{0}^s(\overline{t})$ are in the union of the {\sf red} and {\sf white} nodes. 
The coupling time $\overrightarrow{\tau}$ coincides with the time where there is no {\sf white} vertex left. 
\begin{figure}[h]
	\centering
	\begin{subfigure}{0.47\textwidth}
		\begin{tikzpicture}[scale= 0.5,>=stealth']
			
			\tikzstyle{level 1}=[level distance=10mm,sibling distance=80mm]
			\tikzstyle{level 2}=[level distance=10mm,sibling distance=40mm]
			\tikzstyle{level 3}=[level distance=10mm,sibling distance=20mm]
			\tikzstyle{level 4}=[level distance=10mm,sibling distance=10mm]
			\tikzstyle{level 5}=[level distance=10mm,sibling distance=5mm]
			
			\node [arn_r] {}
			child{ node [arn_r] {} 
				child{ node [arn_r] {} 
					child{ node [arn_r] {}  
						child{ node [arn_x] {} 
							child{ node [arn_x] {}} 
							child{ node [arn_n] {}}
						}
						child{ node [arn_x] {}
							child{ node [arn_n] {}}
							child{ node [arn_n] {}}
						}  
					}                          
					child{ node [arn_x] {}
						child{ node [arn_x] {} 
							child{ node [arn_n] {}} 
							child{ node [arn_n] {}}
						}
						child{ node [arn_n] {}
							child{ node [arn_n] {}}
							child{ node [arn_n] {}}
						}                 
					}           
				}
				child{ node [arn_x] {}
					child{ node [arn_x] {}  
						child{ node [arn_n] {} 
							child{ node [arn_n] {}} 
							child{ node [arn_n] {}}
						}
						child{ node [arn_n] {}
							child{ node [arn_n] {}}
							child{ node [arn_n] {}}
						}  
					}                          
					child{ node [arn_x] {}
						child{ node [arn_n] {} 
							child{ node [arn_n] {}} 
							child{ node [arn_n] {}}
						}
						child{ node [arn_n] {}
							child{ node [arn_n] {}}
							child{ node [arn_n] {}}
						}                 
					}           
				}                            
			}
			child{ node [arn_r] {} 
				child{ node [arn_r] {} 
					child{ node [arn_r] {}  
						child{ node [arn_x] {} 
							child{ node [arn_x] {}} 
							child{ node [arn_n] {}}
						}
						child{ node [arn_n] {}
							child{ node [arn_n] {}}
							child{ node [arn_n] {}}
						}  
					}                          
					child{ node [arn_x] {}
						child{ node [arn_n] {} 
							child{ node [arn_n] {}} 
							child{ node [arn_n] {}}
						}
						child{ node [arn_n] {}
							child{ node [arn_n] {}}
							child{ node [arn_n] {}}
						}                 
					}           
				}
				child{ node [arn_r] {}
					child{ node [arn_r] {}  
						child{ node [arn_r] {} 
							child{ node [arn_r] {}} 
							child{ node [arn_r] {}}
						}
						child{ node [arn_r] {}
							child{ node [arn_n] {}}
							child{ node [arn_n] {}}
						}  
					}                          
					child{ node [arn_x] {}
						child{ node [arn_n] {} 
							child{ node [arn_n] {}} 
							child{ node [arn_n] {}}
						}
						child{ node [arn_n] {}
							child{ node [arn_n] {}}
							child{ node [arn_n] {}}
						}                 
					}           
				}                            
			}
			; 
		\end{tikzpicture}
		\caption{Intermediate phase}
	\end{subfigure}
	\begin{subfigure}{0.47\textwidth}
		\centering
		\begin{tikzpicture}[scale= 0.5,>=stealth']
			
			\tikzstyle{level 1}=[level distance=10mm,sibling distance=80mm]
			\tikzstyle{level 2}=[level distance=10mm,sibling distance=40mm]
			\tikzstyle{level 3}=[level distance=10mm,sibling distance=20mm]
			\tikzstyle{level 4}=[level distance=10mm,sibling distance=10mm]
			\tikzstyle{level 5}=[level distance=10mm,sibling distance=5mm]
			
			\node [arn_r] {}
			child{ node [arn_r] {} 
				child{ node [arn_r] {} 
					child{ node [arn_n] {}  
						child{ node [arn_n] {} 
							child{ node [arn_n] {}} 
							child{ node [arn_n] {}}
						}
						child{ node [arn_n] {}
							child{ node [arn_n] {}}
							child{ node [arn_n] {}}
						}  
					}                          
					child{ node [arn_r] {}
						child{ node [arn_r] {} 
							child{ node [arn_r] {}} 
							child{ node [arn_n] {}}
						}
						child{ node [arn_r] {}
							child{ node [arn_n] {}}
							child{ node [arn_n] {}}
						}                 
					}           
				}
				child{ node [arn_r] {}
					child{ node [arn_r] {}  
						child{ node [arn_n] {} 
							child{ node [arn_n] {}} 
							child{ node [arn_n] {}}
						}
						child{ node [arn_n] {}
							child{ node [arn_n] {}}
							child{ node [arn_n] {}}
						}  
					}                          
					child{ node [arn_n] {}
						child{ node [arn_n] {} 
							child{ node [arn_n] {}} 
							child{ node [arn_n] {}}
						}
						child{ node [arn_n] {}
							child{ node [arn_n] {}}
							child{ node [arn_n] {}}
						}                 
					}           
				}                            
			}
			child{ node [arn_r] {} 
				child{ node [arn_r] {} 
					child{ node [arn_r] {}  
						child{ node [arn_r] {} 
							child{ node [arn_n] {}} 
							child{ node [arn_n] {}}
						}
						child{ node [arn_n] {}
							child{ node [arn_n] {}}
							child{ node [arn_n] {}}
						}  
					}                          
					child{ node [arn_n] {}
						child{ node [arn_n] {} 
							child{ node [arn_n] {}} 
							child{ node [arn_n] {}}
						}
						child{ node [arn_n] {}
							child{ node [arn_n] {}}
							child{ node [arn_n] {}}
						}                 
					}           
				}
				child{ node [arn_r] {}
					child{ node [arn_r] {}  
						child{ node [arn_r] {} 
							child{ node [arn_r] {}} 
							child{ node [arn_n] {}}
						}
						child{ node [arn_n] {}
							child{ node [arn_n] {}}
							child{ node [arn_n] {}}
						}  
					}                          
					child{ node [arn_r] {}
						child{ node [arn_r] {} 
							child{ node [arn_n] {}} 
							child{ node [arn_n] {}}
						}
						child{ node [arn_n] {}
							child{ node [arn_n] {}}
							child{ node [arn_n] {}}
						}                 
					}           
				}                            
			}
			; 
		\end{tikzpicture}
		\subcaption{Merged state}
	\end{subfigure}
	\captionn{\label{fig:sqd} A possible evolution of the coupled forward Markov chain, on a binary tree $T$.}
\end{figure}\\
For $t\in \TTTr{r}{T}$ define the set of perimeter sites of $t$ as $V_p(t)=\{ v\in V: d_T(v,t)=1\}$.
The set of leaves of $t$ as $V_\ell(t)=\{ v\in V: v \text{ leaf of } t \}$. Also, define the maximal sizes of the perimeter and leaves sets for a tree with $k$ nodes as
\begin{align}
		\overline{V_p(k)} &= \max\{ |V_p(t)|:  t\in \TTr{r}{T}{k}\}\\
		\underline{V_\ell(k)} &= \min\{|V_\ell(t)|: t\in \TTr{r}{T}{k}\}
\end{align}
\begin{pro} \label{pro:tyher}
  Suppose {\sf Hypothesis~M} holds. If for all $i\in \cro{2,|V(T)|-1}$,
    $\pp_i/\rr_i\leq c \underline{V_\ell(i)}/\overline{V_p(i)}$  then 
	\[
		`E(\overrightarrow{\tau}) = `E(\overleftarrow{\tau}) \leq 
                  (N-1)\sum_{j=2}^{N}\frac{(j-1)}{\rr_j\underline{V_\ell(j)}}\frac{c^{j-1}-1}{c-1}.\]
In particular if $T$ is a complete $d$-ary tree with height $h$, then with $N = (d^{h+1}-1)/(d-1)$ vertices, 
\ben\label{eq:djyqsd}
`E(\overrightarrow{\tau}) \leq \frac{(N-1)c}{\pp_2(1-c)(d-1)} \left((N-1)-\frac{c-c^N}{1-c}\right)
\een
\end{pro}
\begin{proof}
For every $t \in \TTTr{r}{T}$ denote by 
$\tau(t) = \inf\{ s\geq 0: F_0^s(t)= \underline{t} \}$
the hitting time of the tree $\underline{t}$.
From the coupled forward chain we have,  
\be
\max\{\tau(t),~~t \in \TTTr{r}{T}\} =:\tau(T) \geq \ar{\tau},~~\textrm{ a.s.}
\ee

Throughout the proof, we write $N$ instead  of $|V(T)|$. Using the Markov property we get
\begin{align}\label{rec:eq}
		&\frac{|V_\ell(T)| }{N-1} \,\rr_{N}\,`E(\tau(T))= 1 +    \frac{\rr_{N}}{N-1}\sum_{e\in V_\ell(T)}  `E(\tau(T\setminus\{e\})),
\end{align}
and for $t\in \TTr{r}{T}{k}$ for $k\in \cro{2,N-1}$
\ben \label{rec:eq2}
\left(\frac{\pp_{k}|V_p(t)|}{N-1} +\frac{\rr_{k}|V_\ell(t)|}{N-1} \right) `E(\tau(t))= 1 +  \frac{\pp_{k}}{N-1}\sum_{e\in V_p(t)}  `E(\tau(t\cup\{e\})) +  \frac{\rr_{k}}{N-1}\sum_{e\in V_\ell(t)} `E(\tau(t\setminus\{e\})).
\een
Call $E_k = \max\{ `E(\tau(t)): t\in \TTr{r}{T}{k} \}$ and $\Delta_k = E_k-E_{k-1}$. Notice that $E_{1} = 0$. Bounding each term $`E(\tau(T\setminus\{e\}))$ in the right hand side of \eqref{rec:eq} by $E_{N-1}$ and by noticing that $E_N=`E(\tau(T))$ we obtain
\ben\label{eq:egzr}
\Delta_{N} \leq  (N-1)/(\rr_N \underline{V_\ell(N)}).\een 
For $k\in \cro{2,N-1}$ fix $t_k$ one of the trees attaining $E_k$. Now, consider \eqref{rec:eq2} applied to $t_k$ and bound each $`E(\tau(t\cup\{e\}))$ and $`E(\tau(t\setminus\{e\}))$ in the right hand side respectively by $E_{k+1}$ and $E_{k-1}$. 
		\begin{align}\label{rec:eq3}
			\left( \pp_{k}\frac{|V_p(t_k)|}{N-1}+ \rr_{k}\frac{|V_\ell(t_k)|}{N-1}\right) E_{k}&\leq  1 +  \frac{\pp_{k}|V_p(t_k)|}{N-1} E_{k+1} +  \frac{\rr_{k}|V_\ell(t_k)|}{N-1}  E_{k-1}\\
	\label{rec:eqqsdqd3}		\implies\qquad \frac{\rr_{k}|V_\ell(t_k)|}{N-1}\Delta_{k} &\leq \frac{\pp_{k}|V_p(t_k)|}{N-1}  \Delta_{k+1} + 1
                \end{align}
Therefore for $k\in \cro{2,N-1}$, using the definition of $\overline{V_p(k)}$, $\underline{V_\ell(k)}$ and the hypothesis that $\pp_k/\rr_k \leq \underline{V_\ell(k)}/\overline{V_p(k)} $ one obtains
      	\begin{align}
	\Delta_{k} &\leq \frac{\pp_{k}|V_p(t_k)|}{\rr_{k}|V_\ell(t_k)|}  \Delta_{k+1} + \frac{N-1}{\rr_{k} |V_\ell(t_k)|}
	\leq \frac{\pp_{k}\overline{V_p(k)}}{\rr_{k}\underline{V_\ell(k)}}  \Delta_{k+1} + \frac{N-1}{\rr_{k} \underline{V_\ell(k)}}
	\leq c\Delta_{k+1} + \frac{N-1}{\rr_{k} \underline{V_\ell(k)}}\label{rec:k}.
      \end{align}
  By repeatedly applying \eqref{rec:k} and finally \eqref{eq:egzr} one obtains that for all $k\in \cro{2,N}$ one has
  \[\Delta_k \leq \sum_{j=k}^{N}c^{j-k}\frac{N-1}{\rr_j \underline{V_\ell(j)}}. \]
To conclude notice that $`E(\tau(T))= E_{N} = \sum_{k=2}^{N} \Delta_k$ and therefore this gives
\be
`E(\overrightarrow{\tau}) \leq E_N = \sum_{k=2}^{N} \sum_{j=k}^{N}c^{j-k}\frac{N-1}{\rr_j \underline{V_\ell(j)}}= \sum_{j=2}^{N}\left( \sum_{k=2}^{j}c^{j-k}\right)\frac{N-1}{\rr_j \underline{V_\ell(j)}}\leq (N-1)\sum_{j=2}^{N}\frac{1}{\rr_j\underline{V_\ell(j)}}\frac{c^{j-1}-1}{c-1}.
\ee
To conclude the second part on the $d$-regular tree we use that by {\sf Hypothesis~M}, $\pp_j$ is non-decreasing, that $\pp_i/\rr_i\leq c \underline{V_\ell(i)}/\overline{V_p(i)}$ and that the infinite $d$-regular tree satisfies $\overline{V_p(i)} = (i+1)(d-1)-1$ which is bigger than $i(d-1)$ for $d>1$.
\end{proof}

\subsection{Leaf evaporation}
\label{sec:LEV}

\NewSModel{Uniform Leaf evaporation}{ULE}{
Take a tree $T$ with $N$ nodes, and define $({\sf LeafEvaporation} (T , k),0\leq k \leq N-1)$ as follows: ${\sf LeafEvaporation} (T , 0)=T$, and for $k>0$, ${\sf LeafEvaporation} (T , k)$ is obtained by the removal of a uniform leaf of ${\sf LeafEvaporation} (T , k-1)$ (so that $k$ counts the number of evaporated edges).}
\begin{rem}[rooted versus unrooted case]
  There are two natural variants of this algorithm depending on whether we work with unrooted tree $T$, in which case, all nodes of degree 1 are leaves, or if $T=(T,r)$ is a rooted tree, in which the root $r$ is never considered as a leaf (this is the standard convention).
\end{rem}
We consider the rooted case here: the root $r$ is never considered as a leaf. 
Any history of leaf evaporation can be encoded by labelling the edges of the initial tree by the date of evaporation of the leaves from 1 to $|T|-1$. 
For $t \in \TTr{r}{G}{n}$, consider the set $H[T,t]$ of labelling of the edges of $T\setminus t$ by the integers between $1$ and $|T\setminus t|$, such that, the labels of the edges on any injective path from any leave of $T$ to $t$ are increasing. The following result describes the law of the remaining tree after $N-n$ leaf evaporations:
\begin{pro} For $t \in \TTr{r}{G}{n}$,
\[\P\l[ {\sf LeafEvaporation} (T , N-n)=t\r]=\sum_{h \in H[T,t]} \prod_{ x =0}^{|T\setminus t|-1}\frac{1}{|\partial (T \setminus \{e: h(e) \leq x \}|} \]
\end{pro}
\begin{proof}For each history, at each step, the probability to remove a given leaf is the inverse of the current number of leaves.\end{proof}
  
\begin{figure}[h!]
  \centerline{\includegraphics[width=8cm]{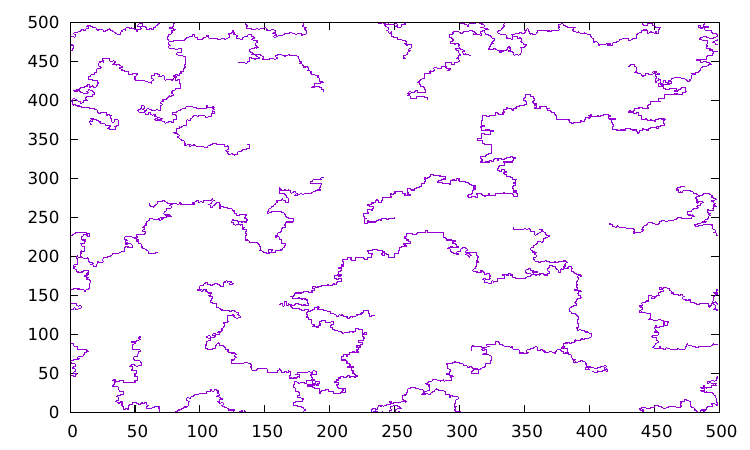}\includegraphics[width=8cm]{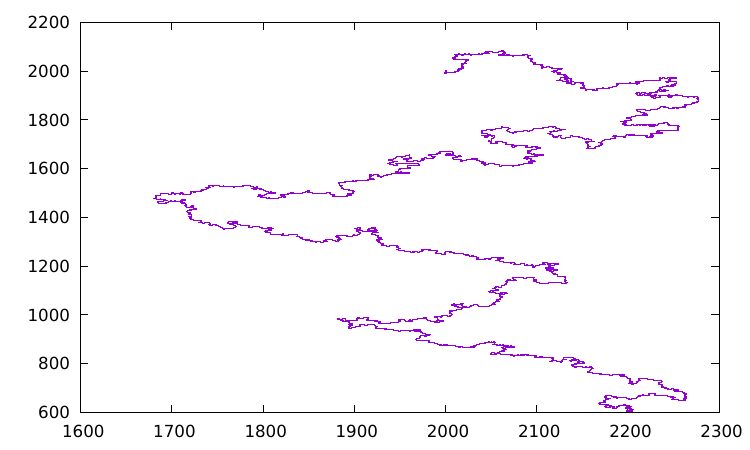}}
  \captionn{\label{ULR} A tree extracted by the leaf evaporation algorithm with $n=10000$ nodes, on a UST of $\Torus{500}$ (left) and on $\Torus{4000}$ (right). }
  \end{figure}

  \begin{rem}The successive removal of leaves induces an order on the set of edges, and this gives a total order if no two edges are removed simultaneously\footnote{Neville code \cite{neville_1953} uses ``evaporation by layers'', and does not provide this ``total order property''}. If one labels the edges by their chronological rank in the evaporation process, the induced labelling of the edges is increasing on any injective path starting at a leaf and ending at the root (or at the terminal node, in the non rooted case): globally, labeling the nodes by their rank provides a decreasing tree; decreasing trees (the classical terminology is increasing trees) have been studied independently, but in general, ``the labels are not added at the end, when the tree is made, but rather, is produced along the construction of edges'' (often, the tree is constructed by successive addition of edges, and the rank of appearance of the new edge, is its label in the tree): see eg. Bergeron et al. \cite{BFS}, Broutin et al. \cite{BDMS}); however, in  Marckert \& Wang \cite{MR3912100}, some processes similar to leave evaporation appear on a uniform Cayley trees in link with the additive coalescent).\end{rem}

\NewSModel{Evaporation of the smallest leaf}{ESL}
{Consider a rooted tree $T$ with $N$ nodes in which the edges are equipped with i.i.d. weights taken under $\mu\sim \uniform[0,1]$.
Successively, remove the leaf adjacent to the edge with the smallest weight among those adjacent to leaves. The set of leaves evolves, as leaves are removed: this forms a sequence of tree $T_N=T,\cdots,T_1=r$ where $T_i$ has $i$ nodes. Return $T_n$ if the target size is $n$ (see some simulation in Fig. \ref{EME}). }
Notice that instead of $\mu$, since only the relative order of the matter of the weight, any atomless measure  $\mu$  gives the same model.

\begin{figure}[h!]
  \centerline{\includegraphics[width=8cm]{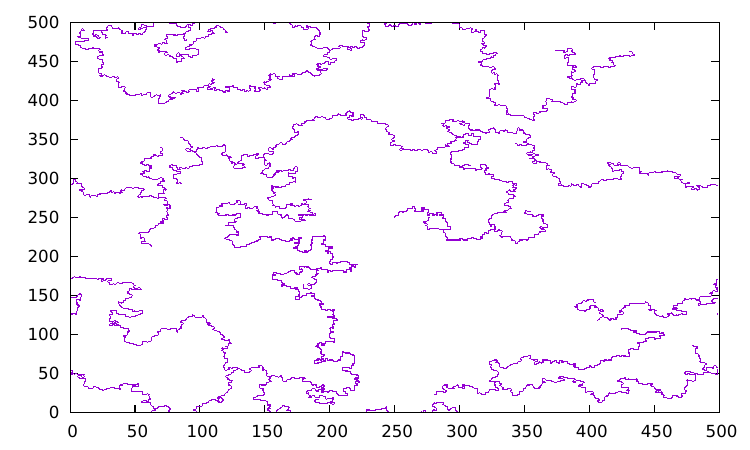}
 \includegraphics[width=8cm]{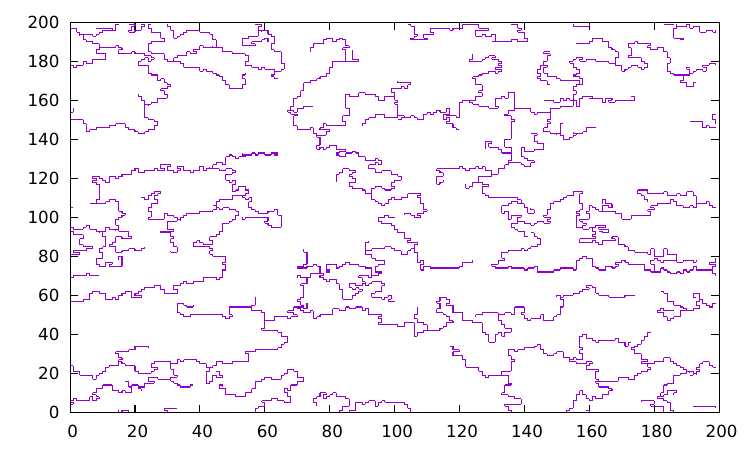}}
  \captionn{\label{EME}  A tree extracted by the removal of the edges with smallest label with $n=5000$ nodes, on a UST of the $\Torus{500}$ (left) and on $\Torus{200}$ (right). As the time passes by, the general tendency is that the edges with minimal weights become leaves, and stay leaves a lot of time.} 
\end{figure}

Denote by $(T,\sigma)$ a labelling of the edges of $T$ by a uniform permutation of $\{1,\cdots,N\}$ where $N=|E(T)|$.
\begin{Ques} Give a nice description of the distribution of the tree remaining when all nodes , except $n$, have evaporated. 
 \end{Ques}
 We have a description of the remaining tree distribution but we feel that something deeper is hidden: the leaf evaporation depends only on the induced random order $\sigma$ of the weighted edges, as defined in \eref{eq:sig}.  Let us put a second label $\ell$ on each edge, corresponding to the date of disappearance of this edge (ranked from 1 to $|T|$); after elimination, the label are increasing on each simple path leading from a leaf of $T$ toward the single remaining node (or the subtree $t$ obtained, if the leaf evaporation is stopped when a certain size is reached): call such a labelling, a valid labelling. Now, the evaporation leads to $t$, if the $N-n$ smallest $\ell$-labels are on the edges of $T\setminus t$.

Consider the map $\pi$ which sends $(T,\sigma)$ onto $(T,\ell)$, that is which gives the elimination order on the edges of $T$. To describe the distribution of the remaining tree $t$, it suffices to be able to compute $ \l|\pi^{-1}(T,\ell)\r|$ for any valid $\ell$. We will see that this is somehow explicit: 
\begin{lem} For $\ell$ valid, the elements of $\pi^{-1}(T,\ell)$ are the $\sigma$ that satisfies,  $\ell(e_1)<\ell(e_2)$ implies
\[\sigma(e_1)<\sigma(e_2), \textrm{ if $e_2$ is a leaf at time $\ell(e_1)-1$}.\]
\end{lem}
\begin{proof} Take two edges $e_1$ and $e_2$, such that $\ell(e_1)<\ell(e_2)$ so that $e_1$ is eliminated before $e_2$. Consider $\sigma(e_1)$ and $\sigma(e_2)$ the corresponding edge values.
Now, consider $T^\star=T\setminus\{ e : \ell(e) <\ell(e_1)\}$ the state of $T$ just before the elimination of $e_1$ (that is, when all the edges with smaller label than $\ell(e_1)$ are removed). In $T^\star$ the edge $e_2$ is still present, so there are two cases:\\
-- If $e_2$ is a leaf, then we must have $\sigma(e_1)< \sigma(e_2)$,\\
-- If $e_2$ is not a leaf, then $\sigma(e_2)$  may be larger or smaller than $\sigma(e_1)$. \end{proof}
For a given $(T,\ell)$, the cardinality of $\#\pi^{-1}(T,\ell)$ can be explicitly computed, but it produces an intricate formula, which needs to be summed over valid $\ell$ to compute $\P(T_n=t)$. 
\medbreak
 
The next model looks similar, but it is different; it defines a tree value process with non-increasing size~: it can reach or not the target size $n$. Up to a change of time, it is independent of $\mu$. 

\NewSModel{Evaporation of the leaves with weight $\leq w$}{EWL}
{Consider a rooted tree $(T,r)$ with $N$ nodes in which the edges are equipped with i.i.d. weights taken under $\mu\sim \uniform[0,1]$. At time $w$, consider the subtree $T(w)$ of $T$ obtained by removing the leaves with weight $\leq w$ (removing these leaves may create new leaves, at which the same procedure applies recursively). When $r$ has degree 1, it is not considered as a leaf.  }
Of course, this model is a percolation model on the weighted graph. One has,
\begin{pro} For any tree $t\in \TTTr{r}{G}$,
  \[\P_\mu(T(w)=t) = \mu(w,+\infty)^{|\partial t|} \mu[0,w)^{|T\setminus t|},\]
  where $|T\setminus t|$ is the number of nodes in $T$ that are not in $t$ (this is also the number of edges).
\end{pro}

Reducing progressively the tree size one by one so that a target size is reached for sure is natural, and \textbf{Models \Modref{ULE}} and \textbf{\Modref{ESL}} are of this type. In the literature, one finds some works \cite{MSZ05} and \cite{MSZ09}, related to distributed algorithms, aiming to ``elect'' a node in a tree, using leaf evaporation (this name does not appear there, however). We present here the general evaporation scheme defined in \cite{MSZ09} which can be used to extract a subtree of a given size by stopping the process when this size is reached (this is not discussed in  \cite{MSZ05,MSZ09}).

In the sequel, we denote by $(T,w)$ an unrooted tree in which nodes are weighted by $w=(w_u,u\in V)$ by some non-negative (possibly random) real numbers, the weight of the leaves being positive; some examples will be given afterwards. The algorithm uses a family of distribution $\mu(q,.)$ on $(0,+\infty)$, for any $q>0$~: this is the lifetime distribution of an active node $u$ with parameter $q$.

\NewModel{Election type  evaporation}{}
{ At time 0 the leaves of $T$ are active and the internal nodes are not.
  A leaf $u$ with weight $w_u$ evaporates after a random time with distribution $\mu(q_u,.)$, independently of the others, where $q_u=w_u$. \par
  \bls Upon evaporation, leaf $u$ transmits its parameter $q_u$ to its single neighbour $v$ in the tree.\par
  \bls A node $v$ with degree $d$, which becomes a leaf after complete evaporation of $d-1$ of the subtrees hanging from it, becomes active (say at time $\tau$). The node $v$ has received the parameters $(q_{v_1},\cdots,q_{v_{d-1}})$ of its neighbours. It then computes its own parameter
  \[q_v = f(w_v, q_{v_1},\cdots,q_{v_{d-1}}),\]
  then generate a random variable $\tau(v)$ with distribution $\mu(q_v,.)$; the node $v$ will evaporate at global time $\tau + \tau(v)$ (hence, $\tau(v)$ is its remaining lifetime, when it becomes active).
}

The function $f$ is a parameter of the algorithm, as well as the initial weights $(w_u,u\in V)$, the family of distributions $\mu(.,.)$, and even some additional parameters can be used to store additional information, as the complete geometry of the evaporated subtrees, as well as their lifetimes, for example.

In \cite{MSZ05}, the model is as follows: the initial weight of all $u\in V$ are $w_u=1$, for all $u\in V$. The map $f$ is given by
\be f(w_v, q_{v_1},\cdots,q_{v_{d-1}})=w_v+q_{v_1}+\cdots+q_{v_{d-1}}=1+q_{v_1}+\cdots+q_{v_{d-1}}\ee
meaning that a node adds to its weight the weights transmitted from its eliminated neighbours (hence, becoming active, its weight is the size of the tree formed by $v$, and the eliminated subtrees which were hanging from it); finally, the remaining lifetime of a node with parameter $q$ is distributed as $m_q\sim {\sf Expo}(q)$,
the exponential distribution with parameter $q$. \\
The main result in \cite{MSZ05} is the following: if one  continues the elimination procedure till a single node ${\bf u}$ remains, then ${\bf u}$ is a uniform node of $V$.\par
 Denote by ${\sf Evaporation}(T,n)$ the random tree obtained from this particular election type evaporation process when only $n$ nodes remain (for $n\leq |V(T)|$).
For a given $t\in \TT{T}{n}$, the graph $T-t$ induced by the removal of edges of $t$ in $T$ is a forest composed of $n$ trees. For any $v\in t$, denote by $\Delta_v$ the tree of $T-t$ attached to $v$. We have 
\begin{theo}
	For any subtree $t\in \TT{G}{n}$,
	\[
	\mathbb{P}\l( {\sf Evaporation}(T,n) = t \r) = \frac{(|\partial t|-1)!(N-n)!}{(|\partial t|+N-n)!}\sum_{v\in \partial t}|\Delta_v|.
	\]
      \end{theo}
      \begin{rem} Notice that the last edge  is then uniform, as well as the last node (this case is stated in \cite{MSZ05}); in the case where the minimum degree of the internal nodes of $T$ is $m$, then the uniformity holds also for all $n\leq m$.
       \end{rem}
       \begin{proof} The evaporation process passes through $t$, if at a given moment all the trees $\Delta_v$ have disappeared, but their root (since their roots belong to $t$). It may be shown, by recurrence that, for a given tree $t'$ (whose root is never considered as a leaf) that the time $A_{t'}$ for the root to be active is distributed as $M_{|t'|-1}$, where for all $k$, $M_k$ is the maximum of $k$ independent exponential random variables with parameter 1. Once the root of such tree becomes active, it has the additional lifetime $a_{t'}\sim{\sf Expo}(|t'|)$, which it is independent of $A_{t'}$. The complete evaporation time $E_{t'}$ of $t'$, which includes the erasure of the root, is distributed as $M_{|t'|}$ since $E_{t'}:=A_{t'}+a_{t'}$. Hence  $(A_{t'},E_{t'})$ is distributed as $(M_{|t'|-1},M_{|t'|-1}+a_{|t'|})$ where the delay $a_{|t'|}$ is independent of $M_{|t'|-1}$
\ben
\label{eq:q1}\P(A_{t'}<y<E_{t'})&=&\P( M_{|t'|-1}\leq y \leq M_{|t'|})  = e ^{-y}(1-e^{y})^{|t'|-1}\\
\label{eq:q2}\P(A_{t'}<y)&=&\mathbb{P}( M_{|t'|-1} \leq y)  =  (1-e^{-y})^{|t'|-1}.
\een
Now, a certificate that the evaporation process passes through $t$ is as follows: a root of one of the $\Delta_v$ disappeared at some time $x$ at which all the other $\Delta_w$ have disappeared, but their root. This gives\\
$\mathbb{P}(  {\sf Evaporation}(T,n) = t )$
\begin{align*}
  &= \sum_{v\in \partial t}\int_0^\infty \l[\prod_{u\in \partial t\setminus\{v\}} \mathbb{P}(M_{|\Delta_{u}|-1}\leq x \leq M_{|\Delta_{u}|})\r]\l[\prod_{u\in t\setminus \partial t}\mathbb{P}( M_{|\Delta_{u}|-1} \leq x )\r]\mathbb{P}(M_{|\Delta_v|}\in dx)\\
  &= \sum_{v\in \partial t}|\Delta_v|\int_0^\infty e^{-|\partial t|x}(1-e^{-x})^{N-n}dx
\end{align*}
which suffices to conclude (the third equality comes from \eref{eq:q1} and \eref{eq:q2}).
\end{proof}

\begin{rem} In \cite{MSZ09}, much more general models of evaporation processes are designed, for which the law of the remaining tree can be computed; they can be turned into evaporation procedure and stopped when a given size is obtained. We don't pursue the description of these results here since it is not clear for the moment that they are useful to target any important distributions.\par
The configurations represented in Fig. \ref{Fig:BG} allow us to reject a lot of algorithms relying on leaf evaporation on the UST to sample $\uniform(\TT{G}{n})$; on this picture, both blue subtrees induce the same subgraph on $G$: to get them after leaf evaporation, the leaf evaporation procedure needs to destroy every green subtree before destroying any blue edge. But, blue edges do not appear simultaneously in both cases: 2 blue edges adjacent to leaves are present in the right-hand side at the beginning, and in the right-hand side, progressively, up to 5 blue leaves may be present at some (random) time during the evaporation. 
\begin{figure}[h!]
  \centerline{\includegraphics[width = 12cm]{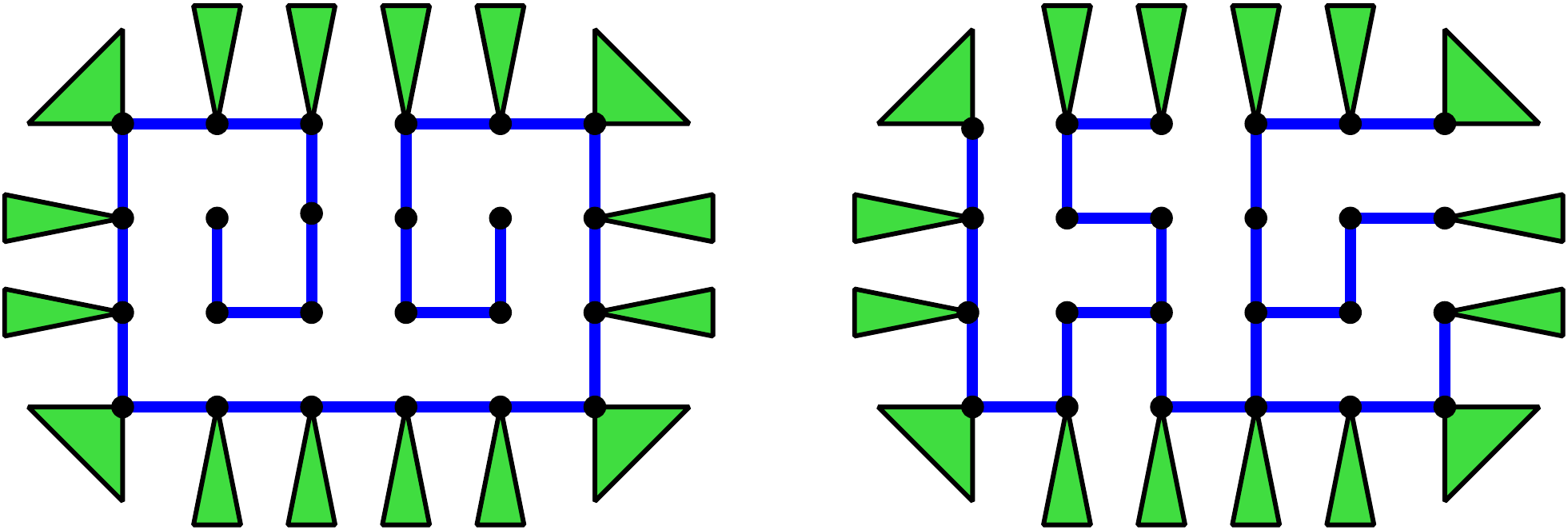}}
  \captionn{\label{Fig:BG} Two trees, in blue, than can be obtained from the elimination of the same ``green subtrees''}
\end{figure}
\end{rem}

\begin{figure}[h!]
  \centerline{\includegraphics[width=8cm]{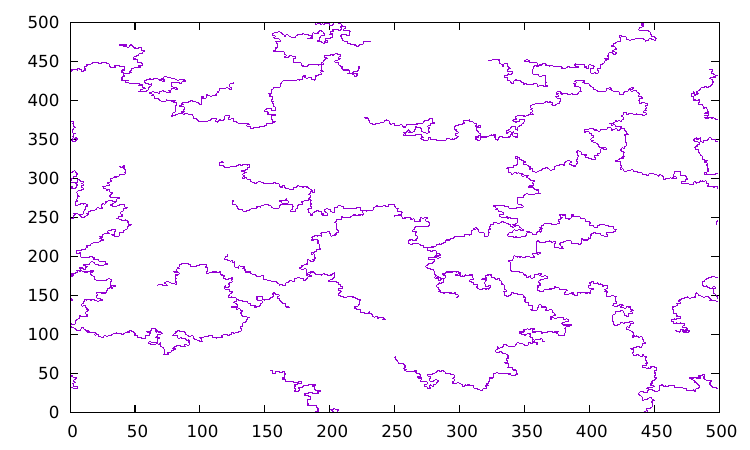}\includegraphics[width=8cm]{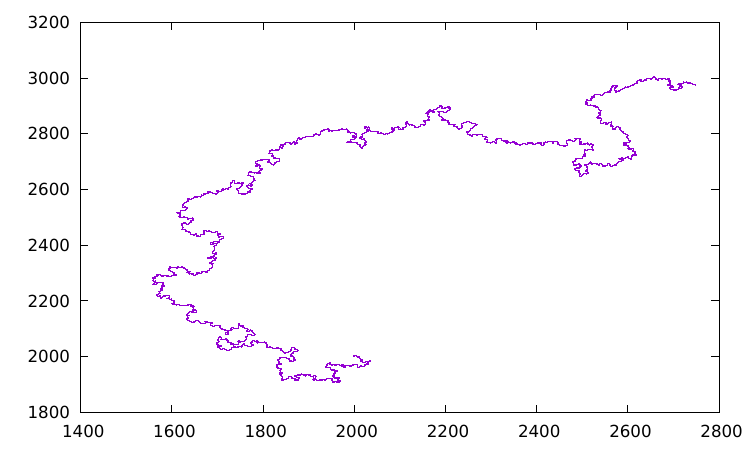}}
  \captionn{\label{Akka} A tree extracted by election type evaporation halted when $n=10000$ nodes remains executed on an UST of $\Torus{500}$ (left) and on a UST of $\Torus{4000}$ (right) }
  \end{figure}

\NewSModel{Removal of uniform edge}{RUE}
{This model is often called ``tree cutting'' in the literature; take a tree $T$ rooted at some node $r$, and remove successively a uniform edge chosen uniformly among the remaining edges of $T$.
  Denote by $T(k)$ the tree obtained from $T$ by the removal of $k$ edges, and $T_r(k)$ the connected component of the origin. The sequence $(T_r(k), 0 \leq k \leq |T|-1)$ coincides with the process $(T_r^{\star}(w), 0\leq w\leq 1)$: the connected components of $r$ by keeping the edges with weight $\geq w$, at its jump time.}
\begin{rem}Uniform edge evaporations of some classical families of trees (of non-embedded trees) have been thoroughly studied following an idea of Meir \& Moon \cite{Memo} in 1970. Many recent developments under the name of ``cut-tree'' have been published which aims at describing the tree structure of the fragmentation history (see e.g. Aldous \& Pitman \cite{MR1675063}, Janson \cite{SJ}, Addario-Berry et al. \cite{ABH}, Bertoin \& Miermont \cite{JBGM}, Broutin \& Wang \cite{Br-Wa} for recent developments).
\end{rem}

\begin{rem} This model is discussed also in Section \ref{edgeremoval} and applied there in the case of a uniform spanning tree (which provides a second level of randomness).  \end{rem}
\begin{pro}\label{pro:sd} For any $t$ subtree of $T$, any $w\in [0,1]$, denote by $B(t)$ the edges of $T\setminus t$ adjacent to $t$ (the boundary of $t$ in $T$).
  \[\P\l(T_r^{\star}(w)= t\r)= w^{|B(t)|}(1-w)^{|E(t)|}\]
  and
  \[\P\l(T_r(k)= t\r)=1_{k\geq |B(t)|} \frac{\binom{ |E(T)|-|E(t)|-|B(t)|)}{k-|B(t)|}}{\binom{|E(T)|}{k}}. \]
  \end{pro}
The same formula are valid for a graph instead.
  \begin{proof} The first formula is easy: the edges in $E(t)$ must be still here, and those of $B(t)$ must have disappeared. For the second formula: since $k$ edges have been suppressed, and by symmetry, they form a uniform subset of $E(T)$; the favourable cases are those for which this subset is $B(t)$ union a subset of size $k-B(t)$ of $E(T)\setminus (B(t)\cup E(t)$; these number of subsets are given by the numerator of the second formula.
    \end{proof}

\small
 
\renewcommand{\baselinestretch}{0.8}
\bibliographystyle{abbrv}

~\newpage
\tableofcontents
\end{document}